\documentclass[11pt,a4paper,twoside]{article}

\usepackage{xcolor} 

\usepackage[utf8]{inputenc}
\usepackage[T1]{fontenc}

\usepackage{times}

\usepackage{amsmath}
\usepackage{amsthm,amsfonts,amssymb} 
\usepackage{bm} 

\usepackage{cases}
\input xy  \xyoption{all}

\usepackage{geometry}

\usepackage{enumerate} 
\usepackage{hyperref} 

\usepackage{paralist}
\usepackage{mathtools}

\newcommand{\commentMR}[1]{}

\newcommand{\verGrB}[1]{#1}
\newcommand{\verNmanifold}[1]{}

\newcommand{\grB}{{graded bundle}}
\newcommand{\grBs}{{graded bundles}}
\newcommand{\GrBs}{{Graded bundles}}

\newcommand{\new}{}

\newcommand{\newMR}{}

\newcommand{\thh}[1]{#1^{\mathrm{th}}} 
\newcommand{\nd}[1]{#1^{\mathrm{nd}}} 
\def\rel{{-\!\!\!-\!\!\rhd}}

\newcommand{\R}{\mathbb{R}}
\newcommand{\Z}{\mathbb{Z}}
\newcommand{\N}{\mathbb{N}}

\newcommand{\wt}[1]{\widetilde{#1}} 
\newcommand{\wh}[1]{\widehat{#1}} 

\DeclareFontFamily{U}{mathx}{}
\DeclareFontShape{U}{mathx}{m}{n}{<-> mathx10}{}
\DeclareSymbolFont{mathx}{U}{mathx}{m}{n}
\DeclareMathAccent{\widehat}{0}{mathx}{"70}
\DeclareMathAccent{\widecheck}{0}{mathx}{"71}
\newcommand{\wch}[1]{\widecheck{#1}} 

\newcommand{\und}{\underline}

\newcommand{\plus}{\boldsymbol{\pmb{+}}} 
\newcommand{\minus}{\boldsymbol{\pmb{-}}}

\newcommand{\LL}{\lambda} 
\newcommand{\LLv}{\lambda^{\mathrm{v}}}

\newcommand{\ideal}[1]{\langle {#1} \rangle} 
\newcommand{\pair}[1]{\langle {#1} \rangle} 
\newcommand{\pullback}[1]{ {#1}^{*}} 
\newcommand{\coreVF}[1]{{#1^{\uparrow}}} 

\newcommand{\Rk}{\mathrm{\Theta}} 
\newcommand{\curvR}{R} 

\newcommand{\Qhat}[1]{\wh{#1}} 
\newcommand{\Qcheck}[1]{\wch{#1}} 

\newcommand{\sym}[1]{{#1}^{\mathrm{sym}}} 
\newcommand{\alt}[1]{{#1}^{\mathrm{alt}}} 

\newcommand{\pa}{\partial}
\newcommand{\ad}{\operatorname{ad}}
\newcommand{\Jac}{\operatorname{Jac}}
\newcommand{\Sec}{\operatorname{\Gamma}}
\newcommand{\rank}{\operatorname{rank}}
\newcommand{\graphW}{\operatorname{graph}}
\newcommand{\id}{\operatorname{id}}

\newcommand{\End}{\operatorname{End}}
\newcommand{\Hom}{\operatorname{Hom}}
\newcommand{\Sym}{\operatorname{Sym}}

\newcommand{\op}{\op} 

\newcommand{\pt}{\operatorname{pt}}
\newcommand{\dd}{\mathrm{d}}
\newcommand{\V}{\mathbf{V}}  

\newcommand{\ra}{\rightarrow}
\newcommand{\mZM}{\Rightarrow}

\newcommand{\veps}{\varepsilon}

\newcommand{\At}[1]{A^{[#1]}} 
\newcommand{\alift}[2]{{#1}^{[#2]}} 
\newcommand{\aliftB}[2]{{#1}^{\langle{#2}\rangle}} 

\newcommand{\T}{\mathrm{T}} 

\newcommand{\jet}[1]{\bm{\mathrm{t}}^{#1}} 
\newcommand{\tclass}[2]{[#2]_{#1}} 

\newcommand{\dbydt}{\frac{\mathrm{d}}{\mathrm{d}t}}
\newcommand{\dbydtk}[1]{\frac{\mathrm{d}^{#1}}{\mathrm{d}t^{#1}}}

\newcommand{\core}[1]{\widehat{#1}}  
\newcommand{\ZMmap}[1]{{\overrightarrow{#1}}} 

\newcommand{\K}{{\new \alpha}}
\newcommand{\KL}{{\new \beta}}  

\newcommand{\pr}{\operatorname{pr}} 

\newcommand{\catVB}{\mathcal{VB}}
\newcommand{\catZM}{\mathcal{VBC}}
\newcommand{\catGB}{\mathcal{GB}}

\newcommand{\derEnd}{\mathfrak{D}} 
\newcommand{\w}{\mathrm{w}} 

\newcommand{\RHS}{RHS}
\newcommand{\LHS}{LHS}
\newcommand{\ip}{{i'}}

\newcommand{\kp}{{k'}}

\newcommand{\G}{\mathcal{G}} 
\newcommand{\Cf}{\mathcal{C}^\infty}

\newcommand{\g}{\mathfrak{g}} 
\newcommand{\RR}{\mathcal{R}} 
\newcommand{\cA}{\mathcal{A}} 

\newcommand{\VF}{\mathfrak{X}} 
\newcommand{\VFvert}{\VF^\V} 
\newcommand{\ev}{\operatorname{ev}}
\newcommand{\sgn}{\operatorname{sgn}}
\newcommand{\comp}{\bm{\wedge}}

\newcommand{\VZM}{\V} 

\def\pLinr{{\operatorname{lin}}} 
\newcommand{\lin}{\ell} 

\newcommand{\reduction}[1]{#1\!\!\!\!\downarrow}

\newcommand{\VBC}{{\textrm{VB} comorphism }}

\newcommand{\iM}[1]{i^{#1}_M} 
\newcommand{\iMM}[2]{i^{#1}_{#2}} 

\newcommand{\jM}[1]{j^{#1}_M} 

\newcommand{\jE}[1]{\jmath^{#1}_{E}} 
\newcommand{\jEE}[2]{\jmath^{#1}_{#2}} 

\newcommand{\iA}[1]{\imath^{#1}_{A}} 
\newcommand{\iAA}[2]{\imath^{#1}_{#2}} 

\newcommand{\jAc}[1]{\jmath^{[#1]}_{A}} 
\newcommand{\jAAc}[2]{\jmath^{[#1]}_{#2}} 

\newcommand{\diag}{\mathrm{diag}} 

\newcommand{\curv}{\mathrm{curv}} 

\newcommand{\bas}{\mathrm{bas}} 

\newdir{|>}{!/4,5pt/@{|}*:(1,-.2)@^{>}*:(1,+.2)@_{>}}
\def\relto{{\rightarrow\!\!\vartriangleright}} 

\def\<#1>{\left\langle #1\right\rangle}
\def\(#1){\left( #1\right)}

\newcommand{\quotient}[2]{{\raisebox{.2em}{$#1$}\left/\raisebox{-.2em}{$#2$}\right.}}

\providecommand*{\xrightrightarrows}[2][]{%
  \ext@arrow 0359\rightrightarrowsfill@{#1}{#2}%
}

\numberwithin{equation}{section} 

\theoremstyle{plain}
\newtheorem{thm}{Theorem}[section]
\newtheorem{prop}[thm]{Proposition}
\newtheorem{lem}[thm]{Lemma}
\newtheorem{conj}[thm]{Conjecture}

\theoremstyle{definition}
\newtheorem{df}[thm]{Definition}
\newtheorem{ex}[thm]{Example}

\theoremstyle{remark}
\newtheorem{rem}[thm]{Remark}
\newtheorem{cor}[thm]{Corollary}

\def\namedlabel#1#2{\begingroup
    #2%
    \def\@currentlabel{#2}%
    \phantomsection\label{#1}\endgroup
}

\makeatletter
\newcommand{\leqnomode}{\tagsleft@true}
\newcommand{\reqnomode}{\tagsleft@false}
\makeatother

\title{Exploring the Structure of Higher Algebroids\footnote{This research was supported by the Polish National Science Center under the grant DEC-2012/06/A/ST1/00256.}} 
\author{Miko\l aj Rotkiewicz\footnote{\emph{Institute of Mathematics,  University of Warsaw} (email: \texttt{mrotkiew@mimuw.edu.pl})}}
\date{}

\begin{document}
\maketitle

\begin{abstract}
The notion of a \emph{higher-order algebroid}, as introduced by Jóźwikowski and Rotkiewicz in their work \emph{Higher-order analogs of Lie algebroids via vector bundle comorphisms} (SIGMA, 2018),  generalizes the concepts of a higher-order tangent bundle $\tau^k_M: \mathrm{T}^k M \to M$ and a (Lie) algebroid. This idea is based on a (vector bundle) comorphism approach to (Lie) algebroids and
the reduction procedure of homotopies from the level of Lie groupoids to that of Lie algebroids.
In brief, an alternative description of a Lie algebroid $(A, [\cdot, \cdot], \sharp)$  is  a vector bundle comorphism $\kappa$, defined as the dual of the Poisson map $\varepsilon: \mathrm{T}^\ast A \ra \mathrm{T} A^\ast$ associated with the  Lie algebroid $A$. The framework of comorphisms has proven to be a suitable language for describing higher-order analogues of Lie algebroids from the perspective of the role played by (Lie) algebroids in geometric mechanics.
In this work, we uncover the classical algebraic structures underlying the somewhat mysterious  description of higher-order algebroids through  comorphisms. For the case $k=2$, we establish  a one-to-one correspondence between higher-order Lie algebroids and pairs consisting of a two-term representation (up to homotopy) of a Lie algebroid and a morphism to the adjoint representation of this algebroid.
\end{abstract}
\paragraph*{MSC 2020:}{58A20, 58A50, 17B66, 17B70 }

\paragraph*{Keywords:}{ higher algebroids, representations up to homotopy, graded manifolds, graded bundles, VB-algebroids, Lie algebroids}


\small
\small

\tableofcontents



\section{Introduction}
In  \cite{MJ_MR_HA_comorph_2018} the notion of a higher algebroid was introduced,
  based on extensive studies of examples we would like to refer to as higher algebroids (HAs, for short) \cite{MJ_MR_HA_var_calc_2013, MJ_MR_models_high_alg_2015}.
Our intuitive thinking was that a higher algebroid should represent a geometric and algebraic structure that generalizes higher-order tangent bundles in a similar manner as algebroids generalize tangent bundles.
In the first order, the algebroid structure is defined on a vector bundle (VB, for short), with the most obvious example being the tangent bundle $\tau_M: \T M \rightarrow M$. In higher orders, $\tau_M$ is replaced by the higher-order tangent bundle $\tau^k_M: \T^k M \rightarrow M$, which for $k>1$ is no longer a vector bundle but a
{\new \emph{graded bundle}, in the terminology introduced in \cite{JG_MR_gr_bund_hgm_str_2011}. }
\verNmanifold{It is referred to as an $\N$-graded manifold in \cite{MJ_MR_HA_comorph_2018}, and we shall follow this terminology as this is a special case of the concept of a \emph{graded manifold} introduced by Voronov \cite{Voronov_2002}.
}
\verGrB{It is referred to as an $\N$-graded manifold in \cite{MJ_MR_HA_comorph_2018}.
}
In a \grB\ there exists a distinguished class of \emph{graded fiber coordinates}, taking over linear coordinates, with transition functions represented as homogeneous polynomials. In a particular case of polynomials of degree one (linear maps), one gets vector bundles as a special case.
From various perspectives discussed in \cite{MJ_MR_HA_comorph_2018}, it became apparent that the structure of a higher algebroid should be defined on a \grB.

The most common way to motivate the concept of a Lie algebroid comes from the reduction of the tangent bundle $\T \G$ of a Lie groupoid $\G$. 
{As a geometric object, the Lie algebroid of $\G$ is} the set $\cA(\G) := \T_M\G^\alpha$, { consisting of  tangent vectors in the direction of the source fibration $\G^\alpha$ of $\G$ and based at $M$ -- the base of $\G$.}
{\new The structure of the tangent bundle  $\T \G$, induces a certain structure on $\T_M \G^\alpha$, leading to the notion of a Lie algebroid. }
Typically, the structure of a Lie algebroid is expressed by means of a bracket operation $[\cdot, \cdot]$  { on the space of sections of a vector bundle $\sigma: A\ra M$} and a VB morphism $\sharp: A\ra \T M$ called \emph{the anchor map}. However, it is obvious that this approach has no direct generalization to higher-order case, because there is no bracket operation on the space of sections of $\tau^k_M:\T^k M\ra M$, since  $\tau^k_M$ is not a vector bundle for  $k>1$ and, in particular, its sections cannot be added.

In light of the above, it is  natural to consider the reduction $\T^k_M\G^\alpha$ of the $\thh{k}$-order tangent bundle $\T^k \G$ as a prototype of a higher-order algebroid of order $k$. The reduction map $\RR^k: \T^k \G^\alpha \rightarrow \cA^k(\G) :=\T^k_M\G^\alpha $ takes a $\thh{k}$-velocity $\tclass{k}{g}$ represented by a curve $g: \R \rightarrow \G$, lying in a single fiber of the foliation $\G^\alpha$, to the $\thh{k}$-velocity based at a point in $M$. A natural problem arises: how to characterize the structure on $\cA^k(\G)$ inherited from the groupoid multiplication. {\new In \cite{MJ_MR_models_high_alg_2015} we proposed an answer to this question by reducing the natural map $\kappa^k_{\G}: \T^k \T \G\rightarrow \T \T^k \G$. }

In the first order, one reduces the canonical involution  $\kappa_{\G}$ which  results in a relation $\kappa\subset \T \cA(\G) \times \T \cA(G)$ and leads to an alternative definition of the structure of a Lie algebroid as a pair $(A, \kappa)$ consisting of a vector bundle $\sigma: A\rightarrow M$ and a relation $\kappa \subset \T A\times \T A$ of a special kind . This viewpoint on Lie algebroids was first introduced in \cite{JG_PU_Algebroids_1999}.
It turns out that $\kappa$ is the dual of the Poisson map $\veps: \T^\ast A \rightarrow \T A^\ast$ associated with the linear Poisson tensor on $A^\ast$.\footnote{Linear Poisson structures on the dual bundle $A^\ast$ are in  a one-to-one correspondence with Lie algebroid structures on~$A$.} As $\veps$ is a VB morphism, the dual $\kappa = \veps^\ast$ is a VB comorphism, see Definition~\ref{df:VBC}.
The comorphism  approach to (Lie) algebroids is also very natural from the perspective of variational calculus. The relation $\kappa$ was recognized as a 'tool' in constructing admissible variations in geometric mechanics \cite{GG_var_calc_gen_alg_2008}.

Based on the properties of the reduction of $\kappa^k_\G$ and its potential applications in variational calculus, we introduced higher algebroids in \cite{MJ_MR_HA_comorph_2018} as pairs $(E^k, \kappa^k)$ consisting of a \grB\ $\sigma^k: E^k\rightarrow M$ of order $k$, equipped with a vector bundle comorphism $\kappa^k$. This comorphism relates the vector bundles $\T^k E^1\rightarrow \T^k M$ and $\T E^k\rightarrow E^k$ and satisfies certain natural axioms.\footnote{In the comorphism approach, it is natural to consider generalizations of the notion of a Lie algebroid obtained by relaxing its axioms. In the literature on geometric mechanics, these generalizations are known as 'almost Lie' algebroids (where the Jacobi identity is not assumed), 'skew' algebroids (where neither the Jacobi identity nor the anchor-bracket compatibility is required), and 'general' algebroids (where, in addition, the skew-symmetry of the bracket is not required).} We recall from \cite{MJ_MR_HA_comorph_2018}   a detailed  formulation   of these axioms in Definition~\ref{df:higher_algebroid}.

The definition of HAs given in \cite{MJ_MR_HA_comorph_2018}, which we consider to be very natural from many perspectives, also appears to be quite mysterious. The goal of the present work was to unveil the vector bundle morphisms, brackets, and other operations hidden within the comorphism description of HAs. A complete solution is achieved in the case of $k=2$.

Our solution situates Lie HAs within the realm of \emph{representations up to homotopy} (representations u.t.h., in short) of Lie algebroids, the concept introduced in \cite{AbCr2012}. The idea is to represent Lie algebroids using cochain complexes of vector bundles. Such a complex is given 'an action' of a Lie algebroid represented  by an $A$-connection which is flat only 'up to homotopy' governed by higher order homotopy operators. When the complex consists of only one term, this notion reduces to a genuine representation of a Lie algebroid on a vector bundle. 
An important example  for us is the notion of the \emph{adjoint representation} whose proper generalization from the field of Lie algebras to that of Lie algebroids is found within the framework of representations up to homotopy.
As explained in \cite{AbCr2012}, the adjoint representation of a Lie algebroid $A$ is manifested by 'an action up to homotopy' on the two-term complex $\sharp: A \rightarrow \T M$, where $\sharp$ is the anchor map.  On the other hand, it was found in \cite{Gr-Meh2010} that 2-term representations u.t.h. have an elegant   description by means of VB-algebroids ---   Lie algebroid objects in the category of vector bundles.  It this correspondence the adjoint representation of a Lie algebroid $A$ is nothing more but the VB-algebroid structure on $\T A$ -- the tangent prolongation of the Lie algebroid $A$. Our solution also recognizes this point of view.

\paragraph{Our results.}

The main result is presented in Theorem~\ref{t:Rep_to_HA} and Corollary~\ref{cor:HA_VB-alg} where we establish a one-to-one 
correspondence between higher algebroids of order two and morphisms between representations u.t.h. of Lie algebroids of a specific nature as presented in the diagram: 

$$
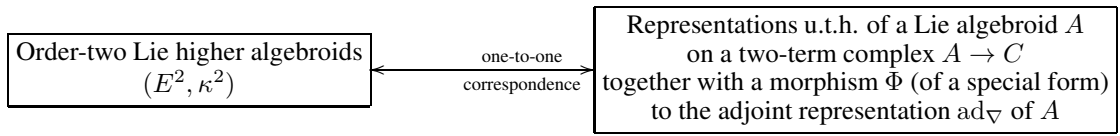
\xymatrix{*+[F]{\txt{Order-two Lie higher algebroids \\ $(E^2, \kappa^2)$}}\ar@{<->}[rrr]^{\text{one-to-one}}_ {\text{correspondence}}&&&*+[F]{\txt{Representations u.t.h. of a Lie algebroid $A$ \\ on a two-term complex $A\ra C$ \\ together with a morphism $\Phi$ (of a special form) \\ to the adjoint representation $\ad_\nabla$ of $A$}}
}
$$
where $\nabla$ is a fixed linear connection on a vector bundle $A\ra M$.
On the left is a Lie HA structure defined on a \grB\ $E^2 \rightarrow M$ characterised by a special type of relation denoted as $\kappa^2$, which is a subset of $\T^2 E^1 \times_M \T E^2$. In this correspondence, $A=E^1$ is the reduction of the \grB\ $E^2$ to degree $1$, with the Lie algebroid structure inherited from $\kappa^2$.  Furthermore, the vector bundle $C\to M$ is introduced as the \emph{ core} of $E^2$, as explained in Section~\ref{sec:pre}.
On the right-hand side, we have a representation u.t.h. of the Lie algebroid  $A$  defined on a two-term cochain complex $A \rightarrow C$. Additionally, there is a morphism denoted as $\Phi$ that connects this representation to the adjoint representation $\ad_\nabla$ of $A$ in the sense of \cite[Definition~3.3]{AbCr2012}, {and further, $\Phi$ is of special type:} the 1-form component of  $\Phi$  vanishes and so $\Phi$ is  a map of cochain complexes and, moreover, $\Phi$ is the identity on $A$ in degree 0.

Following the ideas of \cite{Gr-Meh2010, DJO15}, we found that such a morphism $\Phi$ corresponds to a VB-algebroid morphism to the adjoint representation of $A$ represented as the VB-algebroid $\T A$ (see Corollary~\ref{cor:HA_VB-alg}). This construction makes the choice of a linear connection $\nabla$ unnecessary. In summary, order-two Lie HAs are characterised by VB-algebroid morphisms $\Psi: D\to \T A$ from a VB-algebroid $D$ to the tangent prolongation of a Lie algebroid $A$,  such that $\Psi$ is the identity on the underlying algebroid $A$, and  on the core bundle, which is also identified with the vector bundle $A$.

These results are obtained in a few steps which we discuss below.

 \subparagraph{Map $\Rk^k$.} In any order $k$, we discover a canonical morphism of \grBs\, denoted by $\Rk^k: \At{k} \ra E^k$  (see Definition~\ref{df:Rk}), which is associated with any almost Lie $\thh{k}$-order algebroid\footnote{In the case when $k=2$, the existence of a morphism $\Rk^2$ is already guaranteed by the weaker assumption that $A$ is a skew algebroid.}
 $(E^k, \kappa^k)$. Here,  $A$ is the almost Lie algebroid (AL algebroid, for short) $(E^1, \kappa^1)$ obtained from $(E^k, \kappa^k)$ by means of the reduction to order one, and $(\At{k}, \kappa^{[k]})$ is the $\thh{k}$-order prolongation of $A$ -- a \grB\ with the HA structure naturally induced from the AL algebroid structure on $A$ (see \eqref{df:A[k]} and \eqref{df:kappa[k]}).      The existence of this map is of crucial importance as it allows to relate properties of an abstract higher algebroid with much better recognized   HA $(\At{k}, \kappa^{[k]})$  studied in details in \cite{MJ_MR_models_high_alg_2015}.  We recall that if $A=\cA(\G)$ is the Lie algebroid of a Lie groupoid $\G$ then  $\At{k} = \cA^k(\G)=\T^k_M\G^\alpha$ is the $\thh{k}$-order HA of  $\G$. We conjecture (in Conjecture~\ref{conj:Rk}) that if  $(E^k, \kappa^k)$ is a  Lie HA then the structure map $\Rk^k: \At{k} \ra E^k$ is a morphism of HAs. We were able to prove  this in the case  $k=2$.

 \subparagraph{The structure of the \grB\ of a HA $(E^2, \kappa^2$).}  In general, a \grB\ of order two is obtained from its components:  the vector bundle $E^1$ (the order-one reduction of $E^2$), and its   core vector bundle, denoted by $\core{E^2}$, by gluing transition functions that are homogeneous polynomials of degree 2. In what follows, \textbf{the vector bundles $E^1$ and $\core{E^2}$, are denoted by $A$ and $C$, respectively.}

  With the help of the map $\Rk^2$ we can recover  the  \grB\       $E^2\to M$ as the  quotient of the \grB\  $\At{2}\times_M C_{[2]}$, see Lemma~\ref{l:structureE2}.  Here,  $\At{2}$ is the second-order prolongation of $A$ and $C_{[2]}$ denotes the graded bundle of order 2 obtained from the vector bundle $C$ by assigning weight two to the linear functions on $C$.

 \subparagraph{Structure maps of HAs.}
 By focusing solely on  the \grB\ structure of $\kappa^2$  we encounter equations \eqref{e:local_kappa2}. Our objective is to attribute a geometrical interpretation of the local structure functions { $Q^a_i$, $Q^a_{ij}$, $Q^\mu_{ij, k}$, etc. present in \eqref{e:local_kappa2}.}  It turned out that the functions $Q^\mu_{ij, k}$ do  not  correspond to any geometric object, highlighting that such an interpretation is not always straightforward.
 However,  when we combine  $Q^\mu_{ij, k}$ with $Q^k_{ij}$ there emerges a three-argument operation, denoted by $\delta$,  on the space $\Gamma(A)$ of sections of the vector bundle $A$ with values in $\Sec(C)$, where $C$ is the  core of $E^2$,  see \eqref{df:Qmuijk} and \eqref{df:delta}.
The meaning of the other structure functions\footnote{We will use the symbol $Q^{\cdots}_{\cdots}$ to  refer to the structure functions $Q^a_i$, $Q^a_{ij}$, etc. present in \eqref{e:local_kappa2}.} $Q^{\cdots}_{\cdots}$ proved to be more straightforward.  These include: \begin{inparaenum}[(i)] \item a  skew-symmetric bracket $[\cdot, \cdot]$ on $\Sec(A)$ and a VB morphism $\sharp: A\to \T M$ defining a skew algebroid structure on $A$, \item a morphism of \grBs\ (the second-order anchor map) $\sharp^2: E^2 \ra \T^2 M$, which is the base of the comorphism $\kappa^2$, \item a vector bundle morphisms $\pa: A\ra C$,
\item a map $\Box: \Sec(A) \times \Sec(C) \ra \Sec(C)$, \item a skew symmetric map $\beta:\Sec(A)\times \Sec(A) \ra \Sec(C)$.
\end{inparaenum}

There is also another interesting structure map $\psi: \Sec(A)\times \Sec(A) \to \VF(M)$  (see \eqref{df:psi}), which becomes relevant when studying tensorial properties of the aforementioned structure maps.  Moreover, the symmetric part of $\psi$, denoted by  $\sym{\psi}$, together with the VB morphisms $\sharp^1 = \sharp$ and $\core{\sharp^2}$ (the core of $\sharp^2$), allows us to recover the second-order anchor map $\sharp^2$, see Lemma~\ref{l:grBA[2]_to_C} and Theorem~\ref{th:skew_HA}. This resolves the problem of presenting axioms of a skew order-two HA entirely in terms of VB morphisms and VB differential operators.

 Definitions of all these structure maps are given in Subsection~\ref{sSec:HA_order_two}. Most of them are obtained  through algebroid lifts  $\Sec(A) \ra \VF(E^2)$, $s\mapsto \aliftB{s}{\K}$, associated with the HA $(E^2, \kappa^2)$  and  Lie brackets of vector fields on  $E^2$. The definition of algebroid lifts, as seen in \eqref{df:algebroid_lift},  relies on the  characteristic property of  a VB comorphism: unlike a typical VB morphism, it induces a map between the spaces of sections.

 We also introduce a  structure map $\omega: \Sec(A)\times \Sec(A)\times \Sec(A) \ra \Sec(C)$,
being some modifications of the map $\delta$, see \eqref{df:omega}.  While it carries the same information as $\delta$, it has better algebraic properties and helps to formulate our results more concisely.
Additionally, we define maps $\xi$, $\veps$, $\veps_0$, $\veps_1$ in Definition~\ref{df:epss}, which appear in the Leibniz-type formulas for above-mentioned structure maps (Theorem~\ref{th:skew_HA}).  The vanishing of these maps is also included as  an axiom of AL HAs or Lie HAs, see Theorems~\ref{th:AL_HA_str_maps} and~\ref{th:Lie_axiom_maps}. Moreover, for greater precision in formulating certain results, we found it useful to decompose some of these maps, such as $\delta$, $\omega$, and $\psi$, into their symmetric and anti-symmetric parts.

 Most of the structure maps mentioned above are multi-differential operators on certain vector bundles. We provide a detailed description of the tensorial properties of
these structure maps and prove that  a system of such maps allows to reconstruct the skew HA $(E^2, \kappa^2)$, see Theorem~\ref{th:skew_HA}. This approach can also be extended  to $\thh{k}$-order HAs  for
$k>2$, as discussed in Remark~\ref{r:str_maps_Ek}.

In the next step, we characterize the axioms of an almost Lie HA (Theorem~\ref{th:AL_HA_str_maps}) and Lie HA (Theorem~\ref{th:Lie_axiom_maps}) in terms of the above mentioned structure maps. In other words, we formulate necessary and sufficient conditions that the structure maps $\beta$, $\square$, $\omega$ etc. should satisfy for the related HAs to be, respectively, almost Lie and Lie. Throughout our analysis, we heavily rely on Theorem~\ref{th:HA_axioms_and_lifts}, which provides characterizations of AL and Lie axioms for higher algebroids through  algebroid lifts.

\subparagraph{HAs over a point.}
In case when the base $M$ is a point we find a complete description of order-two skew and Lie higher algebroids, see Theorem~\ref{thm:HA_point}:  An order-two skew  HA over a point has to split, meaning $E^2 = \g_{[1]} \times C_{[2]}$ where $\g=E^1$ and $C=\core{E^2}$. Furthermore, in the Lie case,  there is a one-to-one correspondence between such Lie HAs and Lie module morphisms $\pa: {\g} \to {C}$ from the adjoint module of the Lie algebra $\g$ to the $\g$-module $C$.
\subparagraph{Main result.}
The reformulated axioms of Lie HAs and the description of order-two HAs over a point by means of representations of Lie algebras may suggest  a relation between HAs and representations of Lie algebroids. Note however that there is no concept of the adjoint representation within the framework of representations on  VBs.
It is the setting of representations u.t.h. of Lie algebroids in which the correct generalization of the concept of the adjoint representation of a Lie algebra is possible.
The construction of a Lie algebroid representation out of a Lie HA $(E^2, \kappa^2)$
imitates the construction of the adjoint representation given in \cite{AbCr2012}. It is obtained by means of the structure maps of a Lie HA mentioned earlier. There exist also an obvious map $\Phi$ between the complexes $A\to C$ and $A\to \T M$ which, thanks to the properties of the structure maps of a Lie HA, turns  out to be a morphism to the adjoint representation of $A$.
Conversely, if a representation u.t.h. of $A$ on a two-term complex of the form  $A\to C$ is given, and  a morphism of complexes $\Phi: (A\to C) \ra (A\to \T M)$ is given that also serves as a morphism of representations, then we can extract the structure maps from it and construct a skew HA structure on the  \grB\ described in Lemma~\ref{l:structureE2}. It can be then verified that these maps satisfy the axioms
of AL and Lie HA given in Theorems~\ref{th:AL_HA_str_maps} and ~\ref{th:Lie_axiom_maps}.

\subparagraph{Examples.} Given a Lie algebroid $A$, there are two natural  morphisms to the adjoint representation of a Lie algebroid $A$. One is the identity morphism $\Phi$ on the adjoint representation.  The other one is obtained from  \emph{the double of a vector bundle}, described in \cite{AbCr2012},  which is a representation  of the Lie algebroid  $A$ on a 2-term complex  of the form $E\xrightarrow{\id} E$. 
We illustrate HAs corresponding to these two cases in Examples~\ref{ex:HA_from_id}  and \ref{ex:repr_from_A2}, respectively.

 \paragraph{Organization of the paper.}
Section~\ref{sec:pre} begins by collecting notations and fundamental constructions
concerning \grBs, double vector bundles and VB-algebroids.  We also introduce a functor, denoted by $\LL$, which is a generalization of the linearisation functor discovered in \cite{Bruce_Grabowska_Grabowski_2016} and which is used in the definition of the morphism $\Rk^k$ (Definition~\ref{df:Rk}).  We recall also basic definitions from~\cite{MJ_MR_HA_comorph_2018, MJ_MR_models_high_alg_2015} (VB comorphism,  higher-order algebroid, prolongations of AL algebroids) and give a definition of algebroid lifts in a slightly different way than in~\cite[Definition~4.8]{MJ_MR_HA_comorph_2018}, more convenient for computations which we perform in Section~\ref{sec:app}. Theorem~\ref{th:HA_axioms_and_lifts} extends Proposition~4.9 from~\cite{MJ_MR_HA_comorph_2018} to the AL case.  In Lemma~\ref{l:comp_alg_lifts}  we express the compatibility of algebroid lifts obtained by means of HAs $\kappa^k$ and the reduction of $\kappa^k$ to a lower weight.  We also list a few  canonical inclusions used in the paper  and describe their relationships.

Section~\ref{sec:str} is devoted to a detailed analysis of mathematical structures   standing behind a comorphism $\kappa^k$ that defines a HA structure. W begin with the definition and properties of the map $\Rk^k$ , which connects an  arbitrary HA
$(E^k, \kappa^k)$ with the $\thh{k}$-order prolongation of its  first-order reduction $(E^1, \kappa^1)$. We provide coordinate formulas for $\Rk^2$ and $\Rk^3$, see Example~\ref{ex:Theta2_and_3}.

From this point on, we focus solely on the case
$k=2$. In Lemma~\ref{l:structureE2} we find an explicit construction of a \grB\ $E^2\to M$ which hosts an order-two HA $(E^2, \kappa^2)$. Subsequently,
we introduce  several canonical maps associated with this HA  referring to  them as \emph{"the structure maps of $(E^2, \kappa^k)$"}. Most of these maps are differential operators defined   on  (the product of) the spaces of sections   $\Sec(A)$ or $\Sec(C)$ with values in $\Sec(C)$.  
The term \emph{"structure functions"} is reserved to  functions $Q^a_i$, $Q^a_{ij}$, etc. which are given in Example~\ref{ex:kappa2_coord} as a local representation  of a general order-two HA.  These functions depend on the chosen coordinate system on the \grB\ $E^2$.  Although we   work  in the case $k=2$  we present analogs of   the structure maps in any order, see Remark~\ref{r:str_maps_Ek}.
In Theorem~\ref{th:skew_HA}  we provide an equivalent description of skew, order-two HAs
in terms of the aforementioned structure maps.  
In Theorem~\ref{thm:HA_point} we  discuss the special case when the base $M$ of $E^2$ is a single point (an order-two analog of a (Lie) algebra) and give a characterization of such structures. It turns out that the Lie condition  is very rigorous and all such Lie HAs correspond to morphisms $\pa: C \to A$ of Lie algebra modules, where $A$ represents the adjoint module of the  Lie algebra $A$.

We subsequently examine  the conditions in Definition~\ref{df:higher_algebroid} characterizing AL and Lie HA and translate them to the level of the structure maps, see Theorem~\ref{th:AL_HA_str_maps} and Theorem~\ref{th:Lie_axiom_maps}.
 Moving forward, in Lemma~\ref{l:HA_to_Rep} we recognize that data describing order-two Lie HAs gives rise to a representation u.t.h. of the Lie algebroid $A$ on the structure map $\pa: A \to C$ considered as a two-term complex of vector bundles and also induces a morphism $\Phi$ to the adjoint representation of $A$. Remarkably, this data is also sufficient for recovering a HA structure $(E^2, \kappa^2)$, as demonstrated in  Theorem~\ref{t:Rep_to_HA}. Furthermore, we formulate VB-algebroid version of this correspondence in Corollary~\ref{cor:HA_VB-alg} and illustrate the obtained relationship in Examples~\ref{ex:HA_from_id} and \ref{ex:repr_from_A2}. We also briefly recall the correspondence between representations u.t.h. and VB-algebroids —- providing the necessary facts on this subject to demonstrate our results.


In Appendix~\ref{sec:app} we give proofs for various results, including part (a) of Theorem~\ref{th:skew_HA}, Theorem~\ref{thm:HA_point}, Conjecture~\ref{conj:Rk} in the case $k=2$ and complete the proof of Theorem~\ref{th:AL_HA_str_maps} , where more detailed calculations, including those in coordinates, are carried out. Some of these calculations are supported by additional lemmas. One can also find there a brief recollection on representations up to homotopy, guided by \cite{AbCr2012}. For more in-depth information, interested readers should refer to the existing literature \cite{AbCr2012, Gr-Meh2010, BGV18, GJMM18}.

\paragraph{Historical remarks.} The studies on HAs, as understood in this paper, were initiated by M. Jóźwikowski and the author of present manuscript in \cite{MJ_MR_HA_var_calc_2013}, and  continued in \cite{MJ_MR_models_high_alg_2015, MJ_MR_HA_comorph_2018}. Prior to this,  higher-order analogues of Lie algebroids was the subject of \cite{Voronov_Q_mfds_high_lie_alg_2010} by  T. Voronov who proposed that such analogues should be $Q$-manifolds of spacial kind generalizing Vaintrob approach to  Lie algebroids \cite{Vaintrob_1997}. The most recent studies are due to A. Bruce, K. Grabowska and J. Grabowski \cite{Bruce_Grabowska_Grabowski_2016} whose idea was to imitates the canonical inclusion $\T^k M\subset \T \T^{k-1} M$ on the abstract level of \grBs\, having $\T^k M$ as a prominent example of $\thh{k}$-order analogue of a Lie algebroid. As we pointed in \cite{MJ_MR_HA_comorph_2018}, all these approaches lead to different mathematical objects. This distinctiveness is further  evident from the classification of order-two (Lie) HAs given in this work.

\subsection*{Acknowledgments.} The author is grateful to Michał Jóźwikowski for his insightful comments on the organization and editing of this work, and for many fruitful discussions regarding the research.

\section{Preliminaries}\label{sec:pre}

\subsection{\GrBs\ }

We shall review basic constructions associated with \emph{\grBs\ } that will be used in the present work.  For further details, we  refer to \cite{JG_MR_gr_bund_hgm_str_2011} and additional  works \cite{BGR, MJ_MR_HA_comorph_2018, Voronov_2002}.

A fundamental example of a \grB\ is the $\thh{k}$-order tangent bundle $\tau^k_M: \T^k M\ra M$. The elements of $\T^k M$ are $\thh{k}$-order tangency classes $\tclass{k}{\gamma}$ of curves $\gamma$ in $M$.\footnote{\new $\tclass{k}{\gamma}$ is also called the \emph{$k$-velocity} represented by the curve $\gamma$} Then $\T^1 M = \T M$ is the tangent bundle of $M$ but for $k>1$, $\tau^k_M$ is not a vector bundle; however, the fibers are still equipped  with a special structure, namely,  a natural action of the monoid of real numbers $(\R, \cdot)$,
 $$
    h: \R \times \T^k M \ra \T^k M, \quad (s, \tclass{k}{\gamma}) \mapsto \tclass{k}{\gamma_s}
 $$
 where $\gamma_s(t) = \gamma(st)$.
 Thus, in the terminology of \cite{JG_MR_gr_bund_hgm_str_2011}, $\T^k M$ is a \emph{homogeneity structure}, i.e.,  a manifold  equipped with a smooth action of  $(\R, \cdot)$.
On the other hand, local coordinates $(x^a)$ on $M$ induce \emph{adapted coordinates} \footnote{{$f^{(\K)}\in \Cf(T^k M)$, given by  $f^{(\K)} (\tclass{k}{\gamma}) = \dbydtk{k}_{t=0} f(\gamma(t))$, denotes the $(\K)$-lift of $f\in \Cf(M)$, see \cite{Morimoto_Lifts}. Hence,
$x^{a, (\K)} = (x^a)^{(\K)}$ denotes the $(\K)$-lift of the coordinate  function $x^a$.}
} {\new $(x^{a, (\K)})_{0\leq \K\leq k}$ }  on $\T^k M$ which are naturally graded by numbers $0, 1, \ldots, k$.  On $T^2M$ they transform as
$$
x^{a'}=x^{a'}(x), \quad \dot{x}^{a'} =   \frac{\partial x^{a'}}{\pa x^b} \, \dot{x}^b, \quad
\ddot{x}^{a'} = \frac{\pa x^{a'}}{\pa x^b} \, \ddot{x}^b +  \frac{\pa^2 x^{a'}}{\pa x^b\pa x^c}\dot{x}^b \dot{x}^c,
$$
{where  $\dot{x}^a = (x^a)^{(1)}$, $\ddot{x}^a =  (x^a)^{(2)}$.}
In general, the gradation of coordinates leads to the concept of a \grB\, i.e.,  a smooth fiber bundle $\sigma^k: E^k \ra M$ in which we are given a distinguished class of fiber coordinates, called graded coordinates. Each graded coordinate is assigned its \emph{weight} and transition functions preserve this gradation. An important assumption is made that weights are non-negative integers. (The index $k$ in $E^k$ indicates that all weights are $\leq k$, in which case we say that the \grB\ $\sigma^k$ is of \emph{order} $k$.  \GrBs\ of order $1$ are nothing more than vector bundles.)

It  has been shown that both the concept of a homogeneity structure and a \grB\ are equivalent \cite{JG_MR_gr_bund_hgm_str_2011}.  A \grB\ associated with a homogeneity  structure $(E, h)$ can be  conveniently encoded by means of the \emph{weight vector field} defined as
$\Delta(p)  =  \dbydt \big |_{t=1}h_t(p)$. In graded coordinates $(x^a, y^i_{w})$ \footnote{This notation means that $(x^a)$ are functions defined (locally) on the base $M$ of the \grB\ $h_0: E\ra M$ while   $(y^i_{w})$ are fiber coordinates in this bundle.  Moreover, the (abundant) notation $y^i = y^i_{w}$ indicates that the function $y^i$ has weight $w$, i.e.,  is a homogeneous function (with respect to $h$) of weight $w$.} we have $\Delta  = \sum_{i}w^i y^i_{w}\pa_{y^i_{w}}$. { A morphism $f$ between \grBs\ $E$ and $F$, colloquially described as a map preserving the gradation of coordinates, can be given a short, precise meaning as a smooth map $f: E \ra F$ such that the corresponding weight vector fields, $\Delta_{E}$ and $\Delta_{F}$, are $f$-related. Equivalently, this can be described as a smooth map intertwining the corresponding homogeneity structures, i.e.,    $f \circ h_t^E  = h_t^F \circ f$ for  every $t\in \R$, where $h_t = h(t, \cdot)$. }

{ In this work, we frequently encounter multi-graded structures like $\T \T^k M$, $\T E^k$ (the tangent bundle of a \grB\ of order $k$) or $\T^k E$ ($\thh{k}$-order tangent bundle of $E$, where $\sigma: E\ra M$ is a vector bundle). In all these examples, there are present  two (compatible) \grB\ structures. Such structures can be described as $(F; \Delta_1, \Delta_2)$ -- a manifold $F$ equipped with two weight vector fields $\Delta_1, \Delta_2$, and
the condition of compatibility can be expressed as  $[\Delta_1, \Delta_2] =0$. Equivalently, the last condition can be stated as $h_t^1\circ h_s^2 = h_s^2 \circ h_t^1$ for any $t, s\in \R$, where $(F, h^i)$ are the homogeneity structures with weight vector field $\Delta_i$, for $i=1, 2$.  Moreover, the bases of the \grBs\ $(F, \Delta_i)$, where $i=1, 2$,  carry induced \grB\ structures. In this paper, we shall mostly encounter the case when one of these \grB\ structures has order 1 (like in $\T E^k$ ot $T^k E$) and will refer to them as \emph{weighted vector bundles.}\footnote{Weighted structures, e.g. weighted algebroids,  where intensively studied in  \cite{BGG_gr_bndls_Lie_grpds15, Bruce_Grabowska_Grabowski_2016}. }
They can be  presented as  a diagram like }
\begin{equation}\label{diag:weighted_VB}
      \xymatrix{
        F^k \ar[r]^{\sigma^k} \ar[d]_{\pi^k} & F^0 \ar[d]^{\pi^0} \\
        \und{F}^k  \ar[r]^{\und{\sigma}^k} & \und{F}^0.
      }
 \end{equation}
where $k$ indicates the order of the \grB\ $\sigma^k:F^k\ra F^0$; 
$\sigma^k$ is a VB morphism and $\pi^k$ is {\new a morphism of \grBs.} 
In the special case $k=1$, we recover the notion of a \emph{double vector bundle} (DVB, in short), e.g. \cite{Mackenzie_lie_2005}.

Given a \grB\ $\sigma^k: E^k\ra M$ of order $k$ and an integer $0\leq j\leq k$ we may consider   a natural projection, denoted by $\sigma^k_j: E^k \ra E^j$, where $E^j$   is a \grB\ of order $j$ over $M$ obtained from $E^k$ by \emph{removing from the atlas for $E^k$ all coordinates of weights greater than $j$}. The graded bundle $E^j$ obtained this way is denoted by $E^k[\Delta\leq j]$   \cite{BGR} and  called \emph{the reduction of $E^k$ to order $j$}. Taking $j=k-1, k-2, \ldots, 0$ we arrive at the \emph{tower of affine bundle projections}
$$
E^k\xrightarrow{\sigma^k_{k-1}} E^{k-1}\xrightarrow{\sigma^{k-1}_{k-2}} E^{k-2}\xrightarrow{\sigma^{k-2}_{k-3}} \ldots \xrightarrow{\sigma^2_1} E^1\xrightarrow{\sigma^1_0} M=E^0.
$$
We have $\sigma^k_j = \sigma^{j+1}_j\circ \ldots \circ \sigma^k_{k-1}$, and we write shortly $\sigma^j$ for $\sigma^j_0$.

A complementary construction is obtained by setting to zero all fiber coordinates in the bundle $\sigma^k: E^k\ra M$ of weight less than a given number  $1<j\leq k$. The resulting submanifold, denoted by $E^k[\Delta\geq j]$, is a graded subbundle of $E^k$ with the same base $M$.  In case $j=k$, $E^k[\Delta\geq k]$ is called the \emph{ core} of $E^k$ and denoted by $\core{E^k}$. The  core can be endowed with a natural VB structure. This way we obtain a functor $\core{\cdot}:\catGB[k] \ra \catVB$, where $\catGB[k]$ is the category of \grBs\ of order $k$, and  $\catVB = \catGB[1]$ is the category of vector bundles. In the case of multi-graded structures $(F; \Delta_1, \ldots, \Delta_n)$, we write $F\in \catGB[k_1, \ldots, k_n]$,  indicating that $(F, \Delta_i) \in \catGB[k_i]$ and $[\Delta_i, \Delta_j]=0$ for $i\neq j$. The  core of the \grB\ $(F, \Delta_1+ \ldots + \Delta_n)$ 
is denoted in the same way {\new as } $\core{F}$. (It will be usually  clear which weight field of $F$ we are referring to.)

\commentMR{A paragraph on $F[X\geq 0]$ removed, as it is not used in this work.}

{ There is an obvious \grB\ structure on the product $E\times F$ of the \grBs\ $E$ and $F$, defined by $h^{E\times F}_t = h^E_t\times h^F_t$ where $t\in \R$.  If $E, F$ have the same base $M$, then $E\times_M F$ is a graded subbundle of $E\times F$.

{ Given  a positive integer $k$ and a vector bundle $E\ra M$ we write $E_{[k]}$ for the \grB\ $(E, k \cdot \Delta)$, where $\Delta$ is the Euler vector field of $E$. Then, for example, $E_{[1]}\times_M F_{[2]}$ refers to  a \grB\ of order two.  It is the \grB\ associated with the graded vector bundle $E_{[1]}\oplus F_{[2]}$, where  $E$, $F$ are VBs over $M$.}
}

\subsection{Double vector bundles and VB-algebroids}

As we already mentioned, a  structure of a DVB on a manifold $D$ is a pair of VBs $\sigma_E: D\to E$ and $\sigma_A: D\to A$ such that for any $x\in D$ and  $t, s\in \R$ holds
$$
    t\cdot_E (s \cdot_A x) = s\cdot_A (t \cdot_E x)
$$
where $\cdot_E$ (respectively, $\cdot_A$) denotes the multiplication by scalars in the vector bundle $\sigma_E$ (resp., $\sigma_A$).
The bases $E$ and $A$ carry induced VB structures over a common base $M$ giving rise to a diagram
\begin{equation}\label{e:VBalg}
\xymatrix{
        D \ar[r]^{\sigma_A} \ar[d]_{\sigma_E} & A \ar[d]^{\sigma^A_M} \\
        E \ar[r]^{\sigma^E_M} & M.
      }
\end{equation}
There is also a third vector bundle over $M$, known as \emph{the core} of the DVB $(D, \sigma_E, \sigma_A)$, defined as the intersection of the kernels of the VB morphisms $\sigma_E$ and $\sigma_A$, $C  = \ker \sigma_E \cap \ker \sigma_A$. From the perspective of graded manifolds, DVBs are $\Z\times \Z$-graded manifolds admitting coordinates only in weights $(0, 0)$, $(1, 0)$, $(0, 1)$ and $(1, 1)$. From this perspective, the core $C$ is the  core of the \grB\ $(D, \Delta_E + \Delta_A)$, where   $\Delta_E$ (resp.,  $\Delta_A$) is the Euler vector field of the vector bundle $\sigma_E$ (resp., $\sigma_A$), and it will be denoted simply as $C = \core{D}$.

{  There is a well-defined action $D \times_M C \ra D$, denoted by $(d, c)\mapsto d\plus c$,  which arises from the affine bundle structure of $D$ over its order-one reduction, $E\times_M A$.
     A section $c\in \Sec(C)$ gives a so-called \emph{core section} $c^\dag$ of the VB $\sigma_E: D \to E$, given by $\sigma(e_m) = e_m \plus c(m)$ where $m\in M$, $e_m \in (\sigma^E_M)^{-1}(m)$.}
     A section $s \in \Sec_E(D)$ is called \emph{linear} if it is a VB morphism from $(E, \sigma^E_M)$  to $(D, \sigma_A)$. The subspace of linear sections (resp. core sections) is denoted  by $\Sec^{\lin}_E(D)$ (resp., $\Sec^{c}_E(D)$).

A \emph{decomposition} of a DVB $D$ as in \eqref{e:VBalg}  is a DVB morphism from $D$ to its \emph{split form} $\bar{D} := E\times_M A\times_M C$
    which is the identity on each component: the {\emph side bundles} $E$, $A$ and the core  $C$. Decompositions are in bijective correspondence with \emph{inclusions} (also referred to as  decompositions)  $\sum: A\times_M E \to D$, which are  DVB morphisms  inducing  the identity on the side bundles $A$ and $E$. 
    Decompositions are also in bijective correspondence with \emph{horizontal lifts} $\theta_A: \Sec(A) \to \Sec^{\lin}_E(D)$, which are defined as splittings of the short exact sequence
    \begin{equation}\label{e:ses_Sec_lin}
        0 \to \Hom(E, C) \to \Sec^{\lin}_E(D) \to \Sec(A) \to 0
    \end{equation}
    { of $\Cf(M)$-modules where $s\in \Sec^{\lin}_E(D)$ projects to its base map, which turns out to be a section of $\sigma^A_M$.}

The foundation on DVBs   were laid by J. Pradines \cite{Pradines1977}. Double structures such as DVBs, as well as  double Lie groupoids and algebroids where extensively studied by K. C. H. Mackenzie and his collaborators (see \cite{Mackenzie_lie_2005} and references therein). In this paper, we shall deal with VB-algebroids -- a pair of  an algebroid and a VB structures,  in compatibility, defined on a common manifold.

The compatibility condition can be stated in various equivalent ways, presenting such a structure as a Lie algebroid object in the category of vector bundles (the origins of the notion of VB-algebroids) or as a vector bundle object in the category of Lie algebroids (LA-vector bundles). See \cite{Gr-Meh2010} for definitions and the equivalence of both concepts.

Following the ideas from \cite{JG_MR_hi_VB_2009}, one can formulate the compatibility condition as follows:  \emph{a VB-algebroid structure on a manifold $D$} is a pair of  VBs $\sigma_A:D\ra A$, $\sigma_E: D\ra E$, and a Lie algebroid structure on the vector bundle $\sigma_E$, such that for each $t\in \R$ the map $x\mapsto t\cdot_A x$, $x\in D$, is an algebroid morphism, see \cite[Definition~2.10]{BGV18}.

It follows that $(D, \sigma_E, \sigma_A)$ is a DVB;  $\sigma^A_M: A\to M$ carries an induced algebroid structure. Moreover, the anchor map $\sharp^D: D\to \T E$ is  a DVB morphism, and $\Sec^{c}_E(D) \oplus \Sec^{\lin}_E(D)$ is a graded Lie algebra, concentrated in degrees $-1$ (the space of   core sections)  and $0$ (the space of linear sections), with respect to the Lie bracket on $\Sec(\sigma_E)$.

\subsection{Linearisation  of graded bundles and the functor \texorpdfstring{$\LL$}{Lambda}} 
We define a functor $\LL: \catGB[k, 1] \ra \catGB[k-1, 1]$. It is slightly more general then 
the functor of linearisation $\pLinr:\catGB[k] \ra \catGB[k-1,1]$ introduced in \cite{Bruce_Grabowska_Grabowski_2016}. Actually, $\pLinr = \LL \circ \T$ is the composition of the tangent functor $\T: \catGB[k] \ra \catGB[k, 1]$ with the functor $\LL$. The construction of the functor $\LL$ is given in two steps.
In the first step, we set to zero all coordinates for $F^k\in \catGB[k,1]$ of weight $(0,1)$. { After shifting in weight by $(-1, 0)$, the target $\LL(F^k)$ is obtained from the latter by removing coordinates of  weight  $(k, 0)$.}

\begin{df}[Functor $\LL$]  \label{df:Lambda} Let $(F^k; \Delta^k_1, \Delta^1_2)$ be a weighted  vector bundle
    as in \eqref{diag:weighted_VB}, { where $(F^k, \Delta^k_1) \in \catGB[k]$ and $(F^k, \Delta^1_2)$ is a vector bundle.}
    Let $\LLv$ denote the kernel of the VB morphism $\sigma^k: F^k\to F^0$.
    Although $\Delta := \Delta_1^k-\Delta_2^1$ is not a combination
with  non-negative coefficients, it is a weight vector field on the submanifold $\LLv\subseteq F^k$ .
{We define the \grB\ $\LL(F^k)$ as
the reduction of the \grB\ $(\LLv, \Delta|_{\LLv})$ from order $k$ to $k-1$, }
\begin{equation}\label{e:Lambda}
   {  \LL(F^k) := \LLv[\Delta|_{\LLv}\leq k-1], \text{ where } \LLv = \ker  \sigma^k.}
\end{equation}
\end{df}
In other words,  we set to zero the coordinates of weight $(0, 1)$ and then we remove the coordinates of weight $(k, 0)$.
{ Consider the following diagrams:
$$
    \xymatrix{F^k\ar[d]  & \ar@<0.5ex>@{_{(}->}[l] \ker \sigma^k \ar[d]  \ar[r] & \LL(F^k) \ar[d] \\
    \und{F}^k & \ar[l]_{=} \und{F^k} \ar[r]^{\und{\sigma}^k_{k-1}} & \und{F}^{k-1}
    }, \quad \quad
    \xymatrix{
        \LL(F^k)\ar[r] \ar[d] & \LL(F^1) \simeq \core{F^1} \ar[d] \\
        \und{F}^{k-1} \ar[r] & M
      }
$$
}
{
In the diagram on the left, the projection  $\ker \sigma^k \to  \LL(F^k)$ is a fiber-wise linear isomorphism, so $\ker \sigma^k$ is the pullback of the vector bundle $\LL(F^k)$ with respect to the projection $\und{\sigma}^k_{k-1}$.

In the diagram on the right, $\LL(F^k)\in  \catGB[k-1, 1]$ is recognized as a weighted vector bundle whose weight vector fields are inherited from $\Delta$ and $\Delta_2^1$. The base of the \grB\  $\LL(F^k)$  is identified as the core  of the DVB $F^1\in\catGB[1,1]$. } {\newMR If $(x^a, y^i_{(\alpha, \beta)})$ are graded  coordinates on the weighted VB $F^k$, then the adapted coordinates on $\LL(F^k)$ are obtained by omitting those  $y^i_{(\alpha, \beta)}$ with $(\alpha, \beta) \in \{(k, 0), (0, 1)\}$, and the coordinates of weight $(w, 1)$ are assigned a new weight $(w-1, 1)$.}

\begin{lem}\label{l:LL} Let $\sigma^k: E^k \ra M$ be a \grB\ of order $k$. There are canonical isomorphism of weighted vector bundles:
\begin{enumerate}[(i)]
\item \label{item:LL_TEk} If $\sigma^k: E^k \ra M$ is a \grB\ of order $k$, then $\LL(\T E^k) \simeq \pLinr(E^k)$.
\item \label{item:LL_TkE} If $\sigma: E\ra M$ is a vector bundle, then  $\LL(\T^k E) \simeq \T^{k-1} E$.
\end{enumerate}
\end{lem}
\begin{proof} Only \eqref{item:LL_TkE} needs a proof, as  \eqref{item:LL_TEk} 
follows directly from the construction of $\LL$ and the linearisation functor.

 For the proof of \eqref{item:LL_TkE}, observe   that the inclusion $\LLv(\T^k E) \subset \T^k E$ is realized by the mapping
\begin{equation}\label{e:LL_incl}
\T^{k-1} E\times_{\T^{k-1}M}\T^k M \hookrightarrow \T^k E, \quad (\tclass{k-1}{a}, \tclass{k}{\gamma})\mapsto \tclass{k}{t\mapsto t a(t)},
\end{equation}
 where curves $a:\R\ra E$ and $\gamma: \R\ra M$ are such that $\sigma\circ a = \gamma$.  Indeed,  the image of the mapping \eqref{e:LL_incl} is the subbundle $\T^k_M E \subset \T^k E$. In the standard local coordinates $(x^a, y^i)$  on $E$, { it }is given by the vanishing coordinates of weight $(0,1)$, i.e., {\new   $\LLv(\T^k E) = \T^k_M E = \{(x^{a, (\K)}, y^{i, (\KL)}):  y^{i, (0)}=0\}$.}

Finally, we realize that the canonical projection $\LLv(\T^k E) \ra \LL(\T^k E)$ defined locally by  removing coordinates of weight $(k, 0)$, i.e.,  the coordinates $x^{a, (k)}$,  coincides with the projection
$$
\T^{k-1} E\times_{\T^{k-1}M}\T^k M \ra \T^{k-1} E, \quad (\tclass{k-1}{a}, \tclass{k}{\gamma})\mapsto \tclass{k-1}{a}.
$$
{\new From \eqref{e:LL_incl} we easily find that the obtained isomorphism $\T^{k-1} E \ra \LL(\T^k E)$, denoted by $I^k_E$,  has the formula \begin{equation}\label{e:IkE-coord}
    {\new
    \pullback{(I^k_E)}(x^{a, (\K)}) = x^{a, (\K)}, \quad  \pullback{(I^k_E)}(y^{i, (\KL)}) = \KL\dot y^{i, (\KL-1)}
    }
\end{equation}
where $(x^{a, (\K)}, y^{i, (\KL)})$, $0\leq \K \leq k-1$, $1\leq \KL \leq k$, are the coordinates for $\LL(\T^k E)$ inherited from  $\T^k E$.}
\end{proof}
\subsection{Natural inclusions and isomorphisms I}
For later use, we shall fix the natural inclusions:
\begin{equation}\label{e:i_kM}
\jM{k}: \T M \ra \core{\T^k M} \subset \T^k M, \tclass{1}{\gamma} \mapsto \tclass{k}{t\mapsto \gamma(t^k/k!)},
\end{equation}
and
\begin{equation}\label{e:iM_kl}
\iM{k,l}: \T^{k+l} M \ra \T^k \T^l M, \tclass{k+l}{\gamma} \longmapsto \tclass{k}{t\mapsto \tclass{l}{s\mapsto \gamma(t+s)}},
\end{equation}
so $\iM{k, l}(\tclass{k+l}{\gamma}) = \tclass{k}{\jet{l}\gamma} = {\left(\jet{k}\jet{l} \gamma\right)(0)}$ where $\jet{l}\gamma: \R \ra \T^l M$ is the $l$-th tangent lift of the curve $\gamma$.
In coordinates,
\begin{equation}\label{e:jMk_coord}
\jM{k}(x^a, \dot{x}^a) = (x^a, 0, \ldots 0,  \dot{x}^a),
\end{equation}
\begin{equation}\label{e:iMkl_coord}
(\iM{k, l})^\ast(x^{a, { (\K, \KL)}}) = x^{a, (\K+\KL)}.
\end{equation}
{ where $x^{a, (\K, \KL)} =  \left(x^{a, (\K)}\right)^{(\KL)}$. } In addition to {  $\jM{k}$}, given a vector bundle $\sigma: E\ra M$, { there is a canonical VB isomorphism of the  core bundle of $(\T^{k-1} E, \dd_{\T^{k-1}} \Delta_E + \Delta_{\T^{k-1} E})$ and the vector bundle $(E, \Delta_E)$} which is defined by
\begin{equation}\label{df:jEk}
    \jE{k}: E \xrightarrow{\simeq} \core{\T^{k-1} E} \subset \T^{k-1} E, \quad v \mapsto \tclass{k-1}{t\mapsto \frac{t^{k-1}}{(k-1)!} v}
\end{equation}
The compatibility with the map { $\jM{k}$} is expressed by the commutative diagram
\begin{equation}\label{e:comp_jMjTM}
\xymatrix{
\T M  \ar[rr]^{\jM{k}} \ar[rrd]^{\jmath^{k-1}_{\T M}} &&  \T^k M \ar[d]^{\iM{k-1, 1}} \\
&& T^{k-1}\T M
}
\end{equation}
A \grB\ $(E^k, \Delta)$ embeds naturally into its linearisation via the digitalisation map
\begin{equation}\label{e:diag}
    \diag^k: E^k \hookrightarrow \pLinr{E^k}, \quad  (\diag^k)^{*}(\dot{y}^i_w) =  w y^i_w, \, 1\leq w\leq k
\end{equation}
{ in the adapted coordinates $(x^a, y^{i'}_{w'}; \dot{y}^{i}_{w})$, where $1\leq w\leq k$, $1\leq w'\leq k-1$,  on $\pLinr(E^k)= \LL(\T E^k)$ induced from $\T E^k$, as mentioned earlier. } Moreover, $\diag^k$ covers  the identity over $E^{k-1}$.
This map is induced from the weight vector field considered as a map $\Delta: E^k\ra \V E^k$, {\new where $\V E^k = \ker \T \sigma^k$ denotes the \emph{vertical} subbundle of $\T E^k$.} In other words, the weight vector field $\Delta$ is projectable with respect to the canonical projection $\V E^k \ra \pLinr(E^k)$.
Moreover, in the special case $E^k = \T^k M$,  the map $\diag^k$ coincides with  $\iM{1, k-1}: \T^k M \ra \T \T^{k-1} M$ composed with  the inverse of the isomorphism $I: \pLinr(\T^k M) = \LL(\T \T^k M) \xrightarrow{\LL(\kappa^k_M)} \LL(\T^k \T M) \xrightarrow{I^k_{\T M}}  \T^{k-1} \T M  \simeq  \T \T^{k-1} M $, where $I^k_{\T M}$ is the isomorphism  established in  {\new Lemma~\ref{l:LL} (\ref{item:LL_TkE})}. The isomorphism $I: \pLinr(\T^k M) \ra \T \T^{k-1} M $  coincides with the isomorphism found in (\cite[Example 2.2.3]{BGG_2015}, \cite{BGR}) and is  given by
$$
I^\ast (\dd x^{a, (\K)}) = \frac{1}{\K+1} \dd x^{a, (\K+1)}
$$
for $\K=0, 1, \ldots, k-1$.

\begin{lem}\label{l:plin_incl} Let $\sigma_i: E^k_i\ra M_i$, for $i=1, 2$, be \grBs\ of order $k$ and let $\phi: E^k_1\ra E^k_2$ be a $\catGB[k]$-morphisms. Then the linearisation of $\phi$ intertwines the canonical inclusions $\diag_i: E_i^k \hookrightarrow \pLinr(E^k_i)$:
$$
\xymatrix{\pLinr(E^k_1) \ar[r]^{\pLinr(\phi)} & \pLinr(E^k_2) \\
E^k_1 \ar@{^{(}->}[u]^{\diag_1} \ar[r]^\phi & E^k_2 \ar@{^{(}->}[u]^{\diag_2}
}
$$
\end{lem}
\begin{proof}
The map $\pLinr(\phi)$ is the unique map which makes the following diagram commutative:
 $$
\xymatrix{
\V E^k_1 \ar[d] \ar[r]^{\T \phi} & \V E^k_2 \ar[d] \\
\pLinr(E^k_1) \ar[r]^{\pLinr(\phi)} & \pLinr(E^k_2)
}
$$
{\newMR where $\V E^k_i = \ker \T \sigma_i \subseteq \T E^k_i$. } { The weight vector fields $\Delta_1$, $\Delta_2$ are $\phi$-related as $\phi: E^k_1\ra E^k_2$ is a $\catGB[k]$-morphism ( \cite[Theorem~2.3]{JG_MR_gr_bund_hgm_str_2011}). Hence $\T \phi \circ \Delta_1 = \Delta_2 \circ \phi$, and the thesis follows directly from the definition of the diagonalisation map.}
\end{proof}

\subsection{Vector bundle comorphisms}

We shall recall the definition of a comorphism between vector bundles from \cite{MJ_MR_HA_comorph_2018} where one can also find more information  and references on the origins and generalizations  of this concept.

\begin{df} \label{df:VBC} A \emph{vector bundle comorphism} (VB comorphism, for short), from a vector bundle $\sigma_1: E_1\ra M_1$ to a vector bundle $\sigma_2: E_2\ra M_2$, is a relation $r\subset E_1\times E_2$,  for which there exist a \emph{base map} $\und{r}: M_2\ra M_1$ and a VB morphism $r^{!}: \pullback{\und{r}} E_1 \ra E_2$ covering the identity on $M_2$ such that
$$r = \{(v, r^{!}(v, y)):  v\in E_1, y\in M_2,  \sigma_1(v) = \und{r}(y)\}$$
where $\pullback{\und{r}} E_1 \subset E_1 \times M_2$ is the pullback of the vector bundle $\sigma_1$ with respect to the map $\und{r}$.
We say that the base map $\und{r}: M_2\ra M_1$ (which is uniquely defined) \emph{covers} $r$, and we depict this in the following diagram:
$$
\xymatrix{E_1 \ar[d]^{\sigma_1} \ar@{-|>}[r]^{r} & E_2 \ar[d]^{\sigma_2}  \\ %
M_1  & M_2 \ar[l]_{\und{r}}  }.
$$
\end{df}
Thus, $r$ is { the } union of graphs of linear maps $r_y: (E_1)_{\und{r}(y)} \ra (E_2)_{y}$ between
the corresponding fibers, where $y$ varies in $M_2$.
{\new There is a one-to-one correspondence between VB comorphisms $\sigma_1\rel \sigma_2$ and VB morphisms $\sigma_2^\ast \ra \sigma_1^\ast$ between the dual bundles.}
A \VBC $r: \sigma_1\rel \sigma_2$ gives rise to a mapping between the spaces of sections,
$$
\ZMmap{r}:\Sec(\sigma_1)\ra \Sec(\sigma_2), \quad \ZMmap{r}(s)(y) = r_y(s(\und{r}(y))).
$$
The map $\ZMmap{r}$ satisfies
\begin{equation} \label{e:hat_r_axiom}
\ZMmap{r}(s + s')  =\ZMmap{r}(s) + \ZMmap{r}(s'), \quad \quad \ZMmap{r}(f\cdot s) = \und{r}^\ast(f) \cdot \ZMmap{s}
\end{equation}
 and any such map gives rise to a \VBC $r: \sigma_1 \rel \sigma_2$.

{ VB comorphisms }form a category denoted by $\catZM$. A morphism from  $r\in \catZM$ to $r'\in\catZM$, where
$r: \sigma_1\rel \sigma_2$ and $r': \sigma_1'\rel \sigma_2'$ are \VBC and $\sigma_i: E_i\ra M_i$, $\sigma_i': E_i'\ra M_i'$ are vector bundles, is given by a pair $(\phi_1, \phi_2)$ of VB morphisms $\phi_i: E_i\ra E_i'$ such that $(\phi_1 \times \phi_2)(r) \subset r'$ (\cite[Definition 2.3 and Proposition 2.6]{MJ_MR_HA_var_calc_2013}). {\new  It is denoted by $(\phi_1, \phi_2): r \Rightarrow r'$.}

A \VBC $r: \sigma_1\rel \sigma_2$ is \emph{weighted} of order $k$ if the total spaces $E_1, E_2$ are given a structure of a \grB\ of order $k$ with respect to which $r$ is a graded subbundle of the product $E_1\times E_2$.

We shall need the following result in Section~\ref{sec:str}. Roughly speaking, it states that  $\LL$ is also a functor on  the category of weighted vector bundle comorphisms.
\begin{lem} \label{lem:lambda_of_ZM}
   Let $F_1^k, F_2^k\in \catGB[k,1]$ be weighted, order $k$, vector bundles and let $\pi_i: F_i^k \ra \und{F}_i^k$ denotes the corresponding VB projections.  Let $r^k: \pi_1 \relto \pi_2$ { be } a weighted, order $k$, \VBC covering $\und{r}^k: \und{F}_2^k \ra \und{F}_1^k$. Then
$\LL(r^k): \LL(F_1^k) \relto \LL(F_2^k)$ is a \VBC covering $\und{r}^{k-1}: \und{F}_2^{k-1} \ra \und{F}_1^{k-1}$:
$$
\xymatrix{
\LL(F_1^k) \ar[d] \ar@{-|>}[rr]^{\LL(r^k)} && \LL(F_2^k) \ar[d] \\
\und{F}_1^{k-1} && \und{F}_2^{k-1} \ar[ll]_{\und{r}^{k-1}}
}
$$
{ Moreover, if $(\phi_1, \phi_2): r \Rightarrow r'$ is  a morphism between weighted VB comorphisms $r: F_1 \relto F_2$ and $r':{F_1'}\relto {F_2'}$ then $(\LL(\phi_1), \LL(\phi_2)): \LL(r) \Rightarrow \LL(r')$ is the  same.  }
\end{lem}

{
\begin{proof}
    Note that $r^k$ is  a weighted vector subbundle of $F_1^k\times F_2^k$, hence $\LL(r^k)$ is a weighted vector subbundle of $\LL(F_1^k\times F_2^k) = \LL(F_1^k)\times \LL(F^2_k)$. Let us trace the subsequent steps of the construction of the weighted vector bundle $\LL(r^k)$, as in  Definition~\ref{df:Lambda}. We have $\LLv(r^k) = r^k \cap (\ker \sigma^k_1 \times \ker  \sigma^k_2)$ where $\sigma_i^k: F_i^k \ra F^0_i$ are as in \eqref{diag:weighted_VB} for $i=1, 2$. Hence, $\LLv(r^k)$ is a VB comorphism $\ker \sigma^k_1 \relto \ker \sigma^k_2$ covering $\und{r}^k: \und{F}_2^k \ra \und{F}_1^k$. The goal  $\LL(r^k)$ is obtained from $\LLv(r^k)$ by the reduction to order $k-1$ of the base map $\und{r}^k$. Since  the projections $\LLv(F^k_i) \ra \LL(F^k_i)$ are fiber-wise linear isomorphisms, $\LL(r^k)$ remains a VB comorphism.

    For the last part of Lemma, we have already noticed that the functor $\LL$ preserves the products and inclusions. By \cite[Proposition~2.6]{MJ_MR_HA_comorph_2018} $(\phi_1\times \phi_2)(r)\subseteq  r'$, hence $(\LL(\phi_1) \times \LL(\phi_2))(\LL(r)) \subseteq r'$, so $(\LL(\phi_1), \LL(\phi_2))$ is  a morphism in the category $\catZM$.
\end{proof}
}

The  core $\core{E^k}$ acts naturally on {the \grB\ $E^k$.}
{ This action
$E^k\times_M \core{E^k} \to E^k$ is denoted by $(a^k, v)\mapsto a^k\plus v\in E^k$ and gives rise to a VB comorphism,}
\begin{equation}\label{df:corePlus}
\xymatrix{
v \in \core{E^k}  \ar[d] \ar@{-|>}[rr] && \T E^k \ni \coreVF{v}(a^k) \ar[d] \\ 
M && E^k\ni a^k \ar[ll]_{\sigma^k}
}
\end{equation}
{\new where  $\coreVF{v}(a^k)\in \T_{a^k} E^k$  is the vector represented by the curve $t\mapsto a^k \plus (t v)$.}
 In coordinates $(x^a, y^i_w, z^\mu_k)$ on $E^k$, where $y^i_w$'s have weights $1\leq w \leq k-1$, { and $\w(z^\mu_k) =k$,}  the associated map on sections is given by   $v\mapsto \coreVF{v}   =  v^\mu(x) \pa_{z^\mu_k}$, where $v = \sum_\mu v^\mu(x) c_\mu  $ and $(c_\mu)$ is  a local frame of $\Sec(\core{E^k})$. { Since $E^k\ra E^{k-1}$ is  an affine bundle modelled on the pullback of the core $\core{E^k}\ra M$, there is}  a map
 \begin{equation}\label{df:minus}
 E^k \times_{E^{k-1}} E^k \ra \core{E^k}, \quad (a', a) \mapsto a'\minus a,
 \end{equation}
 where $a'-a$ is the unique vector $v\in \Sec(\core{E^k})$ such that $a \plus v =a'$.
 \begin{lem}\label{l:VF-k}
 The  mapping associated with the \VBC \eqref{df:corePlus},
\begin{equation}\label{e:Sec_Ek VFk}
    \Sec(\core{E^k})\rightarrow \VF_{-k}(E^k), v\mapsto \coreVF{v},
 \end{equation}
is a {$\Cf(M)$-module isomorphism. }
Moreover, if $\sigma_i: E^k_i\ra M_i$, for $i=1, 2$, are \grBs\ of order $k$ and  $\phi: E^k_1\ra E^k_2$ is a $\catGB[k]$-morphism  then weight $-k$ vector fields $X_i\in \VF_{-k}(E^k_1)$   are  $\phi$-related if and only if
the corresponding sections $v_i\in \Sec(\core{E^k_i})$ are $\core{\phi}$-related. { If } $M_1=M_2$ and  $\phi$ covers the identity, then the last condition means that  $v_2 = \core{\phi}\circ  v_1$.
\end{lem}
{
\begin{proof} Let $v_i\in \Sec(\core{E^k_i})$, $X_i= \coreVF{v_i}$ for $i=1,2$.  The  vector fields $X_i$ are represented by the families of curves $t\mapsto a_i \plus t v_i(\und{a_i})$ where $a_i\in E_i^k$ and $\und{a_i} = \sigma^k_i(a_i)\in M_i$.
The sections $v_1, v_2$ are $\core{\phi}$-related if and only if $v_2(m_2) = \core{\phi}(v_1(m_1))$ for any pair  $(m_1, m_2)$ such that $\und{\phi}(m_1) = m_2$. Note that $\phi(a_1 \plus t v_1(\und{a_1})) = \phi(a_1) \plus  t \core{\phi}(v_1(\und{a_1}))$, hence if $v_1, v_2$ are $\core{\phi}$-related then $(\T\phi)X_1(a_1) = X_2(\phi(a_2))$. Thus, $X_1, X_2$ are $\phi$-related. The proof in the converse direction is very similar { and is left to the reader. }
\end{proof}}

\subsection{Higher algebroids}
It is  well-known that a Lie algebroid $(\sigma: A\to M, [\cdot, \cdot], \sharp)$ can be represented as a linear Poisson tensor on $A^\ast$, the total space of the dual vector bundle. This, in turn,   gives rise to a VB morphism $\veps: \T^\ast A \to \T A^\ast$ which is a Poisson map and retains  all the information about the algebroid structure on $A$. The dual of $\veps$ is a VB comorphism $\kappa: \T \sigma \relto \tau_A$ which was a starting point in the concept of HAs originated in \cite{MJ_MR_HA_var_calc_2013}.

A general  algebroid structure on a vector bundle $\sigma: E\ra M$ can be encoded as a \VBC $\kappa: \T \sigma \rel \tau_{E}$ of a special kind, see \cite{MJ_MR_HA_comorph_2018}, Proposition~2.15. In this correspondence
$\kappa$ should be also a vector subbundle of $\tau_E \times \T \sigma$, and the induced VB morphism between the  core bundles should be the identity, {
\begin{equation} \label{e:id_core}
    \core{\kappa} = \id_{\core{\T E}} 
\end{equation}
Let us recall that the core  of the  DVB
$\T E$ is  the subbundle $\V_M E$ of the vertical bundle $\V E$ of $E$, and it is naturally identified with the vector bundle $E$ itself. \commentMR{The vertical subbundle for a \grB\ was introduced just after \eqref{e:diag}.}
Moreover,  algebroid morphisms $\phi: (E_1, \kappa_1)\ra (E_2, \kappa_2)$ are in a one-to-one correspondence with $\catZM$-morphisms $(\T \phi, \T \phi):\kappa_1\mZM \kappa_2$.
The above concept of an algebroid has a direct analogue in higher-order, which we shall recall now.

\begin{df}\cite{MJ_MR_HA_comorph_2018} \label{df:higher_algebroid} A general ($\thh{k}$-order) higher algebroid (HA, in short) is a
   \grB\ $\sigma^k: E^k\ra M$ of order $k$ together with a weighted \VBC $\kappa^k \subset \T^k E^1\times \T E^k$ from $\T^k \sigma^1$ to $\tau_{E^k}$ (covering a  mapping $\sharp^k: E^k \ra \T^k M$) such that  the relation $\kappa^1:\T \sigma^1\relto \T \tau_{E^1}$, being the reduction to order one of $\kappa^k$, equips $\sigma^1:E^1\ra M$ with an algebroid structure:
\begin{equation}
\label{diag:HA}
\xymatrix{\T^kE^1\ar[d]^{\T^k\sigma^1} \ar@{-|>}[rr]^{\kappa^k}&&\T E^k\ar[d]^{\tau_{E^k}}\\
\T^kM &&\ar[ll]_{\sharp^k} E^k
}
\end{equation}
In addition:
\begin{enumerate}[(i)]
    \item If { $\kappa^1$} is a symmetric relation, then the HA $(E^k, \kappa^k)$ is called  \emph{skew}.\footnote{This is equivalent to saying that the bracket  $[\cdot, \cdot]$ on $\Sec(E^1)$ is skew-symmetric, see \cite{MJ_MR_HA_comorph_2018}.}
    \item\label{i:dHA:AL_axiom}  If $(\sigma^1, \kappa^1)$ is skew {and, in addition, the diagram}
\begin{equation}\label{diag:AL}
\xymatrix{\T^kE^1\ar[d]^{\T^k\sharp^1} \ar@{-|>}[rr]^{\kappa^k}&&\T E^k\ar[d]^{\T \sharp^k}\\
\T^k\T M \ar[rr]^{\kappa^k_M} && \T \T^k M
}
\end{equation}
is commutative, i.e.,  $(\T^k \sharp^1, \T\sharp^k): \kappa^k \mZM \kappa^k_M$ is a morphism in $\catZM$, then we call $(\sigma^k, \kappa^k)$ an \emph{almost Lie higher algebroid};
 \item\label{i:dHA:Lie_axiom} Both vector bundles, $\T^k \sigma^1$ and $\tau_{E^k}$ in the diagram \eqref{diag:HA}, carry a canonical algebroid structure.\footnote{ The $k$-tangent lift of $(\sigma^1, \kappa^1)$ gives an algebroid structure on $\T^k \sigma^1$.}   {\new If $(E^k, \kappa^k)$ is a skew HA and  $\kappa^k$ }  is a subalgebroid of the product of these algebroids then $(\sigma^k,\kappa^k)$ is called a \emph{Lie} HA.\footnote{\new This condition can be restated as the dual VB morphism $\veps^k: \T^\ast E^k \ra \T^k (E^1)^\ast$ is a Poisson map, e.g. \cite{JG_mod_class_skew_alg_rel_2012}. Moreover, a Lie HA has to be AL, i.e., the condition \eqref{i:dHA:Lie_axiom} implies \eqref{i:dHA:AL_axiom}, see \cite{MJ_MR_HA_comorph_2018}. }
\end{enumerate}
 A \emph{morphism} between higher algebroids  $(\sigma^k_E:E^k\ra M,  \kappa^{k,E})$ and  $(\sigma^k_F:F^k\ra N, \kappa^{k,F})$ is a morphism of \grBs\ $\phi^k: E^k\ra F^k$ such that $(\T^k \phi^1, \T \phi^k): \kappa^{k,E}\mZM \kappa^{k,F}$ is a $\catZM$-morphism. { Higher algebroids with $\catZM$-morphisms } form a category.  
 The reduction of a  HA $(E^k, \kappa^k)$  to a lower  order $j$, $1\leq j<k$ gives a HA denoted by $(E^j, \kappa^j)$ which is skew (resp. AL, Lie) if $(E^k, \kappa^k)$ was so.
\end{df}

\begin{ex}[\cite{MJ_MR_HA_comorph_2018} HAs of order $2$, in coordinates] \label{ex:kappa2_coord}
Let $(x^a, y^i, z^\mu)$ be  local graded coordinates on a  \grB\ $\sigma^2: E^2\ra M$  of order $2$.
Taking into account only the graded bundle structure of $\kappa^2$, we obtain the following system of equations for $\kappa^2 \subset \T^2 E^1 \times_M \T E^2 \in\catGB[1,2]$. (We have underlined the coordinates on $\T E^2$ in order to distinguish them from the coordinates on  $\T^2 E^1$.)
\begin{equation}\label{e:local_kappa2}
\kappa^2:
\begin{cases}
\dot{x}^a = & Q^a_i \,\und{y}^i \\
\ddot{x}^a = & \frac{1}{2}\, Q^a_{ij}\,\und{y}^i\und{y}^j + Q^a_\mu\,\und{z}^\mu, \quad \text{ where } Q^a_{ij}=Q^a_{ji}, \\
\dot{\und{x}}^a = & \tilde{Q}^{a}_i\,y^i
\\
\dot{\und{y}}^i = & Q^i_j \dot{y}^j +  Q^i_{jk}\,\und{y}^j y^k, \quad \text{ where } Q^i_j =\delta^i_j,
\\
\dot{\und{z}}^\mu = & Q^\mu_i\,\ddot{y}^i + Q^\mu_{ij} \, \und{y}^i \dot{y}^j + Q^\mu_{\nu i}\,\und{z}^\nu y^i + \frac{1}{2} Q^\mu_{ij,k} \und{y}^i\und{y}^j y^k, \quad \text{ where } Q^\mu_{ij,k}=Q^\mu_{ji,k},
\end{cases}
\end{equation}
for some \emph{structure functions} $Q^{\cdots}_{\cdots}$.  {\new The condition $Q^i_j =\delta^i_j$ corresponds to  \eqref{e:id_core} and   it ensures that the order-one reduction of $\kappa^2$ gives a (general) algebroid structure on $A=E^1$. }   If $(E^2, \kappa^2)$ is a skew HA then $\tilde{Q}^{a}_i=Q^a_i$ and $Q^i_{jk} =  - Q^i_{kj}$  since {  $\kappa^1$ is a symmetric relation. } The structure functions satisfy certain equations reflecting the axioms of a higher-order algebroid. These equations are derived in Appendix, Subsection~\ref{sSec:AL_and_Lie_HA_eqns}.
\end{ex}
{
\begin{ex} \label{ex:kappa^k_M} The natural 
diffeomorphism $\kappa^k_M : \T^k \T M\ra \T \T^k M$ defines a Lie, order $k$ HA on $\tau^k_M: \T^k M\ra M$. {\new Indeed,  $(\T^k M, \kappa^k_M)$ satisfies the Lie condition (Definition~\ref{df:higher_algebroid}\eqref{i:dHA:Lie_axiom}) because  $\veps^k_M : = (\kappa^k_M)^\ast: \T^\ast \T^k M \ra \T^k \T^\ast M$ is a Poisson map.}  It also comes from a more general result, see \cite[Proposition~4.13]{MJ_MR_HA_comorph_2018}.
\end{ex}
}

\subsection{Reformulation of the definition of  a HA  in terms of algebroid lifts}

{
We shall review the construction of higher lifts $s^{(\K)}$ of sections of  a vector bundle. This notion is used in various parts of this work, such as in the definition of algebroid lifts $\aliftB{s}{\K-k}\in \VF(E^k)$ (see \eqref{df:algebroid_lift}), which facilitate the convenient description of the axioms of HAs (see Theorem~\ref{th:HA_axioms_and_lifts}).
}

{\newMR Fix $k\in \N$ and let $s\in \Sec(\sigma)$ be a section of a vector bundle $\sigma: E\to M$. We can interpret $s$ as a linear function $\iota(s)$ on $E^\ast$, the linear dual of $E$. Let $0\leq \K\leq k$. Then the $(\K)$-lift of $\iota(s)$ is a function on $\T^k E^\ast$, commuting with $h^{\T^k E^\ast}_t = \T^k (h^{E^\ast}_t)$, the homogeneity structure on $\T^k E^\ast$.  Therefore, $\iota(s)^{(\K)}$ can be interpreted as a section of the linear dual of the vector bundle $\T^k \sigma^\ast:  \T^k E^\ast \ra \T^k M$, which is identified with the vector bundle $\T^k \sigma: \T^k E \ra \T^k M$ via the non-degenerate pairing
\begin{equation}\label{e:T^k_pair}
\pair{\cdot, \cdot}_{\T^k \sigma} : \T^k E^\ast \times_{\T^k M} \T^k E \simeq \T^k (E^\ast \times_M E) \xrightarrow{\pair{\cdot, \cdot}_\sigma^{(k)}} \R,
\end{equation}
obtained as $(k)$-lift  of the pairing $\pair{\cdot, \cdot}_\sigma: E^\ast\times_M E \ra \R$. The section of $\T^k \sigma$ obtained this way is denoted by $s^{(\alpha)}$ and called \emph{the $(\K)$-lift of the section $s$}.  In standard coordinates $(x^a, y^i)$ on $E$, and $(x^a, \xi_i)$ on $E^\ast$, where $\xi_i = \iota(e_i)$,  the $(k)$-lift of the function $\pair{\cdot, \cdot}_\sigma = y^i \xi_i$ is obtained using the general Leibniz rule, and has the form
$$
    \pair{(x^{a, (\K)}, \xi_i^{(\KL)}), (x^{a, (\K)}, y^{i, (\KL)})}_{\T^k \sigma} = \sum_{\K = 0}^k \binom{k}{\K} \xi_i^{(\K)} y^{i, (k-\K)}.
$$
}
It follows that the family $(e_i^{(\K)})$, where $0 \leq \alpha\leq k$ and $\iota(e_i^{(\K)}) =  \xi_i^{(\K)}$, forms a local frame of sections of the vector bundle $\T^k \sigma$. Moreover,   $\pair{y^{j, (\beta)}, e_i^{(k-\alpha)}}_{\T^k \sigma}  = \delta^i_j\delta^{\alpha}_{\beta} \binom{k}{\alpha}$, hence
\begin{equation}\label{e:e_i_lifts}
y^{i, (\K)} \circ e_i^{(k-\K)} = \binom{k}{\K}^{-1},
\end{equation}
as the composition of functions $e_i^{(k-\K)}:\T^k M \ra \T^k E$ and $y^{i, (\K)}: T^k E\ra \R$. From this it is  straightforward to verify that this construction of
$s^{(\K)}$  is equivalent to the one presented in \cite{MJ_MR_HA_comorph_2018}.  We have
\begin{equation}\label{e:fs_beta}
(f\cdot s)^{(\beta)}= \sum_{\alpha=0}^\beta \binom{\beta}{\alpha} f^{(\alpha)}s^{(\beta-\alpha)}
\end{equation}
for $\beta = 0, 1, \ldots, k$, $f\in \Cf(M)$ and $s\in \Sec(E)$. This is simply  the Leibniz rule for the iterated derivative.}

\begin{df}[Vertical lifts]\label{df:sec_lifts} Let $0\leq \alpha\leq k$. We define a \VBC $\VZM^{k}_{\alpha}$,
\begin{equation}\label{diag:vertical_lift_ZM}
\xymatrix{\T^{\alpha} E \ar[d] \ar@{-|>}[r]^{\VZM^k_{\alpha}} & \T^k E \ar[d]  \\ %
\T^{\alpha} M  & \T^k M \ar[l]_{\tau^k_{\alpha}}}
\end{equation}
covering the natural projection $\tau^k_{\alpha}: \T^k M\ra
\T^{\alpha} M$ by
\begin{equation}\label{df:V}
    \left(\VZM^k_{\alpha}\right)_{\tclass{k}{\gamma}}(\tclass{\alpha}{a}) = \tclass{k}{t\mapsto \frac{\alpha!}{k!} t^{k-\alpha} a(t)}
\end{equation}
where $\gamma$ is a curve in $M$ and  $a$ is a curve in $E$  such that $\tclass{k}{\und{a}} = \tclass{k}{\gamma}$ where $\und{a} = \sigma\circ a$.
\end{df}
Note that for $\alpha=k-1$ we recover the map \eqref{e:LL_incl}.  It is clear that \eqref{df:V} does not depend on the choice of representatives $\gamma$ and $a$.

{\newMR The $(\alpha)$-lift $s^{(\alpha)}$ of the section $s\in\Sec(E)$ can be  presented as the composition of the complete lift $\T^{\alpha} s$ with the vertical lift $\ZMmap{\VZM}^k_{\alpha}$:
\begin{equation}\label{df:alpha_lifts}
s^{(\alpha)}= \ZMmap{\VZM}^k_{\alpha}(\T^{\alpha}s).
\end{equation}
A  simple coordinate-based proof is left to the reader.
}

{\new A (general) algebroid structure  $\kappa$ on the vector bundle $\sigma: E\ra M$ can be lifted by means of $\thh{k}$-tangent functor to the vector bundle $\T^k\sigma: \T^k E\ra \T^k M$ (see \cite{MJ_MR_HA_comorph_2018}). The lifted structure is called \emph{$\thh{k}$-order tangent lift of $(\sigma, \kappa)$} and denoted as $(\T^k \sigma,\dd_{\T^k}\kappa)$.
The algebroid bracket $[\cdot,\cdot]_{\T^k \sigma}$ on $\T^k\sigma$ satisfies
\begin{equation}\label{e:algebroid_Tk_E_item_bracket_formula}
[ \frac{k!}{(k-\alpha)!} s_1^{(k-\alpha)}, \frac{k!}{(k-\beta)!}  s_2^{(k-\beta)}]_{\T^k \sigma} = \frac{k!}{(k-\alpha-\beta)!} ([s_1, s_2]_{\sigma})^{(k-\alpha-\beta)}\ ,
\end{equation}
for any integers $\alpha,\beta=0,1,\hdots,k$ such that $\alpha + \beta \leq k$,  and any sections $s_1,s_2\in\Sec(E)$.
Additionally,  $[s_1^{(k-\alpha)}, s_2^{(k-\beta)}]_{\T^k \sigma} = 0$ if $\alpha+\beta>k$. Moreover,  if $(\sigma,\kappa)$ is a skew/AL/Lie algebroid, then so is $(\T^k \sigma,\dd_{\T^k}\kappa)$.
}

Assume  that $(\sigma, \kappa)$ { is Lie.}
From \eqref{e:algebroid_Tk_E_item_bracket_formula}, we observe that assigning the weight $\alpha-\beta$ to a section of the form $f^{(\alpha)} s^{(k-\beta)}$, where $f\in \Cf(M)$ and $s\in \Sec(E)$, turns the Lie subalgebra of $\Sec(\T^k \sigma)$ generated by homogeneous sections into a graded Lie algebra concentrated in weights $\geq -k$. This Lie algebra  has a Lie subalgebra $\Sec_{\leq0}(\T^k \sigma)$ generated by homogeneous sections of non-positive weights.  It is of finite rank over $\Cf(M)$.

Using the structure of a higher algebroid on a \grB\ $\sigma^k: E^k \ra M$ one can define 
\emph{algebroid lifts} of a section $s\in\Sec_M(E^1)$ as  follows:
{\begin{equation}\label{df:algebroid_lift}
\aliftB{s}{-\alpha}:= \frac{k!}{(k-\alpha)!} \ZMmap{\kappa^k}(s^{(k-\alpha)})\in \VF_{-\alpha}(E^k), \quad  {\newMR -k\leq -\alpha\leq 0.}
\end{equation}}
\commentMR{Zamieniliśmy: $\aliftB{s}{-\alpha} :=  \frac{k!}{(k-\alpha)!} \alift{s}{-\alpha}$. Dla $k=2$ zamieniamy $\aliftB{s}{0} = \alift{s}{0}$ oraz $\aliftB{s}{-\alpha} = 2 \alift{s}{-\alpha}$ dla $\alpha=1, 2$.}
{ The notation is slightly different from that in \cite{MJ_MR_HA_comorph_2018} where the algebroid $(k-\alpha)$-lift of a section $s$ was
denoted by $\alift{s}{k-\alpha}$ and it is related as $\aliftB{s}{-\alpha} = \frac{k!}{(k-\alpha)!} \alift{s}{k-\alpha}$.
Thanks to this correction, the vector field $\aliftB{s}{-\alpha}$ has weight $-\alpha$ and the equation (4.6) in \cite{MJ_MR_HA_comorph_2018} simplifies to
\begin{equation} \label{e:Lie_alifts}
    [\aliftB{s_1}{\K}, \aliftB{s_2}{\KL}] = \aliftB{[s_1, s_2]}{\K+\KL}
\end{equation}
for any $s_1, s_2\in \Sec(E^1)$ and $\K, \KL \leq 0$ such that $-k \leq \K + \KL$.}

Using \eqref{e:fs_beta} we get
\begin{equation}\label{e:fs_lift}
 \aliftB{(fs)}{-\alpha} = \sum_{\beta = 0}^{k-\alpha} \frac{1}{\beta!} (\sharp^k)^\ast f^{(\beta)} \, \aliftB{s}{-\alpha-\beta}.
\end{equation} 
In particular,
    \begin{equation}\label{e:fs_k}
        \aliftB{(fs)}{-k} =  f  \, \aliftB{s}{-k}.
     \end{equation}

{\new Any vector field $X \in \VF_{0}(E^k)$  of weight $0$ has a form
 $$
    X = X^a(x) \pa_{x^a} + \sum_i X^i(x, y) \pa_{y^i_w},
 $$ and has a well defined   projection on $M$,  denoted by $\reduction{X}^k_0 =  X^a(x) \pa_{x^a}\in \VF(M)$. Similarly,
 a vector field $Y \in \VF_{-1}(E^k)$  of weight $-1$  is projectable onto $E^1$, the projection is denoted by  $\reduction{Y}^k_1 \in \VF_{-1}(E^1) \simeq \Sec(E^1)$, see Lemma~\ref{l:structure_VF_Ek}.  }
   Below  is a reformulation of  axioms of higher-order algebroids in terms of algebroid lifts.
\begin{thm}\label{th:HA_axioms_and_lifts}  Let $\sigma^k: E^k \ra M$ be a \grB\ of order $k$.
\begin{enumerate}[(i)]
    \item \label{i:genHA} {\new Assume that the order-one reduction of $\sigma^k$  is a trivial VB of rank $n$, i.e., it admits a trivialization $E^1 \simeq M\times \R^n$, and
    let $(e_i)_{i=1, \ldots, n}$ be the corresponding  frame of  $\sigma^1: E^1\ra M$.  A general HA is provided by a \grB\ morphism  $\sharp^k: E^k \ra \T^k M$ and a collection of homogeneous vector fields $X_{i, {\K}} \in \VF_{\K}(E^k)$, where $-k\leq \K \leq 0$, $1\leq i \leq n$,   such that the projection of each vector field $X_{i, -1} \in \VF_{-1}(E^k)$ onto $E^1$ coincides with $e_i \in \Sec(E^1) \simeq \VF_{-1}(E^1)$. Moreover, the vector fields which define $\kappa^1$ -- the order-one reduction  of $\kappa^k$, are the projections of $X_{i,0}$ and $X_{i,-1}$ onto  $E^1$.}
    \item \label{i:p:AL_axiom} A skew HA $(E^k, \kappa^k)$ is almost Lie if and only if for any section $s\in \Sec(E^1)$ and $-k\leq \K \leq 0$ the vector fields $\aliftB{s}{\K} \in \VF_{\K}(E^k)$ and $\aliftB{\left(\sharp^1 s\right)}{\K}\in\VF_{\K}(\T^k M)$ are $\sharp^k$-related.
    \item \label{i:p:Lie_axiom} \cite[Proposition~4.9]{MJ_MR_HA_comorph_2018} An almost Lie HA $(\sigma^k, \kappa^k)$ is Lie if and only if
$$
\ZMmap{\kappa^k}|_{\Sec_{\leq0}(\T^k \sigma^1)}: \Sec_{\leq 0}(\T^k\sigma^1)\ra \VF_{\leq 0}(E^k)
$$
is a Lie algebra homomorphism.
\end{enumerate}
\end{thm}
\begin{proof}
    \begin{enumerate}[(i)]
        \item The sections $(e_i^{(\K)})$ form  a frame for $\T^k E^1\ra \T^k M$, hence their pullbacks $\left((\sharp^k)^\ast e_i^{(\K)}\right)$ form a frame for the pullback vector bundle $(\sharp^k)^\ast \T^k \sigma: \T^k E^1 \times_{(\T^k \sigma, \sharp^k)} E^k \ra E^k$.  { To set a comorphism $\kappa^k: \T^k \sigma \relto \tau_{E^k}$, this amounts to defining a VB morphism from the VB  $(\sharp^k)^\ast \T^k \sigma$ to the tangent bundle of $E^k$, covering the identity $\id_{E^k}$.} This is done by assigning vector fields to the sections from the local frame.  We send { $(\sharp^k)^\ast e_i^{(\K)}$ to $X_{i, \K}$. In other words, $\ZMmap{\kappa^k}(e_i^{(\K)}) =  X_{i, \K}$.} Then the obtained comorphism $\kappa^k$ is  weighted, as the vector fields $X_{i, \K}$ are homogeneous and $\sharp^k$ preserves the weight.

            {  The condition \eqref{e:id_core} corresponds to the fact that $e_i\in \Sec(E^1) \simeq \VF_{-1}(E^1)$ coincides with  the projection of $X_{i, 1}$ onto $E^1$. } 

    \item The commutativity of the { diagram} \eqref{diag:AL}, corresponding to  the almost Lie axiom, can be reformulated as follows: For any section $s\in \Sec(\T^k \sigma^1)$, the vector fields $\ZMmap{\kappa^k}(s) \in \VF(E^k)$ and $\kappa^k_M \circ \T^k \sigma^1(s)\in \VF(\T^k M)$ are $\sharp^k$-related (see the proof of \cite[Proposition~4.9]{MJ_MR_HA_comorph_2018}). In particular, in any AL HA $(E^k, \kappa^k)$, for any section $s\in\Sec(E^1)$ and $\K\geq -k$, the vector fields $\aliftB{s}{\K}\in\VF(E^k)$  and $\aliftB{(\sharp s)}{\K}\in \VF(\T^k M)$ are $\sharp^k$-related. (The latter are algebroid lifts with respect to the HA  structure on $\tau^k_M: \T^k M\ra M$.)     On the other hand, if $f\in \Cf(\T^k M)$ and $s\in \Sec(\T^k \sigma^1)$ then
    $$
        \ZMmap{\kappa^k}(f \cdot s) =  (\sharp^k)^\ast(f) \cdot  \ZMmap{\kappa^k}(s), \quad  \kappa^k_M \circ \T^k \sharp^1(f\cdot s) = f \cdot \kappa^k_M \circ \T^k \sharp^1(s).
     $$
     Therefore, if the vector fields $\ZMmap{\kappa^k}(s)$ and  $\kappa^k_M \circ \T^k \sharp^1(s)$ are  $\sharp^k$-related, then the same is true if we replace $s$ with $f \cdot s$. Hence,  the thesis  \eqref{i:p:AL_axiom} holds since sections of the form $s^{(k-\K)}$, where $0\leq \K\leq k$, span $\Sec(\T^k \sigma^1)$ as $\Cf(\T^k M)$-module. The proof of \eqref{i:p:Lie_axiom} is presented in \cite{MJ_MR_HA_comorph_2018}. 
     \end{enumerate}
\end{proof}

\begin{rem} \label{r:Lie_axiom} It  suffices to verify the conditions given in Theorem~\ref{th:HA_axioms_and_lifts} locally.  Moreover,  it is sufficient to take  the sections of the  vector bundle $\sigma^1:E^1\ra M$ to be the elements of a frame $(e_i)$ of local sections. In this way, the almost Lie axiom and Lie axiom can be reduced (locally) to a finite number of equations:
\begin{enumerate}
\item[(AL axiom)]\label{i:AL_axiom} The vector fields $\aliftB{e_i}{-\K}$  and  $\aliftB{(\sharp e_i)}{-\K}$  are $\sharp^k$-related for any $0\leq \K\leq k$.
\item[(Lie axiom)]\label{i:Lie_axiom} 
   { $[\aliftB{e_i}{-\K}, \aliftB{e_j}{-\beta}]_{\tau_{E^k}} =  \aliftB{[e_i, e_j]_{\sigma^1}}{-\K-\beta}$} for any $0\leq \K, \beta$ such that $\alpha + \beta \leq k$.
\end{enumerate}
\end{rem}

\begin{rem} There is also a dual construction of the algebroid lifts $\aliftB{s}{-\K}$ associated with a HA $(E^k, \kappa^k)$, which coincides with the construction presented in \cite{JG_PU_Algebroids_1999} for Lie algebroids, i.e.,  when $k=1$. Given a section $s\in\Sec(E)$  considered as a linear function $\iota(s)$ on $E^\ast$, we have $(\K)$-lifts $\iota(s)^{(\K)} \in \Cf(\T^k E^\ast)$ for $\K = 0, 1, \ldots, k$. As we mentioned (see \eqref{e:T^k_pair}), the vector bundles $\T^k \sigma: \T^k E \ra \T^k M$ and $\T^k \sigma^\ast: \T^k E^\ast\ra \T^k M$ are in natural duality, hence the dual of $\kappa^k$ is  a  weighted vector bundle morphism $\veps^k$ of the form
$$
\xymatrix{\T^\ast E^k\ar[d]^{\tau^\ast_{E^k}} \ar[rr]^{\veps^k} && \T^k E^\ast \ar[d]^{\T^k \sigma^\ast} \\
E^k \ar[rr]^{\sharp^k} && \T^k M
}
$$
By pulling back $\iota(s)^{(\K)}$ via $\veps^k$ we obtain  linear functions on $\T^\ast E^k$, thus vector fields on $E^k$.
{It is evident (by working fiberwise) that  this way we recover our algebroid lifts, i.e,}
$$
 \langle \aliftB{s}{-\K}, \cdot \rangle_{\tau_{E^k}} = (\veps^k)^\ast \iota(s)^{(k-\K)}
$$
\end{rem}

{ Let $(E^k, \kappa^k)$ be a HA and $(E^j, \kappa^j)$ be its reduction to order $j$, where $1\leq j<k$. The following lemma states that algebroid lifts ${s}^{\langle \alpha \rangle_{\kappa^k}}$ and ${s}^{\langle \alpha \rangle_{\kappa^j}}$ obtained using $\kappa^k$ and $\kappa^j$, respectively, are compatible in some natural sense.
\begin{lem}\label{l:comp_alg_lifts} Let $s\in \Sec(E^1)$ and $0\leq \alpha\leq j<k$. Then the vector field ${s}^{\langle  {\new - }\alpha \rangle_{\kappa^k}} \in \VF(E^k)$ is projectable onto $E^j$ and its projection is ${s}^{\langle -\alpha \rangle_{\kappa^j}}$.
\end{lem}
\begin{proof}
{We shall use the construction of $(\K)$-lifts of a section $s$, as defined  in \eqref{df:alpha_lifts}. We have}
$${s}^{\ideal{- \alpha}_{\kappa^k}} =  \frac{k!}{(k-\alpha)!} \, \ZMmap{\kappa^k}\left( \ZMmap{\VZM}^k_{k-\alpha}(\T^{k-\alpha}s)\right) = \ZMmap{\kappa^k}(\xi^k_\alpha),$$
where the section $\xi^k_\alpha : \T^k M\ra \T^k E^1$  is defined by $\tclass{k}{\gamma} \mapsto \tclass{k}{t\mapsto  t^\alpha \cdot s(\gamma(t))}$ where $\gamma$ is a curve in $M$.
Hence, the vector field ${s}^{\ideal{-\alpha}_{\kappa^k}}$ is the composition of maps $\id_{E^k} \times (\xi^k_\alpha \circ \sharp^k): E^k \ra E^k\times_{\T^k M} \T^k E^1$ with the VB morphism $(\kappa^k)^{!}: E^k\times_{\T^k M} \T^k E^1 \ra \T E^k$ induced by $\kappa^k$. The thesis follows from the commutativity of the  diagram
$$
\xymatrix{E^k \ar[d]^{\sigma^k_j} \ar[rr]^-{\id_{E^k} \times (\xi^k_\alpha \circ \sharp^k)} &&  E^k\times_{\T^k M}\T^k E^1 \ar[d] \ar[rr]^-{(\kappa^k)^{!}} && \T E^k \ar[d]^{\T \sigma^k_j} \\
E^j  \ar[rr]^-{\id_{E^j} \times (\xi^j_\alpha \circ \sharp^j)} &&  E^j\times_{\T^j M}\T^j E^1  \ar[rr]^-{(\kappa^j)^{!}} && \T E^j
}
$$
\end{proof}
}


\subsection{Prolongations of an almost Lie algebroid}  Let $\G$ be a Lie groupoid with source and target maps denoted by $\alpha, \beta: \G\ra M$, respectively.  We consider a foliation $\G^\alpha$ on $\G$ defined by $\alpha$-fibers $\G_x^\alpha = \{g\in \G: \alpha(g) =x\}$, the distribution $\T \G^\alpha \subset \T \G$ tangent to the leaves of $\G^\alpha$, related objects like $\T^k \G^\alpha$   and the right action of $\G$ on itself, $R_g: h\mapsto h g$ where $h\in \G^\alpha_{\beta(g)}$.
The Lie algebroid of $\G$ is usually defined as the vector bundle $\sigma: \cA(\G):= \T_M \G^\alpha \ra M$ equipped with a map $\sharp: \cA(\G)\ra \T M$ called the anchor, defined as $\sharp = \T \beta|_{\cA(\G)}$ and the Lie bracket on $\Sec(\cA(\G))$ inherited from the Lie bracket of right-invariant vector fields on $\G$. (Such vector fields are  in a one-to-one correspondence with sections of $\sigma$.) {Another, yet equivalent } construction of the Lie algebroid structure on $\cA(\G)$, is provided by the \emph{reduction map} $\RR^1$,
\begin{equation}\label{eqn:A_1}
\xymatrix{
\T\G^\alpha\ar[rr]^{\RR^1}\ar[d]^{\tau_\G} && \cA(G)\ar[d]^\tau\\
\G\ar[rr]^\beta && M,}\end{equation}
which is a fiber-wise VB isomorphism  obtained from the collection of maps $\T R_{g^{-1}}: \T_g\G^\alpha \ra \T_{\beta(g)} \G^\alpha$. The Lie algebroid structure on $\cA(\G)$ is defined by means of the VB comorphism $\kappa: \T \cA(\G) \relto \T\cA(\G)$ which is obtained as the reduction  of $\kappa_\G: \T \T \G \ra \T \T \G$.  The advantage of the latter over the standard construction of  the Lie functor is that it can be easily generalized to higher orders. This is obtained by means of the higher-order reduction map $\RR^k: \T^k\G^\alpha \ra \T^k_M \G^\alpha$, defined analogously to $\RR^1$,  by the collection of maps  $\T^k R_{g^{-1}}: \T^k_g\G^\alpha \ra \T^k_{\beta(g)} \G^\alpha$.
\begin{df}
\cite[Definition 3.3, Lemma 3.4]{MJ_MR_models_high_alg_2015} The $\thh{k}$-order Lie algebroid of a Lie groupoid $\G$ is the \grB\ $\cA^k(\G) := \T^k_M \G^\alpha$ together with a \VBC $\kappa^k :=  (\T^k \RR^1, \T \RR^k)(\wt{\kappa}^k_{\G})$ where   $\wt{\kappa}^k_\G$ is the restriction of $\kappa^k_\G$ to $(\T^k\T) \G^\alpha\times (\T\T^k) \G^\alpha$ subject to the natural inclusions $(\T^k \T) \G^\alpha \subset \T^k (\T \G^\alpha)$ and $(\T\T^k) \G^\alpha \subset \T (\T^k \G^\alpha)$.
\end{df}
Actually, $(\cA^k(\G), \kappa^k)$ is a  Lie HA in the sense of Definition \ref{df:higher_algebroid} (\cite[Proposition~4.13]{MJ_MR_HA_comorph_2018} and \cite[Section 5]{MJ_MR_models_high_alg_2015}).

In \cite{MJ_MR_models_high_alg_2015} we introduced a slightly bigger class of examples of HAs obtained by means of the construction called \emph{the prolongation} of an  almost Lie algebroid $(A, \kappa)$.  We will  outline this construction, highlighting a possible more general context for certain constructions.

A pair of a vector bundle $\sigma: A\ra M$ and a VB morphism $\sharp: A\to \T M$ covering the identity $\id_M$ is called an \emph{anchored vector bundle}. A curve $a:\R\ra A$ is called \emph{admissible} if the tangent lift of the curve $\und{a} = \sigma\circ a: \R\ra M$ coincides with the curve  $\sharp \circ a$, i.e.,
$\jet{}\und{a} = \sharp \circ a$.  {\newMR The subset $\At{k}$ of  $\T^{k-1} A$, defined as
\begin{equation} \label{df:A[k]}
    \At{k} = \{\tclass{k-1}{a} \mid  a: \R \to A \text{ is an  admissible curve }\},
\end{equation}
is called the \emph{$\thh{k}$-order prolongation of the anchored vector bundle  $A$} (see \cite{P04}). According to \cite[Theorem~2.2.7]{BGG_2015}, we have
\begin{equation}\label{e:A[2]}
    \At{2} = \{X \in \T  A: (\T \sigma) X = \sharp \, \tau_A(X)\},
\end{equation}
and
$$
    \At{k} = (\T^{k-1} \sharp)^{-1}(\T^k M)
$$
} {\newMR
where $\T^k M$ is considered as a subset  of $\T^{k-1} \T M$ via $\iM{k-1, 1}$.
It follows that the constructions of the sets $\At{k}$ here and $E^k$ in \cite[Definition~4.1]{MJ_MR_models_high_alg_2015} are equivalent, see also \cite{BGG_2015} or \cite[Theorem~4.5 (viii)]{MJ_MR_models_high_alg_2015}, or \cite{Martinez_2015}. In particular,   $\At{k+1} = \T \At{k} \cap \T^k A$, considered as subsets of $\T \T^{k-1} A$.

Define $\thh{k}$-order anchor map $\sharp^{[k]}: \At{k}\ra \T^k M$ as $\sharp^{[k]} = (\T^{k-1} \sharp)|_{\At{k}}$. It is a \grB\ morphism.

An AL algebroid structure on the vector bundle $A$ can be prolonged to a  HA structure on $\At{k}$ by means of the comorphisms $\kappa^{[k]}: \T^k A \relto \T \At{k}$, covering $\sharp^{[k]}$, defined as
\begin{equation} \label{df:kappa[k]}
\kappa^{[k]} = (\kappa_A^{k-1} \circ \T^{k-1} \kappa) \cap (\T^k A \cap \T \At{k}),
\end{equation}
see \cite[Proposition~4.6]{MJ_MR_models_high_alg_2015}. The comorphism $\kappa^{[k]}$ can also be defined inductively as it is presented in \cite[Definition~4.2]{MJ_MR_models_high_alg_2015}.
\commentMR{I deleted jet-style definition of $\kappa^{[k]}$ as we do not use it in the work.}
}

\subsection{Canonical inclusions II}  Here, we highlight some natural embeddings induced by the anchored bundle structure on a vector bundle $A\ra M$.

 In addition to the  inclusion $ \At{k} \subseteq \T^{k-1} A$ from the definition of $\At{k}$ \eqref{df:A[k]}, there are inclusions $\iA{k, l}: \At{k+l} \to \T^k \At{l}$ defined by the restriction of {\new $\iMM{k, l-1}{A}$ to $\At{k+l}$: }
 $$
 \xymatrix{\T^{k+l-1}A \ar[rr]^{\iMM{k, l-1}{A}} && \T^k \T^{l-1} A \\
 \At{k+l} \ar@{^{(}->}[u] \ar[rr]^{\iA{k, l}} && \T^k \At{l} \ar@{^{(}->}[u]
 }
 $$
 We should prove that the image $\iAA{k, l}{A}(\At{k+l})$ is in $\T^k \At{l}$, considered as a subset of $\T^k \T^{l-1} A$. Let $\tclass{k+l-1}{a} \in \At{k+l}$ where $a$ is an admissible path in $A$. Then $\iAA{k, l}{A}(\tclass{k+l-1}{a}) =  \jet{k}_{t=0} \jet{l-1}_{s=0} a(t+s)$. For any $t\in \R$, the path $s\mapsto a(t+s)$ is admissible,  hence the curve $t\mapsto \jet{l-1}_{s=0} a(t+s)$ lies in $\At{l-1}$, so $\iAA{k, l}{A}(\tclass{k+l-1}{a}) \in \T^k \At{l-1}$ as we claimed. {\new Using \eqref{e:iMkl_coord} we find that
 \begin{equation}  \label{e:iAkl-coord}
   (\iA{k, l})^\ast (y^{i,(\K, \KL)}) = y^{i, (\K+\KL)}
 \end{equation}
where $\alpha \leq k$, $\KL\leq l-1$. Recall,  $(x^a, y^{i, (\K)})$, $0\leq \K\leq k-1$ is a  coordinate chart for $\At{k}$ induced from $\T^{k-1} A$.}

 The rank of the \grB\ $\At{k}$ is $(r, r, r, \ldots, r)$ where $r = \rank A$.    Since  $\iA{k-1, 1}$ is an inclusion and the ranks of the VBs $\core{\At{k}}$ and  $\core{\T^{k-1} A}$ are the same, it induces an isomorphism $\core{\iA{k-1, 1}}: \core{\At{k}} \ra \core{\T^{k-1} A}$ of the  core bundles. We define an  isomorphism $\jAc{k}: A \xrightarrow{\simeq} \core{\At{k}} \subset \At{k}$ using the diagram
\begin{equation}\label{df:jAc}
\xymatrix{
A  \ar@{-->}[r]^{\jAc{k}} \ar[rd]^{\jEE{k}{A}} & \core{\At{k}} \subset \At{k} \ar[d]^{\new \iA{k-1, 1}} \\
& \core{\T^{k-1} A} \subset \T^{k-1} A.
}
\end{equation}
i.e.,  {$\iA{k-1, 1} \circ \jAc{k}: A \ra \T^{k-1} A$} coincides with $\jEE{k}{A}: A \xrightarrow{\simeq} \core{\T^{k-1} A} \subset \T^{k-1} A$. In the special case $A=\T M$, the map $\jAAc{k}{\T M}$ coincides with $\jM{k}: \T M\xrightarrow{\simeq} \core{\T^k M}$, due to \eqref{e:comp_jMjTM}.

 The following statement concerns the structure of the prolongation of an AL algebroid:

\begin{lem}\label{l:VF-k_and_core} Let $(A, \kappa)$ be an AL algebroid. The following diagram of {\new isomorphisms } is commutative
$$
    \xymatrix{\Sec(A) \ar[rr]^{\jAc{k}} \ar[rrd]_{s\mapsto \frac{1}{k!} \aliftB{s}{-k}}&& \Sec(\core{\At{k}}) \ar[d]^{v\mapsto \coreVF{v}} \\
    && \VF_{-k}(\At{k}).
    }
$$
In particular, for $X\in\VF(M)$ we have a commutative diagram
$$
    \xymatrix{ \VF(M) \ar[rr]^{\jM{k}} \ar[rrd]_{X\mapsto {\new \frac{1}{k!}} \aliftB{X}{-k}}&& \Sec(\core{\T^k M}) \ar[d] \\
    && \VF_{-k}(\T^k M),
    }
$$
where $\aliftB{X}{-k}$ is  the algebroid $(-k)$-lift of the vector field $X$,  associated with the HA $(\T^k M, \kappa^k_M)$. { Moreover, the  core of the anchor map $\sharp^{[k]}: \At{k} \ra \T^k M$ can be identified with $\sharp$ under the isomorphisms $\jAc{k}: A\ra \core{\At{k}}$ and $\jM{k}: \T M \ra \core{\T^k M}$.}
\end{lem}
\begin{proof}
In view of \eqref{e:fs_k}, it suffices to check that the first diagram is commutative for  sections  from the local frame $(e_i)$ of $\Sec(A)$. Due to the definition of $\jAc{k}$, this problem  reduces to verifying  that the vector fields $\frac{1}{k!} \aliftB{s}{-k} \in \VF_{-k}(\At{k})$ and $\coreVF{\left(\jEE{k}{A}(s)\right)} \in \VF_{-k}(\T^{k-1} A)$ are $\iA{k-1, 1}$-related, where we can take $s=e_i\in \Sec(A)$. From \eqref{df:jEk}, we see that $\coreVF{\left(\jEE{k}{A}(e_i)\right)} =\pa_{y^{i, (k-1)}}$. {\newMR The vector field $\aliftB{s}{-k}$ denotes the algebroid lift of $s$ with respect to $(\At{k}, \kappa^{[k]})$, the $\thh{k}$-order prolongation of the algebroid $(A, \kappa)$. From the definition of algebroid lifts and \eqref{e:e_i_lifts} we obtain
    $$
        \frac{1}{k!} \aliftB{e_i}{-k} = \ZMmap{\kappa^{[k]}}(e_i^{(0)}) = \ZMmap{\kappa^{[k]}}(\ZMmap{\VZM^k_0} e_i) = \pa_{y^{i, (k-1)}},
$$
The last equality follows from the fact that $\kappa$ is the identity on the core bundle; hence, the same holds for $\T^{k-1} \kappa$, as well as for $\kappa_A^{k-1}: \T^{k-1} \T A \to \T \T^{k-1} A$ and $\kappa^{[k]}$.
}

For the last statement, concerning the case $A=\T M$,  note that the inclusions $\At{k}\to \T^{k-1} A$ and $\T^k M \ra \T^{k-1} \T M$ induce the identity on the  cores.
Hence, $\core{\sharp^{[k]}}$ coincides with $\core{\T^{k-1} \sharp}$, which can be identified with $\sharp: A\ra \T M$, as  claimed.
\end{proof}

\section{Structure of higher algebroids}\label{sec:str}
In this section $(E^k, \kappa^k)$ is a HA of order $k$ and $(A = E^1, \kappa= \kappa^1)$ is its reduction to order one.

\subsection{Morphism \texorpdfstring{$\Rk^k:  \At{k}\ra E^k$}{Theta}.}
We shall construct a canonical VB morphism from $\thh{k}$-order prolongation $\At{k}$ of an AL algebroid $A$ (see  Preliminaries) to a given $\thh{k}$-order HA $(E^k, \kappa^k)$ whose order-one reduction coincides with $A$.

\begin{df}\label{df:Rk} Let $(E^k, \kappa^k)$ be a HA of order $k$. {\new We apply } the functor $\LL$ to the relation $\kappa^k\in\catGB[k,1]$ and define the relation $\Rk^k$ to be the intersection of $\At{k}\times E^k$ with $\LL(\kappa^k)$ subject to the natural inclusions and isomorphisms: $\iA{k-1, 1}: \At{k} \hookrightarrow  \T^{k-1}A \simeq \LL(\T^k A)$, and $\diag^k: E^k \hookrightarrow  \pLinr(E^k) = \LL(\T E^k)$ (defined in Preliminaries):
$$
\xymatrix{
\T^{k-1} A \simeq \LL(\T^k A) \ar@{-|>}[rr]^{\LL(\kappa^k)} && \LL(\T E^k) {=} \pLinr(E^k) \\
\At{k} \ar@{^{(}->}[u] \ar@{--|>}[rr]^{\Rk^k} && E^k \ar@{^{(}->}[u]
}
$$
\end{df}

\begin{thm}\label{thm:Rk} Let $(\sigma^k: E^k\ra M, \kappa^k)$ be an AL HA and let $(A, \kappa)$ be its order-one reduction. Then
\begin{enumerate}[(a)]
\item \label{it:Rk_GBk} $\Rk^k$ is (the graph of) a $\catGB[k]$-morphism, $\Rk^k: \At{k} \ra E^k$,
\item \label{it:Rk_rho} $\Rk^k$ intertwines the anchor morphisms: {\newMR $\sharp^k\circ \Rk^k = \sharp^{[k]}$. }
\end{enumerate}
\end{thm}

\begin{proof}
First, we shall prove that if $(U, V)\in \Rk^k$ then $\sharp^{[k]}(U) = \sharp^k(V)$. Then we shall show that $\Rk^k$ is a mapping, and this  will complete the proof of \eqref{it:Rk_GBk} and \eqref{it:Rk_rho}. Indeed, in vie of the characterisation of \grB\ morphisms \cite{JG_MR_gr_bund_hgm_str_2011},  we only need to add that the relation $\Rk^k$ is invariant with respect to  the homogeneity structure  on $\At{k}\times E^k$. This is because  the inclusions $\At{k} \hookrightarrow \LL(\T^k A)$, $E^k\hookrightarrow \LL(\T E^k)$ are \grB\ morphisms.

Take $(U, V) \in \Rk^k$, let us denote by $\tilde{U}$ (resp., $\tilde{V}$) the images of $U$ (resp. $V$) in $\T^{k-1} A$ (resp., $\pLinr(E^k)$) and consider the diagram
$$
\xymatrix{
& \tilde{U}\in \T^{k-1} A \ar[ddr]^(.3){\T^{k-1}\sharp^1}  \ar@{-|>}[rrr]^{\LL(\kappa^k)} &&& \pLinr(E^k)\ni \tilde{V} \ar[ddl]_(.3){\pLinr(\sharp^k)} \\
U\in \At{k} \ar@{^{(}->}[ur] \ar[ddr]^{\sharp^{[k]}}  \ar@{--|>}[rrrrr]^{\Rk^k} &&&&& E^k\ni V \ar@{_{(}->}[ul] \ar[ddl]_{\sharp^k} \\
&& \T^{k-1}\T M \ar[r]^{\kappa_M^{k-1}} & \T\T^{k-1} M && \\
& \T^k M \ar@{^{(}->}[ur]^{\iM{k-1, 1}} \ar[rrr]^{=} &&& \T^k M \ar@{_{(}->}[ul]_{\iM{1, k-1}}
}
$$
\begin{itemize}
\item {\new The top trapezoid in the middle commutes (i.e., $(\T^{k-1}\sharp^1, \pLinr(\sharp^k)): \LL(\kappa^k) \Rightarrow \kappa_M^{k-1}$ is a morphism in the category $\catZM$)  because it is obtained by applying the functor $\LL$ to the diagram \eqref{diag:AL}, which is commutative since $(E^k, \kappa^k)$ is almost Lie.   Here we used Lemmas~\ref{l:LL} and~\ref{lem:lambda_of_ZM}. }
\item The parallelogram on the left also commutes as $\sharp^{[k]}= \T^{k-1}\sharp^1|_{\At{k}}$, see \cite[Theorem~4.5~(ix)]{MJ_MR_models_high_alg_2015}.
\item The parallelogram on the right also commutes. This follows from a more general
Lemma~\ref{l:plin_incl}.
\end{itemize}
As $(U, V) \in \Rk^k$, so $(\tilde{U}, \tilde{V}) \in  \LL(\kappa^k)$, hence $\T^{k-1}\sharp^1(\tilde{U})\in \T^{k-1}\T M$ and  $\pLinr(\sharp^k)(\tilde{V})\in \T \T^{k-1} M$ are related by means of $\kappa_M^{k-1}$. However, due to the commutativity of the left and right parallelograms, both $\T^{k-1}\sharp^1(\tilde{U})$ and  $\pLinr(\sharp^k)(\tilde{V})$ are images of $\sharp^{[k]}(U)$ and $\sharp^k(V)$, respectively, under the canonical inclusions of $\T^k M$ into $\T^{k-1}\T M$ and $\T \T^{k-1} M$, respectively. Moreover, these images are $\kappa_M^{k-1}$-related, due to the commutativity of the top trapezoid. Since $\kappa^{k-1}_M$ intertwines the canonical inclusions, $\kappa_M^{k-1} \circ \iM{k-1, 1} = \iM{1, k-1}$,
  we get $\sharp^{[k]}(U)  =\sharp^k(V)$, as was claimed.

 Now we shall prove \eqref{it:Rk_GBk}, i.e.,  that the relation $\Rk^k$ is a  mapping. We shall proceed by induction on $k$.

 Obviously, $\Rk^1: A \ra A$ is the identity mapping.
Let $k>1$ and assume that $\Rk^{k-1}: A^{[k-1]} \ra E^{k-1}$ is a mapping. The graph of $\Rk^{k-1}$ is  invariant with respect to the  homogeneity structure of $A^{[k-1]}\times E^{k-1}$, hence $\Rk^{k-1}$ is   a  morphism of \grBs.

\noindent \emph{Step A. } We shall fist prove that for any $U^k \in \At{k}$ there is at least one $V^k\in E^k$ such that $(U^k, V^k)\in \Rk^k$.

 We know  from Lemma~\ref{lem:lambda_of_ZM} that $\LL(\kappa^k)$ is a VB comorphism covering $\sharp^{k-1}: E^{k-1}\ra \T^{k-1}M$. Set $\tilde{V}^k = \LL(\kappa^k)_v(\tilde{U}^k) \in \pLinr(E^k)$, where $v := \Rk^{k-1}(U^{k-1}) \in E^{k-1}$ and $U^{k-1} = \sigma^k_{k-1} (U^k)\in A^{[k-1]}$. Consider  the diagram
$$
\xymatrix{
\tilde{U}^k\in \T^{k-1} A \ar[dd]_{\T^{k-1}\sigma^1}  \ar@{-|>}[rrr]^{\LL(\kappa^k)} &&& \pLinr(E^k)\ni \tilde{V}^k \ar[dd] \\
& U^k \in \At{k} \ar[d] \ar@{}[l]|-{\circlearrowleft}  \ar@{_{(}->}[ul] \ar@{--|>}[r]^{\Rk^k} & E^k \ar[d] \ar@{^{(}->}[ur] & \\
\T^{k-1} M & \At{k-1} \ar[l]_{\sharp^{[k-1]}} \ar[r]^{\Rk^{k-1}} & E^{k-1} \ar[r]^{=} & E^{k-1}\ni v \ar@/^1pc/[lll]^{\sharp^{k-1}}
}
$$
We shall check first that the definition of $\tilde{V}^k$ is correct, i.e.,
\begin{equation}\label{eqn:thm_Rk_1}
\sharp^{k-1}(v) = \T^{k-1}\sigma^1(\tilde{U}^k).
\end{equation}
This amounts to show that the compositions $\At{k} \to \At{k-1} \xrightarrow{\Rk^{k-1}} E^{k-1} \xrightarrow{\sharp^{k-1}} \T^{k-1} M$ and  $\At{k} \xhookrightarrow{} \T^{k-1} A \to \T^{k-1} M$ coincide.
According to our inductive hypothesis, $\sharp^{k-1}\circ \Rk^{k-1}$ is equal  to $\sharp^{[k-1]}$, hence \eqref{eqn:thm_Rk_1} reduces to the commutativity of the square diagram on the left (pointed by the circular arrow $\circlearrowleft$).
{ The map $\sharp^{[k-1]}$ is the restriction of $\T^{k-2} \sharp$ to $\At{k-1}\subset \T^{k-2} A$, see \cite[Theorem~4.5~(ix)]{MJ_MR_models_high_alg_2015}, hence it suffices to prove that the following diagram  is commutative. }
$$
\xymatrix{
\At{k} \ar@{^{(}->}[d] \ar[r] & \At{k-1} \ar@{^{(}->}[r] & \T^{k-2} A \ar[d]^{\T^{k-2} \sharp}\\
\T^{k-1} A \ar[r]^{\T^{k-1} \sigma}  & \T^{k-1} M \ar@{^{(}->}[r] & \T^{k-2}\T M
}
$$
The inclusions $\At{k} \xhookrightarrow{} \T^{k-1} A$ are compatible with projections $\T^{k+1} A \ra \T^k A$, i.e.,  the following diagram on the left is commutative:
$$
   \xymatrix{
\At{k} \ar@{^{(}->}[d] \ar[r] & \At{k-1} \ar@{^{(}->}[d] \\
\T^{k-1} A \ar[r] & \T^{k-2} A,
}
\quad \quad \quad
\xymatrix{
\T^{k-1} A \ar[d]^{\T^{k-1}\sigma} \ar[r] & \T^{k-2} A \ar[d]^{\T^{k-2} \sharp}\\
\T^{k-1} M \ar[r] & \T^{k-2}\T M
}
$$
The diagram on the right is not commutative in general, however for $X \in  \At{k} \subset \T^{k-1} A$ we have (by \cite[Theorem~4.5~(viii)]{MJ_MR_models_high_alg_2015}):
$ \iM{k-2, 1} \circ(\T^{k-1}\sigma) (X)  = (\T^{k-2} \sharp)\circ \tau^{k-1}_{k-2}(X)$. This is enough for our claim \eqref{eqn:thm_Rk_1}.

Now we prove that $\tilde{V}^k\in\pLinr(E^k)$ is   in $E^k\subset \pLinr(E^k)$. Consider the diagram
$$
\xymatrix{
E^k \ar@{^{(}->}[r]^(.3){\diag^k}  & \pLinr(E^k)\ni \tilde{V}^k\ar[d]^\pi \ar[dr]^{\pLinr(\sigma^k_{k-1})}
&  \\
& v\in E^{k-1} \ar@{^{(}->}[r]^{\diag^{k-1}} & \pLinr(E^{k-1})
}
$$
The subset of $\pLinr(E^k)$ of those elements $X$ for which $\pLinr(\sigma^k_{k-1})(X) = \diag^{k-1} (\pi(X))$  coincides with $E^k$.  We are given $v = \Rk^{k-1}(U^{k-1}) \in E^{k-1}$ and $\tilde{V}^k\in \pLinr(E^k)$ such that $\pi(\tilde{V}^k) = v$ and $\diag^{k-1}(v) = \pLinr(\sigma^k_{k-1})(\tilde{V}^k)$. It follows that $\tilde{V}^k \in E^k$.

\noindent \emph{Step B.} We shall prove that for a given $U^k\in \At{k}$ there is at most one $V^k\in E^k$ such that $(U^k, V^k)\in \Rk^k$. This will finish the proof of \eqref{it:Rk_GBk} and \eqref{it:Rk_rho}.

Assume $(U^k, V^k_i)$  are in $\Rk^k$ for $i=1,2$. Using the inductive hypothesis, we know that $V^{k-1}_1 = V^{k-1}_2 \in E^{k-1}$, hence $V^k_i$, considered as elements of $\pLinr(E^k)$, are in the same fiber of the vector bundle $\pi: \pLinr(E^k)\ra E^{k-1}$. As $(U^k, V^k_i)\in \LL(\kappa^k)$ and $\LL(\kappa^k)$ is a VB comorphism over $\sharp^{k-1}: E^{k-1}\ra \T^{k-1} M$, we must have $V^k_1 = \LL(\kappa^k)_v(U^k) = V^k_2$ as we claimed.
\end{proof}

\begin{ex}[$\Rk^2$ and $\Rk^3$ in coordinates] \label{ex:Theta2_and_3}
We shall provide an explicit coordinate expression for $\Rk^2: \At{2}\ra E^2$, assuming $\kappa^2$ is given a general  local form as in  \eqref{e:local_kappa2}. 
According to the procedure given in the Definition~\ref{df:Lambda}, to get $\LL(\kappa^2)$, in  the first step we set $y^i=0$ and $\und{x}^a=0$.  Then we eliminate coordinates of weight $(2,0)$, i.e.,   the coordinates $\ddot{x}^a$ and $\und{z}^\mu$. In this way  we arrive at { the VB comorphism $\LL(\kappa^2): \LL(\T^2 A) \relto \pLinr(E^2)$ over $\sharp: A\ra \T M$ given by}
\begin{equation}\label{e:LL_kappa2}
\LL(\kappa^2):
\begin{cases}
\dot{x}^a = & Q^a_i \,\und{y}^i, \\
\dot{\und{y}}^i = & \dot{y}^i,\\
\dot{\und{z}}^\mu = & Q^\mu_i\,\ddot{y}^i + Q^\mu_{ij} \, \und{y}^i \dot{y}^j.
\end{cases}
\end{equation}
Let $(U^2, V^2)\in \At{2}\times E^2$, 
and let $(\tilde{U}^2, \tilde{V}^2 ) \in \LL(\T^2 A) \times \pLinr(E^2)$ be the image of $(U^2, V^2)$ under
 the canonical inclusions (see \eqref{e:iAkl-coord}, \eqref{e:IkE-coord}, \eqref{e:diag}):
$$
I^2_E \circ \iA{1,1}: \At{2} \xhookrightarrow{} \T A  \simeq \LL(\T^2 A), \quad
(x^a, \dot{x}^a, \dot{y}^i, \ddot{y}^i)(\tilde{U}^2) = (x^a, Q^a_i y^i, y^i, 2 \dot{y}^i)(U^2)
$$
and
$$
\diag^2: E^2 \hookrightarrow \pLinr(E^2),\quad
(\und{x}^a, \und{y}^i,  \und{\dot{y}}^i, \und{\dot{z}}^\mu)(\tilde{V}^2) = (\und{x}^a, \und{y}^i, \und{y}^i, 2 \und{z}^\mu)(V^2).$$
(Recall that $(x^{a, (\K)}, y^{i, (\KL)})$, $0\leq \K \leq k-1$, $1\leq \KL \leq k$, are coordinates for $\LL(\T^k A)$ inherited from the adapted coordinates on  {\new $\T^k A$}. The coordinate system for $\pLinr(E^k) = \LL(\T E^k)$  is inherited from the adapted coordinate system on the tangent bundle of $E^k$.)
{\new By plugging these expressions to \eqref{e:LL_kappa2} we find that
$(\tilde{U}^2, \tilde{V}^2 ) \in \LL(\kappa^2)$ if and only if
$$
\begin{cases}
Q^a_i y^i &=     Q^a_i \,\und{y}^i, \\
\und{y}^i &=  y^i,\\
2 \und{z}^\mu &=  Q^\mu_i\, 2 \dot{y}^i + Q^\mu_{ij} \, \und{y}^i y^j,
\end{cases}
$$
hence}
$\Rk^2$
is an affine bundle morphism $\Rk^2: \At{2} \ra E^2$ covering the identity $\id_{A}$ given by
\begin{equation}\label{e:R2_coord}
\Rk^2(x^a, y^i, \dot{y}^i) =  (\und{x}^a= x^a, \und{y}^i = y^i, \und{z}^\mu = Q^\mu_i \dot{y}^i + \frac12 Q^\mu_{(ij)} y^i y^j
),
\end{equation}
where
$Q^\mu_{ij} = Q^\mu_{(ij)} + Q^\mu_{[ij]}$ is the decomposition into symmetric and anty-symmetric part, namely
\begin{equation} \label{d:Q_symm_and_Q_asym}
Q^\mu_{(ij)} = \frac12 Q^\mu_{ij} + \frac12 Q^\mu_{ji}, \quad Q^\mu_{[ij]} = \frac12 Q^\mu_{ij} - \frac12 Q^\mu_{ji}.
\end{equation}
In order three,  additional equations for $\kappa^3$ appear. Let $(x^a, y^i, z^\mu, t^\alpha)$ be graded coordinates on a \grB\ $\sigma^3: E^3\ra M$ where the coordinates $t^\alpha$ have order 3. The additional equations for $\kappa^3$, extending those for $\kappa^2$,  are of the form (we have omitted { expressions } that do not account for $\LL(\kappa^3)$): $\dddot{x}^a = \dddot{x}^a(\und{x}, \und{y}, \und{z}, \und{t})$ and
$$
\dot{\und{t}}^\alpha =  Q^\alpha_i \dddot{y}^i + Q^\alpha_{ij} \und{y}^i \ddot{y}^j + Q^\alpha_{\nu i} \und{z}^\nu \dot{y}^i + \frac12 Q^\alpha_{ij, k} \und{y}^i\und{y}^j \dot{y}^k + { q^\alpha_i(\und{x}, \und{y}, \und{z}, \und{t}) y^i }
$$
for some functions $q^\alpha_i$ on $E^3$ of weight $3$.
Now we set to zero the coordinates of weight $(0, 1)$, i.e.,  $y^i=0$, $\und{\dot{x}}^a = 0$  and eliminate the coordinates of weight $(3, 0)$,  $\dddot{x}^a$ and $\und{t}^\alpha$. This way we get, in addition to \eqref{e:LL_kappa2}, the following  equations defining $\LL(\kappa^3)$:
\begin{equation}\label{e:LL_kappa3}
\LL(\kappa^3) : \begin{cases}
\ddot{x}^a &= Q^a_\mu \und{z}^\mu + \frac12 Q^a_{ij} \und{y}^i \und{y}^j, \\
\dot{\und{t}}^\alpha &=  Q^\alpha_i \dddot{y}^i + Q^\alpha_{ij} \und{y}^i \ddot{y}^j + Q^\alpha_{\nu i} \und{z}^\nu \dot{y}^i + \frac12 Q^\alpha_{ij, k} \und{y}^i\und{y}^j \dot{y}^k.
\end{cases}
\end{equation}
which is a VB comorphism over $\sharp^2$:
$$
    \xymatrix{
 A^{[3]} \subset \T^{2} A \simeq \LL(\T^3 A) \ar[d] \ar@{-|>}[rr]^{\LL(\kappa^3)} &&  \pLinr(E^3) \supset E^3  \ar[d] \\
\T^2 M  && E^2 \ar[ll]_{\sharp^2}
}
$$
{\new
Let $U^3\in \At{3}$, $V^3\in E^3$ and let $\tilde{U}^3$ in  $\LL(\T^3 A)$ and $\tilde{V}^3$ in $\pLinr(E^3)$ denote their images subject to the inclusions $\At{3} \xhookrightarrow{} \LL(\T^2 A)$ and $\diag: E^3 \xhookrightarrow{} \pLinr(E^3)$.  We have
\begin{equation}\label{e:U3V3}
\begin{split}
(x^a, \dot{x}^a, \ddot{x}^a; \dot{y}^i, \ddot{y}^i, \dddot{y}^i)(\tilde{U}^3) = (x^a,  Q^a_i y^i, X^a;  y^i, 2 \dot{y}^i, 3 \ddot{y}^i)(U^3)\\
(x^a, \und{y}^i, \und{z}^\mu; \und{\dot{y}}^i, \und{\dot{z}}^\mu,  \und{\dot{t}}^\alpha)(\tilde{V}^3) = (x^a, \und{y}^i, \und{z}^\mu, \und{y}^i, 2 \und{z}^\mu, 3 \und{t}^\alpha) (V^3).
\end{split}
\end{equation}
}
where, due to the definition of $\At{k}$,
$$
    X^a = \ddot{x}^a(\tilde{U}^3) =  (Q^a_i y^i)^{\cdot} = \frac{\pa Q^a_i}{\pa x^b} \dot{x}^b y^i + Q^a_i \dot{y}^i = \frac12 \wh{Q}^a_{ij} y^i y^j + Q^a_i \dot{y}^i,
$$
where
\begin{equation}\label{df:Qhat}
\Qhat{Q}^a_{ij} = \frac{\pa Q^a_i}{\pa x^b} Q^b_j +  \frac{\pa
Q^a_j}{\pa x^b} Q^b_i.
\end{equation}
Now assume that $\tilde{U}^3$, $\tilde{V}^3$  are $\LL(\kappa^3)$-related. We shall show that the first equation for $\LL(\kappa^3)$ in \eqref{e:LL_kappa3}
is satisfied automatically. It amounts to show that $\ddot{x}^a(\tilde{U}^3)$ given above coincides with
$$
    Q^a_\mu \und{z}^\mu + \frac12 Q^a_{ij} \und{y}^i \und{y}^j (\text{on } \tilde{V}^3) = Q^a_\mu \left(Q^\mu_i \dot{y}^i + \frac12 Q^\mu_{(ij)} y^i y^j\right) + \frac12 Q^a_{ij} y^i y^j (\text{on }\tilde{U}^3).
$$
The last equality is due to \eqref{e:R2_coord} as the reductions of $\tilde{U}^3$, $\tilde{V}^3$ to order 2 are $\LL(\kappa^2)$-related. By comparing the coefficients at $\dot{y}^i$ and at $y^iy^j$ we see that the first equation in \eqref{e:LL_kappa3} is equivalent to the equations \eqref{e:QamuQmu_i} and \eqref{e:QamuQmu_ij_symm} considered in Appendix, which are true in any order-two AL HA.
The second equation for $\LL(\kappa^3)$ gives
$$
    \und{t}^\alpha = \und{t}^\alpha(V^3) = \frac13 \und{\dot{t}}^\alpha(\tilde{V}^3) =  \frac13 Q^\alpha_i \dddot{y}^i(\tilde{U}^3) + \frac13 Q^\alpha_{ij} \und{y}^i(\tilde{V}^3) \ddot{y}^j (\tilde{U}^3) +  
    \frac 13 Q^\alpha_{\nu i} \und{z}^\nu (\tilde{V}^3) \dot{y}^i(\tilde{U}^3)  +
    \frac16 Q^\alpha_{ij, k} \und{y}^i \und{y}^j (\tilde{V}^3)\dot{y}^k (\tilde{U}^3)
$$
from which, using \eqref{e:U3V3}, we find a complete formula for $\Rk^3$,
\begin{equation}\label{e:R3_coord}
\begin{split}
\Rk^3:(x^a, y^i, \dot{y}^i, \ddot{y}^i) \mapsto (\und{x}^a= x^a, \und{y}^i = y^i, \und{z}^\mu =  \frac12 Q^\mu_{(ij)} y^i y^j, \\
\und{t}^\alpha = Q^\alpha_i \ddot{y}^i +  (\frac{2}{3} Q^\alpha_{ij} + \frac{1}{3} Q^\alpha_{\nu i} Q^\nu_{j}) y^i\dot{y}^j + \frac{1}{6} (Q^\alpha_{\nu i} Q^\nu_{jk} + Q^\alpha_{ij, k}) y^i y^j y^k).
\end{split}
\end{equation}
\end{ex}
{
\begin{rem}\label{r:R2_skew_alg} In deriving the formula for the mapping $\Rk^2$ we did not use the assumption from Theorem~\ref{thm:Rk}  that $(E^2, \kappa^2)$ is AL. Actually, $\Rk^2$ is a well defined mapping for any skew HA $(E^2, \kappa^2)$.
\end{rem}
}

\begin{conj}\label{conj:Rk} Let $(\sigma^k: E^k\ra M, \kappa^k)$ be a  Lie HA.  Then $\Rk^k: \At{k}\ra E^k$ is a HA morphism.
\end{conj}
 It suffices to prove that $(\T^k \Rk^1, \T \Rk^k): \kappa^{[k]}\mZM \kappa^k$ is a $\catZM$-morphism. Recall,  $\Rk^1$ is the identity on $E^1$, hence it remains to verify  that the diagram
\begin{equation}\label{diag:R_is_morphism_of_AHA}
\xymatrix{
\T^k A \ar@{-|>}[rr]^{\kappa^{k}} \ar[dd] \ar@{-|>}[rd]^{\kappa^{[k]}} && \T E^k \ar[dd] \\
& \T \At{k} \ar[ru]^{\T \Rk^k} \ar[dd] & \\
\T^k M && E^k \ar[ll]_{\sharp^k} \\
& \At{k} \ar[ul]^{\sharp^{[k]}} \ar[ur]^{\Rk^k} &
}
\end{equation}
is commutative. {\newMR We already know that the bottom triangle is commutative (see Theorem~\ref{thm:Rk}). Therefore, we now need to prove that
\begin{equation}\label{e:rho_and_R_on_fibers}
\left(\kappa^k\right)_{\und{Z}}(X) = \T \Rk^k (\left(\kappa^k\right)_{\und{Y}}(X))
\end{equation}
for any $\und{Y}\in \At{k}$ and $X\in\T^k A$ such that $\T^k \sigma^1 (X) = \sharp^{[k]}(\und{Y})$, {  where $\und{Z} = \Rk^k(\und{Y})\in E^k$.} In other words, this means that the algebroid $\langle\K\rangle$-lifts with respect to $(E^k, \kappa^k)$ and $(\At{k}, \kappa^{[k]})$ are $\Rk^k$ related for all $-k\leq \alpha
\leq 0$.
}
We shall prove Conjecture~\ref{conj:Rk}  for $k=2$  by direct computations. See Appendix, Subsection~\ref{sSec:AL_and_Lie_HA_eqns}. 

\subsection{Higher algebroids in order two}\label{sSec:HA_order_two}
In this subsection, we shall look closer at higher algebroids $(E^2, \kappa^2)$ of order two.  First, we shall describe the structure of the \grB\ $E^2$, see Lemma~\ref{l:structureE2}. Then,  we
shall derive a number of \emph{structure maps} which fully determine $(E^2, \kappa^2)$ and  reformulate the definition of a skew HA in terms of these structure maps and relations between them, see Theorem~\ref{th:skew_HA}. We shall examine skew and  Lie HAs in which the base $M$ is a point, see Theorem~\ref{thm:HA_point}.
We shall also find the relations between the structure maps and the conditions under which $(E^2, \kappa^2)$ becomes an almost Lie (Theorem~\ref{th:AL_HA_str_maps}) and a Lie HA (Theorem~\ref{th:Lie_axiom_maps}). Finally, we will describe the relation between order-two HAs and Lie algebroid representations up to homotopy, see Theorem~\ref{t:Rep_to_HA}.
We assume that  $(E^2, \kappa^2)$ is a \emph{skew} HA.

{\new  Throughout this subsection, we denote by $A$ the order-one reduction of $E^2$, i.e.,  $A= E^1$,  and  $C = \core{E^2}$ -- the  core of $E^2$. We set $(x^a, y^i, z^\mu)$ as  a system of graded coordinates on $E^2$ and fix  local frames $(e_i)$ of $\Sec(A)$ and $(c_\mu)$ of $\Sec(C)$ such that  \begin{equation}\label{e:loc_frames}
    y^i(e_j) = \delta^i_j,  \quad z^\mu|_C(c_\nu)=\delta^\mu_\nu.
\end{equation}
}

\subsubsection{The structure of the \grB\ of a skew HA of order two}

 Define \begin{equation}\label{df:pa0}
 \pa := \core{\Rk^2}: A  \ra C
  \end{equation}
   as the  core of the map $\Rk^2: \At{2} \ra E^2$, see Definition~\ref{df:Rk}, where the  core of $\At{2}$ is identified with $A$ under  the isomorphism $\jAc{2}: A \simeq \core{\At{2}}$, see \eqref{df:jAc}. .
Consider the map $\tilde{\Rk}^2$ defined on the product (over $M$) of \grBs\ $\At{2}$ and $C_{[2]}$ by
\begin{equation}\label{e:Rk_tilde}
    \tilde{\Rk}^2: \At{2}\times_M C_{[2]} \to E^2, \quad (a, c) \mapsto \Rk^2(a)  \plus c
\end{equation}
(We recall that $F_{[k]}$ stands for the \grB\ defined on the total space of the VB $F\ra M$ by assigning weight $k$ to linear functions on $F$.)
The map $\tilde{\Rk}^2$ is a surjective morphism of { \grBs, } hence $E^2$ can be identified as the quotient of  $ \At{2}\times_M C_{[2]}$ by  the  equivalence relations $\sim$, where
$$
    (a, c) \sim (a', c') \iff \Rk^2(a) \plus c = \Rk^2(a') \plus c'.
$$
{\new Since $\Rk^2$ covers the identity $\id_A$,  }  the elements $a$ and $a' \in \At{2}$  project to the same element in $A$. As $E^2\ra E^1$ is an affine bundle, we can write   $\Rk^2(a')\minus \Rk^2(a) = \core{\Rk^2}(a'-a) = c' - c \in C$.
In other words, $(a', c') \minus (a, c) = (a'-a, c'-c)$ is in the graph of the map $-\pa$,  which is a subset of   $ A\times_M C$, the  core of the \grB\  $\At{2}\times_M C_{[2]}$. Therefore, what we need to define $E^2$ is only the map $\pa$.
\begin{lem} \label{l:structureE2} Let $(E^2, \kappa^2)$ be a skew, order-two HA.
 \begin{enumerate}[(i)]
     \item \label{i:iso:th:structureE2} There is a canonical isomorphism of \grBs\
        $$
            E^2 \simeq \quotient{(\At{2}\times_M C_{[2]})}{\graphW(-\pa)}
        $$
        where $\pa:A\ra C$ is given in~\eqref{df:pa0}.
        \item \label{i:R2_and_quotient3} { A choice of local frames $(e_i)$ and $(c_{\mu})$ of $\Sec(A)$ and $\Sec(C)$, respectively, gives rise to a graded coordinate system $(x^a, y^i, w^\mu)$ for $E^2$ (considered as the  quotient of $\At{2}\times_M C_{[2]}$), defined by  $w^\mu = c_\mu^\ast + Q^\mu_i \dot{y}^i$. The composition $\At{2} \hookrightarrow \At{2}\times_M C_{[2]} \ra \quotient{(\At{2}\times_M C_{[2]})}{\graphW(-\pa)}$ coincides with the map~$\Rk^2$  which, in the introduced coordinates $(x^a, y^i, w^\mu)$, read as
            $$
                (\Rk^2)^\ast(x^a)= x^a, \quad (\Rk^2)^\ast(y^i)= y^i, \quad(\Rk^2)^\ast(w^\mu)= Q^\mu_i \dot{y}^i.
            $$}
      \end{enumerate}
\end{lem}
\begin{df}\label{df:adapted_coord} We call $(x^a, y^i, w^\mu)$  an \emph{adapted} system of graded coordinated on a HA $(E^2, \kappa^2)$. It is uniquely defined once we set a system of local frames $(e_i)$, $(c_\mu)$, and is characterised by the equality $Q^\mu_{(ij)} = 0$.
\end{df}
\begin{proof} Set $\tilde{E}^2 =  \At{2}\times_M C_{[2]}$, $V= \graphW(-\pa)$.
    We have already shown that there is  a well defined bijection between $E^2$ and the quotient
$\quotient{\tilde{E}^2}{\sim}$
{\new defined as the set of equivalence classes of the following equivalence relation: $e\sim e'$ if and only if there exists $v\in V$ such that $e' = e \plus v$. }  
    It is also evident that this bijection is an isomorphism of \grBs\ since it is a special case of the following more general construction: given a \grB\ $E^k$ of order $k$ and a vector subbundle $V\ra M$ of the  core  $\core{E^k}$,  the quotient $\quotient{E^k}{V}$ which is the orbit space of the action on $E^k$ of the subbundle $V \subset \core{E^k}$ of the  core,
    inherits a \grB\ structure from $E^k$.

    { Recall that $(y_i)$ and $(c_\mu^\ast)$ denote the dual frames to $(e_i)$ and $(c_\mu)$, respectively. Then  $(x^a, y^i, \dot{y}^i, c_\mu^\ast)$ forms a graded coordinate system on $\At{2}\times_M C_{[2]}$. The introduced equivalence relation on this space reads as: $(x^a, y^i, \dot{y}^i, c_\mu^\ast) \sim  (\und{x}^a, \und{y}^i, \dot{\und{y}}^i, \und{c}_\mu^\ast)$ if and only if $x^a =\und{x}^a$, $y^i = \und{y}^i$, $\und{c}_\mu^\ast - c_\mu^\ast = Q^\mu_i(\dot{y}^i - \dot{\und{y}}^i)$. } Therefore, the functions $w^\mu := c_\mu^\ast + Q^\mu_i \dot{y}^i$ are well-defined on the quotient $\quotient{\tilde{E}^2}{{\sim}}$, and $(x^a, y^i, w^\mu)$ is a graded coordinate system on this quotient.

    Let $(x^a, y^i, z^\mu)$ be as in \eqref{e:loc_frames}. The map $\tilde{\Rk}^2: \tilde{E}^2 \ra E^2$, defined in \eqref{e:Rk_tilde}, is given by
    $$
      \left(\tilde{\Rk}^2\right)^\ast (x^a) = x^a,   \left(\tilde{\Rk}^2\right)^\ast(y^i) = y^i, \left(\tilde{\Rk}^2\right)^\ast(z^\mu) = c_\mu^\ast + Q^\mu_i \dot{y}^i + \frac12 Q^\mu_{(ij)} y^i y^j = w^\mu +  \frac12 Q^\mu_{(ij)} y^i y^j.
    $$
    Hence, the isomorphism from point \eqref{i:iso:th:structureE2}, denoted by $I: \quotient{\tilde{E}^2}{{\sim}} \to E^2$,   is given by $I^\ast(z^\mu) = w^\mu + \frac12 Q^\mu_{(ij)} y^i y^j$, and the composition of the inclusion $\At{2} \hookrightarrow \At{2}\times_M C_{[2]}$ with $\tilde{\Rk}^2$ coincides with \eqref{e:R2_coord}, i.e., with the formula for $\Rk^2$, which proves the  claim  from  point  \eqref{i:R2_and_quotient3} and  completes the proof.
\end{proof}

\subsubsection{The structure maps of a skew HA of order two} \label{par:skew_HA-2}
The subspace $\VF_{-2}(E^2) \subset \VF(E^2)$ of vector fields of weight $-2$ is a locally free $\Cf(M)$-module canonically isomorphic to the space of sections of the  core bundle $C = \core{E^2}$. We shall often identify these spaces without further comment.  In coordinates as in \eqref{e:loc_frames}, the isomorphism takes $c_\mu$ to $\coreVF{c_\mu} = \pa_{z^\mu}$, see Lemma~\ref{l:VF-k}.

The  reduction of $\kappa^2$ to order one yields a skew algebroid,  whose structure maps will be denoted by $[\cdot, \cdot]$ and $\sharp :=\sharp^1: A \ra \T M$. We assume that $\kappa^2$ has a local form {\new introduced  in \eqref{e:local_kappa2}}.  Then
$$
    \sharp e_i = Q^a_i \pa_{x^a}, \quad [e_i, e_j] = Q^k_{ij} e_k,
$$
where $(e_i)$ is a  local frame of sections of $A\to M$ which is dual to the frame $(y^i)$. \commentMR{Spróbuj to napisać gdzieś odgórnie i nie powtarzać.}
The  core of  the anchor map $\sharp^2: E^2 \ra \T^2 M$ provides {\newMR a VB morphism }
\begin{equation} \label{df:sharpC}
\sharp^C: C\ra \T M,  \quad \sharp^C(c_\mu) = Q^a_\mu \pa_{x^a}.
\end{equation}
{\newMR In more detail, $\sharp^C$ is the composition of $\core{\sharp^2}: \core{E^2} \to \core{\T^2 M}$ with the isomorphism $\core{T^2 M} \simeq \T M$, see \eqref{e:jMk_coord}.}

The next mapping is a VB morphism $\tilde{\pa}: A\ra C$ defined by $\Cf(M)$-linear map
\begin{equation}\label{df:pa}
     s\mapsto   {\new \frac12 \aliftB{s}{-2}} \in \VF_{-2}(E^2)\simeq \Sec(C),
\end{equation}
 where $s\in\Sec(A)$, see \eqref{df:algebroid_lift} for algebroid lifts. (It will turn out soon that $\tilde{\pa} = \pa$, the core of the map $\Rk^2$.)
In a similar manner we define
\begin{equation}\label{df:beta}
\beta: \Sec(A)\times \Sec(A) \ra \Sec(C), \quad \beta(s_1, s_2) = {\new \frac12\, [\aliftB{s_1}{-1}, \aliftB{s_2}{-1}]} \in \VF_{-2}(E^2) \simeq \Sec(C),
\end{equation}
which is a skew-symmetric mapping and
\begin{equation}\label{df:gamma}
\Box: \Sec(A)\times \Sec(C)\ra \Sec(C),\quad  \Box_s v = [\aliftB{s}{0}, v]  \in \VF_{-2}(E^2)  \simeq \Gamma(C)
\end{equation}
  called \emph{the action of $A$ on $C$}.
 The system of equations \eqref{e:local_kappa2} describing $\kappa^2$ results in the following formulas for the algebroid lifts {\new $\aliftB{e_k}{\K}$:}
\begin{equation}\label{e:alg_lifts_coord}
\begin{cases}
\aliftB{e_k}{0} &= Q^a_k \pa_{x^a} + Q_{jk}^i y^j \pa_{y^i} + (Q^\mu_{\nu k} \, z^\nu + \frac{1}{2} \, Q_{ij, k}^\mu \, y^i y^j)\,\pa_{z^\mu}, \\
\aliftB{e_k}{-1} &=  \pa_{y^k} +  Q^\mu_{ik}\, y^i\,\pa_{z^\mu}, \\
\aliftB{e_k}{-2} &=  {\new 2} Q^\mu_k \, \pa_{z^\mu}.
\end{cases}
\end{equation}
\commentMR{Wyprowadzę ostatni wzór. According to \eqref{df:algebroid_lift} $\aliftB{s}{-2} = 2 \ZMmap{\kappa^2}(s^{(0)})$,  a $s^{(0)} = \ZMmap{\VZM}^2_0(s)$ \eqref{df:alpha_lifts}. Ze wzoru \eqref{e:e_i_lifts} odczytuję, że  $e_k^{(0)}$ ma
postać $y^i=\dot{y}^i = 0$, $\ddot{y}^i = {\new 1}$,  u nas jest $2\ZMmap{\kappa^2}(s^{(0)})$, więc współczynnik przy $Q^\mu_k$ powinien wynieść $1 = 2\cdot 1 = 2$. I jest wszystko o.k.}
\commentMR{Współczynnik 2 -> kolejne wzory mi się zgadzają.}
\commentMR{Note that $Q^a_{ij}$, $Q^a_\mu$ are not present in these formulas.}
From this, we can easily derive the coordinate expressions for the introduced mappings $\tilde{\pa}, \beta, \Box$:
\begin{equation}\label{e:coord_pa_beta_gamma}
\begin{cases}
\tilde{\pa}(e_i) &=  Q^\mu_i c_\mu,\\
\beta(e_i, e_j) &=  Q^\mu_{[ij]} c_\mu \\
\Box_{e_i} c_\nu &= - Q^\mu_{\nu i} c_\mu.
\end{cases}
\end{equation}
{\new where  $Q^\mu_{[ij]}$ is given in \eqref{d:Q_symm_and_Q_asym}, }
and the minus sign  in the last line arises from our preference for working with left actions.  Note that $\tilde{\pa}$ coincides with $\pa$ as $\core{\Rk^2}(e_i) = Q^\mu_i c_\mu$, see \eqref{e:R2_coord}.

The symmetric part $Q^\mu_{(ij)}$ of $Q^\mu_{ij}$ is involved\footnote{The assignment $(e_i, e_j) \mapsto Q^\mu_{(ij)} c_\mu$ {does not give rise to a globally defined map.} Change $(x^a, y^i, z^\mu)$ to $(x^a, y^i, z^\mu + \frac12 u^\mu_{ij})$ gives another assignment.}
in the canonical map $\Rk^2: \At{2}\ra E^2$ (see equations \eqref{e:R2_coord}).
It turns out that the remaining structure functions $Q^\mu_{ij, k}$ alone do not define any geometric mapping. {Instead,  it is} the functions
{
\begin{equation} \label{df:Qmuijk}
\tilde{Q}^\mu_{i j k} = Q^\mu_{ij, k}
-Q^l_{jk}Q^\mu_{li}-Q^l_{ik}Q^\mu_{lj}+Q^\nu_{ji}Q^\mu_{\nu k} - Q^a_k \frac{\pa Q^\mu_{ji}}{\pa x^a}
\end{equation}
}
that give a mapping
\begin{equation}\label{e:delta_coord}
\delta: \Sec(A)\times \Sec(A) \times \Sec(A) \ra \Sec(C), \delta(e_i, e_{j}, e_{k}) =  \frac12 \tilde{Q}^\mu_{ijk} c_\mu.
\end{equation}
A coordinate-free definition of $\delta$ is
\begin{equation}\label{df:delta}
 \delta(s_1, s_2, s) =  \frac12\, [\aliftB{s_1}{-1}, [\aliftB{s_2}{-1}, \aliftB{s}{0}]]  \in \VF_{-2}(E^2) \simeq \Sec(C).
\end{equation}
It is just a matter of direct computation of Lie brackets to show \eqref{e:delta_coord}.
Introduce
{
\begin{equation}\label{df:delta_sym}
\sym{\delta}_s(s_1, s_2):= \frac12 \delta (s_1, s_2, s) + \frac12 \delta(s_2, s_1, s)\, \text{ and } \quad \alt{\delta}_s(s_1, s_2):= \frac12 \delta (s_1, s_2, s) - \frac12 \delta(s_2, s_1, s),
\end{equation}
}
so $\delta =\alt{\delta} + \sym{\delta}$. The skew-symmetric part $\alt{\delta}_s$  of $\delta(\cdot , \cdot, s)$ satisfies
{
\begin{equation}\label{df:delta_asym}
\alt{\delta}_s(s_1, s_2) := \frac12 \delta(s_1, s_2, s) - \frac12 \delta(s_2, s_1, s) = {\new \frac14} [[\aliftB{s_1}{-1}, \aliftB{s_2}{-1}], \aliftB{s}{0}] = {\new -  \frac12}  \Box_s\beta(s_1, s_2).
\end{equation}
}
Further decomposition of $\sym{\delta}$ by means of the Schur decomposition $V\otimes \Sym^2 V = \Sym^3 V \oplus W$
(where $W$ is the kernel of the total symmetrization map) yields no additional information   as
$\sum_{g\in \mathbb{S}_3} \delta(s_{g(1)}, s_{g(2)}, s_{g(3)}) = 0$ due to the Jacobi identity for vector fields.
It will be convenient to work with
\begin{equation}\label{df:omega}
    \omega: \Sec(A)\times \Sec(A) \times \Sec(A) \ra \Sec(C), \quad \omega(s_1, s_2, s) = \delta(s_1, s_2, s) - \beta(s_1, [s_2, s])
\end{equation}
and its symmetric part
\begin{equation}\label{df:omega_sym}
    {\sym{\omega}_s(s_1, s_2) := \frac12 \omega(s_1, s_2, s) + \frac12 \omega(s_2, s_1, s)}
\end{equation}
instead of $\delta$ and $\sym{\delta}$. In local coordinates we have $\sym{\omega}_{e_k}(e_i, e_j) := {\frac12} \bar{\omega}_{ij,k}^\mu c_\mu$ where
\begin{equation}\label{e:bar_omega}
     \bar{\omega}_{ij, k}^\mu = Q^\mu_{ij, k} + Q^l_{kj}Q^\mu_{(il)} + Q^l_{ki}Q^\mu_{(jl)} + Q^a_j \frac{\pa Q^l_{ik}}{\pa x^a} Q^\mu_l + Q^a_i \frac{\pa Q^l_{jk}}{\pa x^a} + Q^a_k \frac{Q^\mu_{(ij)}}{\pa x^a}.
\end{equation}
{ Note that $\alt{\omega}:= \omega - \sym{\omega}$ satisfies
\begin{equation}\label{df:omega_alt}
    \alt{\omega}_s(s_1, s_2) = - \frac12 \Box_s \beta(s_1, s_2) - \frac12 \beta(s_1, [s_2, s]) + \frac12 \beta(s_2, [s_1, s]),
\end{equation}
due to \eqref{df:delta_asym}.}
\begin{ex} \label{ex:structure_maps_E[2]} We shall describe the structure maps of $(\At{2}, \kappa^{[2]})$ --  the second order prolongation $\At{2}$ of an AL algebroid $(A, \kappa)$.
In standard coordinates $(x^a, y^i, d x^a, d y^i)$ on $\T A$ it is given locally by the equations $d x^a  = Q^a_i y^i$, hence
$(x^a, y^i, dy^i)$ form a coordinate chart for $\At{2}$.
  The coordinate description of $\kappa^{[2]} \subseteq \T^2 A \times \T \At{2}$ is
\begin{equation}\label{e:local_kappa_E2}
\kappa^{[2]}:
\begin{cases}
\dot{x}^a = & Q^a_i \,\und{y}^i \\
\ddot{x}^a = & \frac{1}{2}\Qhat{Q}^a_{ij} \, \und{y}^i\und{y}^j + Q^a_i\,\und{dy}^i, \\
\dot{\und{x}}^a = & Q^{a}_i\,y^i \\
\dot{\und{y}}^i = & \dot{y}^i +  Q^i_{jk}\,\und{y}^j y^k \\
\left(\und{dy}^l\right)^{\cdot} = & \ddot{y}^l + Q^l_{ij} \, \und{dy}^i\,y^j + Q^l_{ij}\, \und{y}^i \dot{y}^j +\frac{1}{2} \wh{Q}^l_{ij,k} \und{y}^i\und{y}^j y^k,
\end{cases}
\end{equation}
where $\Qhat{Q}^a_{ij}$ are defined in \eqref{df:Qhat} and
\begin{equation}\label{df:Ql_ijk}
\wh{Q}^l_{ij,k} = \frac{\pa Q^l_{ik}}{\pa x^a} Q^a_j + \frac{\pa Q^l_{jk}}{\pa x^a} Q^a_i.
\end{equation}
{\new We find that $\pa: A \ra \core{\At{2}}$ defined  in \eqref{df:pa}, coincides with the identity on $A$, with respect to the isomorphism given in  Lemma~\ref{l:VF-k_and_core},
$$\pa = \id_A: A \ra \core{\At{2}} \simeq A.$$}
Moreover, $(\At{2}, \kappa^{[2]})$ is Lie, so $\ZMmap{\kappa^{[2]}}$ is a Lie algebra morphism. Hence, $\beta(s_1, s_2) =  \frac12 [\aliftB{s_1}{-1}, \aliftB{s_2}{-1}] = \frac12 \aliftB{[s_1, s_2]_{A}}{-2}$, see Theorem~\ref{th:HA_axioms_and_lifts}. Thus we may write $\beta(s_1, s_2) =  [s_1, s_2]_A$ up to the isomorphism $\core{\At{2}}\simeq A$. Similarly,  $\Box_s v =  [s, v]_A$, $\delta(s_1, s_2, s) =  [[s_1, s_2]_A, s]_A$ and   $\alt{\omega} = 0$ { from \eqref{df:omega_alt} and the Jacobi identity.} Moreover, $\sharp^C = \sharp$ by Lemma~\ref{l:VF-k_and_core}.
\end{ex}
\begin{rem}\label{r:str_maps_Ek} We can analogously define the following structure maps for any order $k$ HA $(E^k, \kappa^k)$:
$$
   \phi_{\K_1, \ldots, \K_n}(s_1, \ldots, s_n) = 
   \frac{1}{k!} [\ldots [[\aliftB{s_1}{-\K_1}, \aliftB{s_2}{-\K_2}], \aliftB{s_3}{-\K_3}], \ldots, \aliftB{s_n}{-\K_n}] \in \Sec(\core{E^k}),
$$
where $s_i\in \Sec(A)$ and $0\leq \K_i$ are such that $\sum_{i=1}^n \K_i = k$. In particular, $s\mapsto \frac{1}{k!}
\aliftB{s}{-k}$ defines a VB  morphism  $$ \partial^k: A \to \core{E^k}, \quad \Sec(\core{E^k}) \simeq \VF_{-k}(E^k).$$
Moreover, we have the structure maps
$$
\Box: \Sec(A) \times \Sec(\core{E^k})\ra \Sec(\core{E^k}), \quad
\Box_s v = [\alift{s}{0}, v] \in \Sec(\core{E^k}) \\
$$
and
\begin{equation}\label{e:sharp_Ek}
\sharp^{E^k}: \core{E^k} \ra \T M
\end{equation}
defined as the core of the anchor map $\sharp^k: E^k \ra \T^k M$ composed with the isomorphism $\core{\T^k M} \simeq \T M$ given in \eqref{e:i_kM}.
If $(E^k, \kappa^k)$ is Lie then, due to Theorem~\ref{th:HA_axioms_and_lifts},
$\phi_{\K_1, \ldots, \K_n}(s_1, \ldots, s_n) =  \frac{1}{k!} \aliftB{\left([\ldots [[s_1, s_2]_A, s_3]_A, \ldots, s_n]_A\right)}{-k}$, so all the structure maps $\phi_{\K_1, \ldots, \K_n}$ with fixed $n$ coincide with $\phi_{k, 0, \ldots, 0}$.

Let us assume that $(A, \sharp, [\cdot, \cdot]_A)$ is a Lie algebroid. Then $(\At{k}, \kappa^{[k]})$ is a Lie HA.  The map $\phi_k: \Sec(A) \ra \Sec(\core{\At{k}})$ gives the identification $A\simeq \core{\At{k}} \subset \T^{k-1} A$ which coincides with {\new $\jAc{k-1}: A\ra \core{\At{k}} \subset \T^{k-1} A$,}
and
$$
\phi_{\K_1, \ldots, \K_n} (s_1, \ldots, s_n) = [\ldots [[s_1, s_2]_A, s_3]_A, \ldots, s_n]_A,
$$
while {$\Box_s v  = [s, v]_A$.}
\commentMR{$\Box: \Sec{\core{E^k}} \times \Sec{A} \ra \Sec{\core{E^k}}$ to nie to samo, co $\phi_{k, 0}$. Can we assume that at most one of $\K_1, \ldots, \K_n$ is zero?}
\end{rem}
{\new We introduce a few additional  maps, denoted by $\xi$, $\psi$, $\veps$, $\veps_0$, $\veps_1$,
associated with a skew HA of order two. It will turn out that if  $(E^2, \kappa^2)$  is AL   then all these maps, except for $\veps_1$,  vanish. If $(E^2, \kappa^2)$  is  Lie then also $\veps_1$ is zero. These maps will be used in formulation of tensor-like properties of the structure maps we have already introduced.}
\begin{df} \label{df:epss} Let $(E^2, \kappa^2)$ be  a skew HA, $s_1, s_2\in \Sec(A)$, $f\in \Cf(M)$. We define
\begin{align} \label{df:xi}
&\xi: \Sec(A) \times \Sec(A)\ra \VF(M), \quad \xi(s_1, s_2) = \sharp[s_1, s_2] - [\sharp s_1, \sharp s_2], \\
&\psi: \Sec(A)\times \Sec(A)\ra \VF(M), \quad \psi(s_1, s_2)(f) := {\new\frac12} \aliftB{s_1}{-1} \aliftB{s_2}{-1} ((\sharp^2)^\ast \ddot{f}) -
(\sharp s_1)(\sharp s_2)(f),  \label{df:psi} \\
&\veps = \sharp^C \circ \pa - \sharp: A \ra \T M, \label{df:eps}\\
&   \veps_{k}(s_1, s_2) =  [\aliftB{s_1}{-k}, \aliftB{s_2}{-2+k}]-\aliftB{[s_1, s_2]}{-2} \in \VF_{-2}(E^2) \simeq \Sec(C), \label{df:eps_k}
\end{align}
where $k=0$ or $1$.
\end{df}
\begin{lem} \label{l:tensor_psi} The maps $\xi$, $\psi$, $\veps$, $\veps_0$, $\veps_1$ introduced in  Definition~\ref{df:epss} have the following properties:
\begin{enumerate}[(i)]
\item {
The maps $\xi$ and $\veps_1$ are tensorial in both arguments, so they give rise to the VB morphisms $\xi: \bigwedge^2 A\ra \T M$ and $\veps_1: \bigwedge^2 A\ra C$, respectively. Moreover, $\frac12 \veps_1(s_1, s_2) = \beta(s_1, s_2) - \pa([s_1, s_2])$.}
\item The map { $(s_1, s_2)\mapsto \veps_0(s_1, s_2)$} is tensorial in $s_2$, bot not in $s_1$, in general. We have $\veps_0(f s_1, s_2)  - f \veps_0(s_1, s_2) = \veps(s_2)(f) \pa(s_1)$. Moreover, ${ \frac12 } \veps_0(s_1, s_2) =  \Box_{s_1} (\pa s_2) - \pa ([s_1, s_2])$.
\item $\psi(s_1, s_2)$ is a derivation, hence the codomain of $\psi$ is correctly defined. Moreover,  $\psi(s_1, s_2)$ is tensorial in $s_1$, but it  is not tensorial in $s_2$, in general. Namely,
\begin{equation}\label{e:psi}
    \psi(s_1, g s_2) = g \psi(s_1, s_2) + (\sharp s_1)(g) \cdot {\new \veps(s_2)}, \tag{$\mathrm{Eq_\psi}$} \notag
\end{equation}
\end{enumerate}
In coordinates,
\begin{equation}\label{e:psi_coord}
\psi(e_{k'}, e_k) = {\new \frac12} \left(Q^\mu_{k'k} Q^a_\mu + Q^a_{k'k} -  2 Q^b_{k'} \frac{\pa Q^a_k}{\pa x^b}\right)  \pa_{x^a}.
\end{equation}
The skew-symmetric part of $\psi$, {$\alt{\psi} (s_1, s_2) = \frac12 \psi(s_1, s_2) - \frac12 \psi(s_2, s_1)$, is expressed in terms of the other structure maps:}
\begin{equation}\label{e:psi_skew_coord}
    \alt{\psi}(s_1, s_2) = {\new \frac12} \sharp^C(\beta(s_1, s_2)) - {\new \frac12} [\sharp s_1, \sharp s_2].
\end{equation}
The symmetric part of $\psi$, $\sym{\psi} = \psi - \alt{\psi}$, writes  in coordinates as\footnote{
It is tempting to consider a mapping $(e_i, e_j)\mapsto  \Qhat{Q}^a_{ij}\pa_{x^a}$. However, one can easily check that it does not give rise to a globally defined map.}
\begin{equation}\label{e:psi_bar_coord}
\sym{\psi}(e_i, e_j) = {\new \frac12} \left(Q^a_\mu Q^\mu_{(ij)} + Q^a_{ij} - \Qhat{Q}^a_{ij}\right) \pa_{x^a},
\end{equation}
{ where $\Qhat{Q}^a_{ij}$ are defined in \eqref{df:Qhat}.}
{Moreover, the condition $\sharp^2 \circ \Rk^2 = \sharp^{[2]}$ (compare with Theorem~\ref{thm:Rk}) is equivalent to the conjunction} $\sym{\psi} =0$ and $\sharp^C \circ \pa = \sharp$.
\end{lem}
The proof is given in Appendix, subsection~\ref{sSec:Ap:tensor}.

{ It  turns out that the map $\sym{\psi}$ corresponds to a certain \grB\ morphism. A slightly more general result holds:

\begin{lem} \label{l:grBA[2]_to_C} (a) Let   $(A\ra M, \sharp: A\to \T M)$ be an anchored vector bundle, and let $\rho: A \to C$ be a VB morphism . Then,  symmetric maps
    $\Psi: \Sec(A) \times \Sec(A) \ra \Sec(C)$ satisfying
    \begin{equation} \label{e:Phi_Leibniz}
        \Psi(s_1, f s_2) = f \Psi(s_1, s_2) + (\sharp s_1)(f) \rho(s_2)
    \end{equation}
    are in a one-to-one correspondence with \grB\ morphisms $\Phi: \At{2} \ra C_{[2]}$.
    The corresponding \grB\  morphism $\Phi: \At{2} \to C_{[2]}$ has the local form
\begin{equation}\label{e:Phi_coord}
    \Phi(x^a, y^i, \dot{y}^i) = \left(\rho^\mu_i(x) \dot{y}^i + \frac12 \Psi^\mu_{ij}(x) y^i y^j\right) c_\mu,
\end{equation}
   where  $\rho(e_i) = \rho^\mu_i c_\mu$ and $\Psi(e_i, e_j) = \Psi^\mu_{ij}  c_\mu$, and $(e_i)$ (respectively, $(c_\mu)$) is a local frame of sections of the vector bundle $A$ (respectively, $C$).
\end{lem}
The proof is given in Appendix, subsection~\ref{sSec:Ap:tensor}.
}

The structure maps defined above,   $\beta$, $\Box$, $\sym{\psi}$, and $\sym{\omega}$ are not $\Cf(M)$-linear in general, but satisfy certain  tensor-like identities presented in the following result.
\begin{thm}[order-two skew HAs] \label{th:skew_HA} (a)  Let $(E^2, \kappa^2)$ be a skew higher algebroid of order two, $A = E^1$, $C=\core{E^2}$. Let $v\in \Sec(C)$, $s, s_1, s_2 \in \Sec(A)$, $f\in\Cf(M)$. Then
\begin{itemize}
        \item \label{i:HA_beta} The map $\beta$ is skew-symmetric and
        \begin{equation}\label{e:tensor_beta}
            \beta(s_1, f\,s_2) = f\,\beta(s_1, s_2) + (\sharp s_1)(f) \pa(s_2).\tag{$\mathrm{Eq_\beta}$} \notag
        \end{equation}
        \item \label{i:HA_gamma} The map $(s, v)\mapsto \Box_s v$ satisfies 
            \begin{align}
                \Box_{f\,s}  v &= f \Box_{s}  v  - (\sharp^C v)(f)\,\pa(s), \label{e:gamma_1}\tag{$\mathrm{Eq^1_\Box}$} \notag \\
                \Box_{s} (fv) &= f \Box_{s} v  + (\sharp s)(f) v. \label{e:gamma_2}\tag{$\mathrm{Eq^2_\Box}$} \notag
            \end{align}
        \item The symmetric map $\sym{\psi}$ satisfies
\begin{equation}\label{e:tensor_psi_sym}
\sym{\psi}(s_1, f\,s_2) = f\,\sym{\psi}(s_1, s_2) +  \frac{1}{2} (\sharp s_1)(f) \veps(s_2).\tag{$\mathrm{Eq}_{\sym{\psi}}$} \notag
\end{equation}
                \item The map $\sym{\omega}_s (s_1,  s_2)$ is symmetric in $s_1$, $s_2$ and satisfies
                    { \begin{equation}\label{e:tensor_bar_omega_1}\tag{$\mathrm{Eq^1_{\bar{\omega}}}$} \notag
                        \sym{\omega}_s(s_1,  f s_2) = f \sym{\omega}_s( s_1, s_2) - \frac12 (\sharp s_1)(f) \cdot {\new \veps_0(s, s_2)} + {\newMR \frac12 } \xi(s, s_1)(f) \pa(s_2),
                    \end{equation}
                    \begin{equation}\label{e:tensor_bar_omega_2}
                    \sym{\omega}_{fs}(s_1, s_2) = f \sym{\omega}_{s}(s_1, s_2) + \frac14\left(\sharp s_1(f) \veps_1(s_2, s)+\sharp s_2(f) \veps_1(s_1, s)\right) + { \sym{\psi}(s_1, s_2)(f) \pa(s)},\tag{$\mathrm{Eq^2_{\bar{\omega}}}$}\notag
                \end{equation}
        (The maps  $\veps$, $\veps_0$, $\veps_1$, $\xi$, and $\sym{\psi}$  are as in Definition~\ref{df:epss}.)
                 }
            \end{itemize}
       (b)  Conversely,
            {\new let $(A, [\cdot, \cdot], \sharp)$ be a skew algebroid and $C\ra M$ be a vector bundle. }
            Then a system of the following maps:
         \begin{enumerate}[(i)]
            \item \label{i:str_map:pa} VB morphisms $\pa: A\ra C$ and $\sharp^C: C\to \T M$ covering the identity on $M$,
            \item \label{i:str_map:beta} { a skew-symmetric map $\beta: \Sec(A)\times \Sec(A) \ra \Sec(C)$ satisfying \eqref{e:tensor_beta}, }
            \item \label{i:str_map:sharp2:square} a map $\Box: \Sec(A)\times \Sec(C)\to \Sec(C)$ satisfying \eqref{e:gamma_1} and  \eqref{e:gamma_2},
            \item \label{i:str_map:psi_sym} a symmetric map $\sym{\psi}: \Sec(A) \times \Sec(A) \to \VF(M)$ satisfying \eqref{e:tensor_psi_sym},
            \item \label{i:str_map:omega} a map $\sym{\omega}: \Sec(A) \times \Sym^2 \Sec(A) \to \Sec(C)$ satisfying
                \eqref{e:tensor_bar_omega_1} and \eqref{e:tensor_bar_omega_2},
          \end{enumerate}
          define  a skew order-two HA on the \grB\ $E^2 = \quotient{(\At{2} \times_M C_{[2]})}{(\graphW(-\pa))}$ (see Lemma~\ref{l:structureE2}) uniquely.
           (Note that the maps $\veps$, $\veps_0$, $\veps_1$, $\xi$,  which appear in the Leibniz-type identities of the structure maps listed here, can be expressed in terms of  the aforementioned maps, see Definition~\ref{df:epss} and Lemma~\ref{l:tensor_psi}.)
\end{thm}

\begin{proof}
The proof of part (a) -- regarding the tensor-like properties of the structure maps $\beta$, $\sharp^C$, $\Box$, $\sym{\psi}$, and $\sym{\omega}$ listed above -- is technical and has been moved to Appendix, Subsection~\ref{sSec:Ap:tensor}.

Proof of part (b): Let $(A, [\cdot, \cdot], \sharp)$ be a skew algebroid, and assume  the structure maps listed above, $\pa$, $\beta$, $\Box$, $\sym{\psi}$, and $\sym{\omega}$, are given.

Given the  VB morphisms $\pa: A\to C$ and $\sharp: A \to \T M$,  the construction of the \grB\  $E^2$ as the quotient of $\At{2}\times_M C_{[2]}$ is well-founded, see  Lemma~\ref{l:structureE2}. We shall now  present the construction of the \grB\ morphism $\sharp^2: E^2\ra \T^2 M$.

There is  a \grB\ morphism $\Phi: \At{2} \ra (\T M)_{[2]}$ corresponding to $\Psi = \sym{\psi}$ and $\rho  = \veps$, as explained in Lemma~\ref{l:grBA[2]_to_C}. Define a map
\begin{equation}\label{e:sharp2const}
\tilde{\sharp}^2: \At{2}\times_M C_{[2]} \ra \T^2 M, \quad (a^2, v) \mapsto \sharp^{[2]}(a^2) \plus (\sharp^C v + \Phi(a^2)),
\end{equation}
where $a^2\in \At{2}$ and $v\in C$  project to the same point in $M$, and $\plus$ denotes the action of the core bundle $\core{\T^2 M} \simeq \T M$ on $\T^2 M$. We shall show that this map factors through   the action of the graph of $-\pa$, the subbundle of the core bundle $A\times_M C$, giving rise to a map from the quotient \grB\ $E^2$, constructed in Lemma~\ref{l:structureE2}. It remains to show that
$$
    \sharp^{[2]}(a^2 \plus b) \plus (\sharp^C (v - \pa b) + \Phi(a^2\plus b))
$$
does not depend on $b\in A$. Indeed, the change in the core is equal to
$$\core{\sharp^{[2]}} (b) - (\sharp^C\circ \pa)(b) + \core{\Phi}(b) = \sharp  b  - \sharp^C\circ \pa b + \veps (b)= 0, $$
 since $\core{\sharp^{[2]}} = \sharp$ and by the definition of $\veps$. A direct calculation from the coordinate formulas \eqref{e:psi_bar_coord} and \eqref{e:Phi_coord} shows that the resulting  map $E^2 \to \T^2 M$ is indeed given by the desired formula:
\begin{equation*}
\begin{split}
    \pullback{(\sharp^2)}(\ddot{x}^a) = \pullback{(\sharp^{[2]})}(\ddot{x}^2) + \pullback{(\sharp^C)}(\dot{x}^a) + \pullback{\Phi}(\dot{x}^a) =  \\
    Q^a_i \dot{y}^i + \frac12 \Qhat{Q}^a_{ij} y^i y^j + Q^a_\mu c_\mu^\ast + (Q^\mu_i Q^a_\mu - Q^a_i) \dot{y}^i + \frac12 (Q^a_{ij} - \Qhat{Q}^a_{ij}) y^i y^j =
    Q^a_\mu w^\mu + \frac12 Q^a_{ij} y^iy^j,
\end{split}
\end{equation*}
where $(x^a, y^i, w^\mu)$ is the adapted coordinate system on the quotient,  so $c_\mu^\ast = w^\mu - Q^\mu_i \dot{y}^i$ and $Q^\mu_{(ij)} = 0$, see Definition~\ref{df:adapted_coord}.

We  show now how to recover the comorphism $\kappa^2$, which covers the \grB\ $\sharp^2$ and governs the HA structure on the \grB\ $E^2$.
All the local structure functions $(Q^{\cdots}_{\cdots}) = Q^a_i$, $Q^a_{ij}$, $Q^a_\mu$, $Q^i_{jk}$, $Q^\mu_i$, $Q^\mu_{ij}$, $Q^\mu_{\nu i}$, $Q^\mu_{ij, k})$ can be derived from the structure maps listed above once we fix a graded coordinate system $(x^a, y^i, z^\mu)$ on $E^2$. (All these functions are defined locally, over an open subset $U\subset M$.) Without loss of generality, we may assume that $(x^a, y^i, z^\mu)$ is an adapted coordinate system (Definition~\ref{df:adapted_coord}), so $Q^\mu_{(ij)} = 0$.

The local structure functions $Q^a_i$, $Q^a_\mu$, and $Q^a_{ij}$ are derived from the map  $\sharp^2$. Next, $Q^\mu_{[ij]} = Q^\mu_{ij}$ and $Q^\mu_{\nu i}$ are derived from the maps $\beta$ and $\Box$, respectively,  by means of \eqref{e:coord_pa_beta_gamma}. Finally, $Q^\mu_{ij, k}$ is determined from $\sym{\omega}$, see \eqref{e:bar_omega}.

{
The structure functions $(Q^{\cdots}_{\cdots})$ establish a HA structure $(E^2_U, \kappa^2_U)$ over the base $U$, through the equations \eqref{e:local_kappa2}, where $E^2_U = (\sigma^2)^{-1}(U)$.  Moreover, the comorphism $\kappa^2_U$ determines all the structure maps  $\pa_U$, $\sharp^C_U$,  $\beta_U$, $\Box_U$, $\psi_U$, and $\sym{\omega}_U$ which are defined on sections of the vector bundles $\sigma^1_U: A_U\rightarrow U$ and $\core{\sigma^2_U}: C_U\rightarrow U$.
Also the other structure maps present in the formulation of our theorem, the maps  $\veps_U$, $(\veps_0)_U$, $(\veps_1)_U$ and $\xi$ are determined by $\kappa^2_U$, as explained in Definition~\ref{df:epss}.

These maps are consistent with the restrictions of the corresponding maps given at the outset as the latter are local operators and satisfy the same Leibniz-type identities.
For example, $\Box_U(e_i, c_\mu) = \Box|_U(e_i, c_\mu)$ by \eqref{e:coord_pa_beta_gamma}, where on the RHS, $\Box|_U$ denotes the given structure map $\Box$  restricted to $\Sec_U(A)\times \Sec_U(C)$. Moreover, both $\Box_U$ and $\Box|_U$ satisfy the same Leibniz-type identities given in \eqref{e:gamma_1} and \eqref{e:gamma_2}, because $\pa_U$ coincides with $\pa|_{\Sec(A)}$, and similarly for $\sharp$ and $\sharp^C$. Therefore,  $(\Box_U)(s|_U,  v|_U) = \Box_s v$, for any $s\in \Sec(A)$ and $v\in \Sec(C)$.

Therefore, if $U\cap U'\neq \emptyset$, then $(\kappa^2_U)\Big{|}_{U'\cap U}$ and  $(\kappa^2_{U'})\Big{|}_{U\cap U'}$ coincide with $\kappa^2_{U\cap U'}$, which is defined by the restrictions of the structure maps to the sections over $U\cap U'$.  Therefore,  $\kappa^2$ is globally well-defined.
}
\end{proof}
\begin{rem} \label{rm:skewHAs}
The map $\beta$ in the formulation of part (b) of  Theorem~\ref{th:skew_HA} can be replaced by the VB morphism $\veps_1: \bigwedge^2 A \rightarrow C$, defined in \eqref{df:eps_k}. Indeed, the map $\beta$ is related to $\veps_1$ via the formula given in Lemma~\ref{l:tensor_psi}, $\beta(s_1,s_2) = \frac12 \veps_1(s_1, s_2) + \pa([s_1, s_2])$.  Hence, the Leibniz-type identity \eqref{e:tensor_beta} follows from the Leibinz rule of the bracket $[\cdot, \cdot]$ on $\Sec(A)$. Note  also that the anchor $\sharp: A \to \T M$ is uniquely determined by the bracket $[\cdot, \cdot]$ on $\Sec(A)$.
\end{rem}

\subsubsection{HAs over a point}
We shall study   HAs $(\sigma^k: E^k\ra M, \kappa^k)$ in which the base $M = \{\pt\}$ is a point.  Any such  structure is fully described by a weight-respecting  mapping (see Theorem~\ref{th:HA_axioms_and_lifts})
\begin{equation}\label{e:kappa_point}
\ZMmap{\kappa^k}: \T^k \g \to \VF_{\leq 0}(E^k)
\end{equation}{
where the algebra $(\g, [\cdot, \cdot])$ is defined as the order-one reduction of $(E^k, \kappa^k)$. Let $(e_i)$ be a basis of the vector space $E^1= \g$, $(y^i)$ be the corresponding dual basis and let $(y^i, z^\mu)$ be a graded coordinate system for $E^k$ in which the weight $\w(y^i)=1$ and  $\w(z^\mu)\geq 2$. To define a HA $(\sigma^k: E^k\ra M, \kappa^k)$
it amounts to provide vector fields $\aliftB{e_m}{-\K} \in \VF_{-\K}(E^k)$ for $0\leq \K\leq k$ such that
{\begin{align} \label{e:alg_lift_HA}
    \aliftB{e_m}{0} &= \sum_i Q^i_{lm} y^l \pa_{y^i} + \sum_{\mu} f^\mu_{m}(y, z) \pa_{z^\mu}, \\
    \aliftB{e_m}{-1}  &=  \pa_{y^m}  + \sum_{\mu} g^\mu_{m}(y, z) \pa_{z^\mu},
\end{align}
}
where $Q^i_{lm}$ are the structure constants for $(\g, [\cdot, \cdot])$, i.e.,  $[e_l, e_m] = Q^i_{lm} e_i$, {\new see Theorem~\ref{th:HA_axioms_and_lifts}. It follows that } $f^\mu_m$ (resp. $g^\mu_m$) are homogeneous functions on $E^k$ of weight $\w(z^\mu)$ (resp., $\w(z^\mu)-1$).
The obtained (general) HA is AL if and only if the bracket $[\cdot, \cdot]$ is skew-symmetric.

\subparagraph{Order two.}
The  map $\Rk^k:  \T^{k-1}\g \ra E^k$ covers $\Rk^1 = \id_{\g}$, hence it gives a canonical section of the bundle projection $\sigma^k_1: E^k\ra E^1 = \g$. In case $k=2$, $\sigma^2_1: E^2\ra  \g$ is an affine bundle projection, hence the mapping $\Rk^2: \T \g = \g_{[1]} \oplus \g_{[2]} \ra E^2$ yields a canonical splitting
$$E^2 = \g_{[1]} \times  C_{[2]},$$
where $C=\core{E^2}$, and $\Rk^2(x, 0) = (x, 0)$ where $x\in \g$.
We are going to describe the structure of the {\new graded, finite dimensional  Lie algebra $\VF_{\leq0}:= \VF_{\leq0}(E^2)$.}
In standard  graded coordinates $(y^i, z^\mu)$  on $\g\times C$,  vector fields of non-positive weight $\K$, where$-2\leq \K \leq 0$,  have the following form
$$
X_0 = c_{j}^i y^j \pa_{y^i} + (c^\mu_{\nu} \, z^\nu + \frac{1}{2} \, c_{ij}^\mu \, y^i y^j)\,\pa_{z^\mu}, \quad
X_{-1} = c^i \pa_{y^i} +  c_{i}^\mu\, y^i\,\pa_{z^\mu} , \quad X_{-2} =  c^\mu \, \pa_{z^\mu}
$$
where $c^{\cdots}_{\cdots}$ are some constants, and $X_{\alpha} \in \VF_{\alpha}(E^2)$. The Lie algebra $\VF_{\leq 0}$ acts faithfully on the linear subspace $\cA_{\leq 2} \subseteq \Cf(\g \times C)$, spanned by homogenous functions of weight $\leq 2$. It has a $\R$-basis consisting of the functions $1, y^i, y^iy^j, z^\mu$ and we have  $\cA_{\leq 2}  \simeq \R\oplus \g^\ast\oplus C^\ast \oplus  \Sym^2 \g^\ast$. By examining the action of the vector fields $X_0$, $X_1$, $X_2$, we easily find the following decomposition (compare with a more general Lemma~\ref{l:structure_VF_Ek}),
\begin{equation}\label{e:VF_leq2}
   \VF_0\simeq \End(\g)\oplus \End(C) \oplus \Hom(\Sym^2 \g, C), \quad
    \VF_{-1}\simeq  \g \oplus \Hom(\g, C), \quad
    \VF_{-2} \simeq C.
\end{equation}
The formula for the Lie bracket on $\VF_{\leq 0}$ will be given in the proof of Theorem~\ref{thm:HA_point}, see \eqref{e:VF_alg}.

We shall describe  algebroid lifts $\aliftB{e}{-\K}$, where $\K\in\{0, 1, 2\}$, by means of the structure maps of $(E^2, \kappa^2)$. Then it will be straightforward to verify the condition given in  Remark~\ref{r:Lie_axiom}, ensuring that   $(E^2, \kappa^2)$  is a Lie HA.

\begin{thm} \label{thm:HA_point} The structure of a skew, order-two HA  over a point is fully determined by the linear maps $[\cdot, \cdot]: \bigwedge^2\g \ra \g$, $\pa: \g\ra C$, $\beta: \bigwedge^2\g \ra C$, $\Box: {\new \g \otimes C} \ra C$, $\sym{\omega}: \g \otimes \Sym^2 \g\ra C$. The associated { algebroid lifts $e \mapsto \aliftB{e}{-\K}\in \VF_{-\alpha}$ are } given  (with respect to the isomorphisms listed in \eqref{e:VF_leq2}) by
\begin{align*}
    \aliftB{e}{0} &=   [\cdot, e] \oplus {\new \Box_{-e}  (\cdot)} \oplus {\newMR 2 \sym{\omega}_e(\cdot, \cdot),} \\
    \aliftB{e}{-1} &=  e \oplus  \beta(\cdot, e), \\
   \aliftB{e}{-2} &=   \pa(e)
\end{align*}
\commentMR{Może dałoby się z tej definicji usunąć trochę minusów i 2?}
A skew HA $(\g\times C,  [\cdot, \cdot], \pa, \beta, \Box, \sym{\omega})$ is Lie if and only if $\g$ is a Lie algebra, $\Box$ equips $C$ with a $\g$-module structure,  $\pa: {\g} \ra C$ is a $\g$-module morphism, $\sym{\omega} =0$ and the mapping $\beta$ is given by
$$
\beta(x_1, x_2) = \pa([x_1, x_2]).
$$
{Hence,   order-two Lie higher algebroids over a point are in a one-to-one correspondence with  $\g$-module morphisms $\pa: {\g} \ra {C}$.} 
\end{thm}
The proof is straightforward but somewhat lengthy, so it has been  moved to Appendix, Subsection~\ref{sSec:Ap:tensor}.

\subparagraph{Order greater than two.}
 The \grB\ hosting a higher Lie algebroid over a point of order greater than two { need not  split in a canonical way.} A simple example is provided by a non-split graded space $E^k$ such that $E^1 = \{0\}$ -- a vector space of dimension 0. In this case,   the VB comorphism $\ZMmap{\kappa^k}$ must be the zero map, i.e.,  $(\kappa^k)_a: \T^k \{0\}  \ra \T_a E^k$ is the zero map for any $a\in E^k$, as the domain is zero-dimensional.  A concrete example of this is  $E^4  = \T_q^2 M$, where $q\in M$ is a fixed point on a manifold $M$, and  the linear coordinates
$(\dot{x}^a)$ on $\T_q M$  are  assigned weight $2$, while the weight of { $\ddot{x}^a$ is~$4$.}

Following   the example given in \cite[Section 6]{MJ_MR_HA_comorph_2018}, a graded Lie algebra 
 $\bigoplus_{i=0}^{k-1} \g_i$, where  $\g_0 = \g$, equipped  with a  graded Lie algebra morphism $A: \T^{k-1} \g \ra \bigoplus_{i=0}^{k-1} \g_i$ such that $A^0 = \id_{\g_0}$, gives rise to a  split Lie higher algebroid of order $k$.

\subsubsection{AL HAs of order two}
There are a few relations among  the structure maps {\new of a skew HA } $(E^2, \kappa^2)$  that we introduced above,  ensuring that  it is an  almost Lie algebroid.

\begin{thm}[order-two AL HAs] \label{th:AL_HA_str_maps}
Let $(E^2, \kappa^2)$ be a skew order-two HA.  Then $(E^2, \kappa^2)$ is AL if and only if
\leqnomode
\begin{align}
    &\tag{$\mathrm{AL_A}$} \label{i:E1} \text{$A$ is an almost Lie algebroid, i.e., $\xi =0$},\\
    &\tag{$\mathrm{AL_\pa}$} \label{i:AL_sharpEF} \sharp = \sharp^C\circ \pa, \text{i.e., } \veps=0, \\
    &\tag{$\mathrm{AL_\psi}$} \label{i:AL_psi} \psi = 0,\\
    &\tag{$\mathrm{AL_\Box}$} \label{i:AL_nabla}  \sharp^C(\Box_s v) = [\sharp s , \sharp^C v],\\
    &\tag{$\mathrm{AL_\omega}$} \label{i:AL_omega} \sharp^C\circ \omega = 0.
\end{align}
\reqnomode
\end{thm}
{\newMR
\begin{cor}\label{cor:AL_HAs} An order-two AL HA  on a \grB\ $E^2\ra M$ is defined by:
\begin{itemize}
    \item an AL algebroid structure on the vector bundle $A\ra M$;
    \item VB morphisms $\pa$, $\sharp^C$, and $\veps_1$,
    \item maps $\Box$ and $\sym{\omega}$ that satisfy the aforementioned Leibniz-type identities,
\end{itemize}
such that conditions  \eqref{i:AL_nabla} and  \eqref{i:AL_sharpEF} are satisfied, and the images of the maps  $\veps_1$ and $\sym{\omega}$ lie in the kernel of the VB morphism $\sharp^C$.
\end{cor}
}
{\newMR
\begin{proof} This corollary  follows directly from Theorem~\ref{th:AL_HA_str_maps} and Remark~\ref{rm:skewHAs}:

Assume $(E^2, \kappa^2)$ is  an AL HA. Then  $\alt{\psi} =0$ follows from \eqref{i:AL_psi}, and the identity
 \begin{equation}
    \label{i:AL_beta} \sharp^C(\beta(s_1, s_2)) = \sharp [s_1, s_2] \tag{$\mathrm{AL_\beta}$} \notag
        \end{equation}
follows from  \eqref{e:psi_skew_coord} and \eqref{i:E1}.  Therefore,
$\frac12 \sharp^C \circ \veps_1 = \sharp^C \circ \beta - \sharp^C\circ \pa \circ [\cdot, \cdot] = (\sharp- \sharp^C\circ \pa)\circ [\cdot, \cdot] = 0$, due to \eqref{i:AL_sharpEF}.
Clearly,
\begin{equation}
    \label{i:AL_omega_bar} \sharp^C\circ \sym{\omega} = 0.  \tag{$\mathrm{AL_{\bar{\omega}}}$} \notag 
\end{equation}
follows from \eqref{i:AL_omega}.

Conversely, from Theorem~\ref{th:AL_HA_str_maps} and Remark~\ref{rm:skewHAs}, it follows that the maps listed in Corollary~\ref{cor:AL_HAs} define a skew HA  (with $\sym{\psi} =0$).

The identity \eqref{i:AL_beta} holds by the assumption $\sharp^C \circ \veps_1 = 0$, \eqref{i:E1}, and  \eqref{i:AL_sharpEF}.  Hence $\alt{\psi} = 0$, and \eqref{i:AL_psi} follows from the assumption $\sym{\omega} = 0$.

We have, $\sharp^C\circ \alt{\omega} (s_1, s_2, s) =  \frac12  \sharp\circ \Jac(s_1, s_2, s)$, from
\eqref{df:omega_alt}, \eqref{i:AL_beta}, and  \eqref{i:AL_nabla}, where
\begin{equation} \label{e:Jac}
    \Jac(X_1, X_2, X) = [X_1, [X_2, X]]  - [X_2, [X_1, X]] - [[X_1, X_2], X].
\end{equation}
 Moreover,  $\sharp\circ \Jac(s_1, s_2, s) = \Jac(\sharp s_1, \sharp s_2, \sharp s) = 0$ since $A$ is an  AL algebroid. Therefore, $\sharp^C\circ \alt{\omega} =0$ and \eqref{i:AL_omega} follows from the assumption $\sharp^C \circ \sym{\omega} = 0$.
\end{proof}
}
\begin{rem}
The equation \eqref{i:AL_sharpEF} implies that $\psi$ is $\Cf(M)$-linear in both arguments. Also the difference between left and right hand side in \eqref{i:AL_nabla} is $\Cf(M)$-linear in $s$ and $v$ thanks to \eqref{i:AL_sharpEF}. Moreover, in this case, also $\sym{\omega}$ and $\omega$ are $\Cf(M)$-linear in all arguments, see \eqref{e:tensor_bar_omega_1}, \eqref{e:tensor_bar_omega_2}. Therefore, it is enough to check the condition listed in Theorem~\ref{th:AL_HA_str_maps}   for arguments  from local frames $(e_i)$, $(c_\mu)$ of $\Sec(A)$ and $\Sec(C)$, respectively.
\end{rem}
\commentMR{Why the link \eqref{e:sharp2const} does not work correctly? Take another link:  \eqref{e:test2} and \eqref{e:test1}. This is very strange. When I uncomment the equation {e:test2} then {e:sharp2const} works correctly, but {e:test2} - not.}
\begin{rem} \label{r:sharp2}  Note that in the AL case,  \eqref{e:sharp2const} simplifies to $\tilde{\sharp}^2(a^2, v) = \sharp^{[2]}(a^2) \oplus \sharp^C(v)$, giving rise to the $\nd{2}$-order anchor map $\sharp^2: E^2 \ra \T^2 M$.
\end{rem}
\begin{proof}[Proof of Theorem~\ref{th:AL_HA_str_maps}]
Let $(E^2, \kappa^2)$ be an almost Lie HA. Point \eqref{i:E1} is  part of the definition of an AL algebroid. {\newMR
We shall show \eqref{i:AL_beta}. From  Theorem~\ref{thm:Rk}~\eqref{it:Rk_rho} and Lemma~\ref{l:tensor_psi} we obtain  $\sym{\psi} = 0$ and \eqref{i:AL_sharpEF}.  We shall show \eqref{i:AL_beta} from which the condition $\sym{\psi}$ follows, hence \eqref{i:AL_psi}, see again Lemma~\ref{l:tensor_psi}.}

It is well known that if $X_i\in \VF(M)$, $Y_i\in \VF(N)$, and $f: M\ra N$ is a differentiable mapping such that the vector fields $X_i, Y_i$ are $f$-related for $i=1,2$ then  $[X_1, X_2]$ and $[Y_1, Y_2]$-are $f$-related, as well.
Hence, using  Theorem~\ref{th:HA_axioms_and_lifts}, we find that the vector fields $[\aliftB{s_1}{-1}, \aliftB{s_2}{-1}]\in \VF(E^2)$ and
$[\aliftB{(\sharp s_1)}{-1}, \aliftB{(\sharp s_2)}{-1}]\in \VF(\T^2 M)$ are $\sharp^2$-related.
On the other hand, 
$\frac12 [\aliftB{s_1}{-1}, \aliftB{s_2}{-1}] $ is a vector field of weight $-2$ corresponding to the section  $ \beta(s_1, s_2) \in \Sec(C)$ while 
$[\aliftB{(\sharp s_1)}{-1}, \aliftB{(\sharp s_2)}{-1}] = \aliftB{[\sharp s_1, \sharp s_2]}{-2}$ as $(\T^2 M, \kappa^2_M)$ is a Lie HA (see Example~\ref{ex:kappa^k_M}).
The latter corresponds to the section  $2 [\sharp s_1, \sharp s_2]\in \Sec(\T M)\simeq \Sec(\core{\T^2 M})$, see Lemma~\ref{l:VF-k_and_core}. Hence, using Lemma~\ref{l:VF-k} and \eqref{df:sharpC}, we find that $\sharp^C \circ \beta(s_1, s_2) = \core{\sharp^2}(\beta(s_1, s_2)) = [\sharp s_1, \sharp s_2]$. This completes the proof of \eqref{i:AL_beta}.

We shall prove the next two identities,  \eqref{i:AL_nabla} and \eqref{i:AL_omega}, in a similar way.

Consider $v \in \Sec(C)$ as the vector field $\coreVF{v}$ on $E^2$.  Then  the vector fields $v$ and $\frac12 \aliftB{\sharp^C(v)}{-2}\in \VF_{-2}(\T^2 M)$ are $\sharp^2$-related, see Lemmas~\ref{l:VF-k}, ~\ref{l:VF-k_and_core} and the definition of $\sharp^C$. Due to the AL-assumption on $(E^2, \kappa^2)$, the vector fields $\aliftB{s}{0}\in \VF(E^2)$ and $\aliftB{(\sharp s)}{0}\in \VF(\T^2 M)$ are $\sharp^2$-related, so the corresponding Lie brackets, i.e.,  the vector fields ${\new \Box_s v = [\aliftB{s}{0}, v]}\in \VF_{-2}(E^2)$ and ${\new [\aliftB{(\sharp s)}{0}, \frac12 \aliftB{(\sharp^C v)}{-2}]}\in \VF_{-2}(\T^2M)$ are $\sharp^2$-related, as well. Since $(\T^2 M, \kappa^2_M)$ is Lie, the latter vector field corresponds to $[\sharp s, \sharp^C v] \in \VF(M)$  and  \eqref{i:AL_nabla} follows.

For \eqref{i:AL_omega},  due to the AL-assumption,  ${\new 2} \delta(s_1, s_2, s) =  [\aliftB{s_1}{-1}, [\aliftB{s_2}{-1} \aliftB{s}{0}] ]\in \VF_{-2}(E^2)$ is $\sharp^2$-related with
$$ [\aliftB{(\sharp s_1)}{-1},[\aliftB{(\sharp s_2)}{-1}, \aliftB{(\sharp s)}{0}]]  =  [\aliftB{(\sharp s_1)}{-1},[\aliftB{(\sharp s_2)}{-1}, \aliftB{(\sharp s)}{0}]]
=  \aliftB{[\sharp s_1, [\sharp s_2, \sharp s]]}{-2} \in \VF_{-2}(\T^2 M).$$
Hence, the sections $\delta(s_1, s_2, s)$ and $\frac12 [\sharp s_1, [\sharp s_2, \sharp s]]$ are $\sharp^C$-related. Therefore,
$$\sharp^C\circ \omega(s_1, s_2, s) = \sharp^C\circ \delta(s_1, s_2, s) - \sharp^C \circ \beta(s_1, [s_2, s]) = \frac12[\sharp s_1, [\sharp s_2, \sharp s] - \frac12 \sharp [s_1, [s_2,  s]]  = 0 $$
due to \eqref{i:AL_beta} and \eqref{i:E1}.

On the other hand, we assume that the structure maps of $(E^2, \kappa^2)$  satisfy the conditions \eqref{i:E1}-\eqref{i:AL_omega} and shall show that the vector fields $\aliftB{s}{\K}$ and $\aliftB{(\sharp s)}{\K}$, both of weight $\K$, are $\sharp^2$-related for   $\K  =-2, -1, 0$. This implies that $(E^2, \kappa^2)$  is almost Lie due to Theorem~\ref{th:HA_axioms_and_lifts}.

The case $\K=-2$ is simple: $\pa(s) =  \frac12 \alift{s}{-2}$ and $\frac12 \aliftB{(\sharp s)}{-2}$ are $\sharp^2$-related due to the relation \eqref{i:AL_sharpEF}.

Let $\K=-1$.
The vector field $\aliftB{s}{-1}\in\VF_{-1}(E^2)$ is projectable onto $E^1$ and its projection $(\T\sigma^2_1) \aliftB{s}{-1}$ coincides with the $(E^1, \kappa^1)$-algebroid lift {$\aliftB{s}{-1}$}, see Lemma~\ref{l:comp_alg_lifts}. Similarly, $\aliftB{(\sharp s)}{-1} \in \VF_{-2}(\T^2 M)$ is also a vector field projectable onto $\T M$ and it coincides with the tangent algebroid $(\T M, \kappa_M)$-lift $\aliftB{(\sharp s)}{-1}$. Thus the case $\K=-1$ reduces to the condition that $(E^1, \kappa^1)$ is AL, and we are done.

 The proof in the case $\K = 0$ is given in Appendix, Subsection~\ref{sSec:AL_and_Lie_HA_eqns}
 where we perform direct calculations using coordinates.
\end{proof}

\subsubsection{Lie HAs of order two}
In the following result, we provide conditions (referred to as the axioms of Lie HAs) on the structure maps introduced earlier, ensuring that a given AL HA  $(E^2, \kappa^2)$ is a Lie HA.

\begin{thm}[Lie HAs of  order two] \label{th:Lie_axiom_maps}
Let $(E^2, \kappa^2)$ be an AL HA. Then $(E^2, \kappa^2)$ is Lie if and only if
\leqnomode
\begin{align}
&\tag{$\mathrm{Lie_A}$} \label{i:Lie_ax:A1} \text{$A$ is a Lie algebroid,}\\
&\tag{$\mathrm{Lie_\Box}$}  \label{i:Lie_ax:A_on_C} \Box_{[s_1, s_2]} v = \Box_{s_1}\Box_{s_2} v - \Box_{s_2}\Box_{s_1} v, \\
&\tag{$\mathrm{Lie_\pa}$}  \label{i:Lie_ax:pa} \pa([s_1, s_2]) = \Box_{s_1} \pa(s_2), \text{i.e.,  $\veps_0 =0$,}\\
&\tag{$\mathrm{Lie_\beta}$} \label{i:Lie_ax:beta} \beta(s_1, s_2) =  \pa([s_1, s_2]) , \text{i.e.,  $\veps_1 =0$,}\\
&\tag{$\mathrm{Lie_\omega}$}  \label{i:Lie_ax:omega} \omega = 0.
\end{align}
\reqnomode
\end{thm}
\begin{rem}\label{r:Lie_ax} The condition  \eqref{i:Lie_ax:omega} can be replaced with
\begin{equation}
\label{i:Lie_ax:omega_bar} \sym{\omega} = 0.
\tag{$\mathrm{Lie_{\bar{\omega}}}$}
\end{equation}
Indeed, vanishing of $\alt{\omega} = \omega  - \sym{\omega}$ follows from \eqref{df:omega_alt},
 \eqref{i:Lie_ax:pa} and  \eqref{i:Lie_ax:beta}. Thus,  \eqref{i:Lie_ax:omega} follows from \eqref{i:Lie_ax:pa}, \eqref{i:Lie_ax:beta} and \eqref{i:Lie_ax:omega_bar}.
\commentMR{The conditions \eqref{i:Lie_ax:pa} and  \eqref{i:Lie_ax:beta} can not be omitted from the formulation of the above -- they are not implied by   \eqref{i:Lie_ax:omega} and Theorem~\ref{th:skew_HA}, because it can happen that $\xi=0$ but $\veps_0\neq 0$. }
\end{rem}
\begin{rem}\label{r:LieHA:axioms}
It is a  straightforward calculation to show that, in { an AL} HA, the mapping  $\curv_{\Box}(s_1, s_2, v) := \Box_{s_1} \Box_{s_2} v - \Box_{s_2} \Box_{s_1} v - \Box_{[s_1, s_2]} v$, as well as the Jacobiator \eqref{e:Jac}, is a tensor.   {\new Moreover, if $(E^2, \kappa^2)$ is AL, then also the difference between LHS and RHS of the remaining conditions \eqref{i:Lie_ax:pa}, \eqref{i:Lie_ax:beta} and \eqref{i:Lie_ax:omega} is also tensorial.} Hence, it is enough to verify all the conditions given in Theorem~\ref{th:Lie_axiom_maps} on sections from local frames of the VBs $A$ and $C$.
\end{rem}
{
\begin{rem}\label{r:LieHA}
The structure of a Lie HA $(E^2, \kappa^2)$ is fully determined by the Lie algebroid structure on $A\to M$, along with the maps $\pa$, $\Box$, and $\sharp^C$, such that the following compatibility conditions hold: \eqref{i:AL_sharpEF}, \eqref{i:AL_nabla}, \eqref{i:Lie_ax:A_on_C}, and \eqref{i:Lie_ax:pa}. {\newMR  Indeed, we define a skew HA on the \grB\ $E^2$  described in Lemma~\ref{l:structureE2}  by setting $\beta$ via \eqref{i:Lie_ax:beta}, so that $\veps_1=0$, and
$\sym{\psi} = 0$, $\sym{\omega} = 0$, see Theorem~\ref{th:skew_HA}.   The resulting skew HA is AL (see Corollary~\ref{cor:AL_HAs}), and Lie as \eqref{i:Lie_ax:omega} follows from \eqref{i:Lie_ax:omega_bar}.
}\end{rem}
}

\begin{proof}
Assume that $(E^2, \kappa^2)$ is a Lie HA, hence  $(A, \kappa)$ is a Lie algebroid, hence \eqref{i:Lie_ax:A1} holds.
According to Theorem~\ref{th:HA_axioms_and_lifts}, an almost Lie HA $(E^2, \kappa^2)$ is Lie if and only if
\begin{equation}\label{e:kappa_2_bracket}
\aliftB{[s_1, s_2]}{i+j} = [\aliftB{s_1}{i}, \aliftB{s_2}{j}]
\end{equation}
for $s_1, s_2\in \Sec(A)$, and $(i, j) = (0, 0), (-1, 0), (-1, -1)$ and  $(-2, 0)$.
We shall show first that \eqref{e:kappa_2_bracket} implies { the remaining} conditions { \eqref{i:Lie_ax:A_on_C} -- \eqref{i:Lie_ax:omega}.}

The condition \eqref{i:Lie_ax:A_on_C} can be rewritten in the form
{ $$
    [\aliftB{[s_1, s_2]}{0}, v] = [\aliftB{s_1}{0}, [\aliftB{s_2}{0}, v]] - [\aliftB{s_2}{0}, [\aliftB{s_1}{0}, v]]
$$
}
and it follows from \eqref{e:kappa_2_bracket} with $(i, j)=(0,0)$ and the Jacobi identity for vector fields.  The { conditions \eqref{i:Lie_ax:pa} and \eqref{i:Lie_ax:beta}} can be equivalently written as \eqref{e:kappa_2_bracket} with $(i, j)=(-1, -1)$ and $(i, j)=(-2, 0)$. Indeed,
$$\beta(s_1, s_2) = \frac12 [\aliftB{s_1}{-1}, \aliftB{s_2}{-1}] = \frac12 \aliftB{[s_1, s_2]}{-2} = \pa([s_1, s_2]) =  [\aliftB{s_1}{0}, \frac12 \aliftB{s_2}{-2}] = \Box_{s_1}\pa(s_2).$$
Finally, \eqref{i:Lie_ax:omega} reads as {\new
$$
    [\aliftB{s_1}{-1}, [\aliftB{s_2}{-1}, \aliftB{s}{0}]] = [\aliftB{s_1}{-1}, \aliftB{[s_2, s]}{-1}]
$$}
(see \eqref{df:omega}), and this equality is true thanks to  \eqref{e:kappa_2_bracket} with $(i, j)=(-1,0)$.

Conversely, assume that the conditions \eqref{i:Lie_ax:A1}-\eqref{i:Lie_ax:omega} hold. Then, for $\eqref{e:kappa_2_bracket}_{(i, j)=(0, -2)}$, we write
$$
    [\aliftB{s_1}{0},  \frac12 \aliftB{s_2}{-2}] = \Box_{s_1}\pa(s_2) \stackrel{\eqref{i:Lie_ax:pa}}{=} \pa([s_1, s_2]) = \frac12 \aliftB{[s_1, s_2]}{-2}.
$$
Similarly, for  $\eqref{e:kappa_2_bracket}_{(i, j)=(-1, -1)}$:
$$
     \frac12 [\alift{s_1}{-1}, \aliftB{s_2}{-1}] = \beta(s_1, s_2)  \stackrel{\eqref{i:Lie_ax:beta}}{=}  \pa([s_1, s_2]) = \frac12 \aliftB{[s_1, s_2]}{-2}.
$$
The case { $(i, j)=(0, -1)$} is more complicated. Denote $\Delta : = [\aliftB{s_1}{0}, \aliftB{s_2}{-1}] - \aliftB{[s_1, s_2]}{-1}$,  so $\Delta\in \VF_{-1}(E^2)$. We shall show first that $\Delta$ is annihilated by $\T\sigma^2_1$.
We have
$$
    (\T \sigma^2_1) [\aliftB{s_1}{0}, \aliftB{s_2}{-1}]_{E^2} = [(\T \sigma^2_1)\aliftB{s_1}{0}, (\T \sigma^2_1) \aliftB{s_2}{-1}]_{E^1} = [s_1^{\langle 0\rangle_{\kappa^1}}, s_2^{\langle -1\rangle_{\kappa^1}}].
$$
{\new (We have used the compatibility of algebroid lifts with respect to $\kappa^2$ and  its order-one reduction $\kappa^1$, as guaranteed by Lemma~\ref{l:comp_alg_lifts}.)}
The latter is  $[s_1, s_2]^{\langle -1\rangle_{\kappa^1}}$ (as  $(A, \kappa^1)$ is Lie) and this coincides with the projection of $\aliftB{[s_1, s_2]}{-1} \in  \VF_{-1}(E^2)$ onto $E^1$. Hence, the vector field $\Delta $ is vertical with respect to the projection $\sigma^2_1:E^2\ra E^1$, as we claimed. {\newMR Hence,  $\Delta\in  \VFvert_{-1}(E^2) \simeq \Hom(A, C)$ by Lemma~\ref{l:structure_VF_Ek}\eqref{i:VF_E2:VF_-1}, i.e.,  $\Delta$ can be considered a VB morphism $A\ra C$.}

We know that $\omega = 0$, hence
From  condition \eqref{i:Lie_ax:omega} we find that for  any section $s\in \Sec(A)$ we have $[\Delta, \aliftB{s}{-1}]  = 0$.
For $X\in \VFvert_{-1}(E^2) \simeq \Hom(A, C)$ and $s \in \VF_{-1}(E^2) \simeq \Sec(C)$, the Lie bracket $[X, s]$ of vector fields on $E^2$ reads as
    $$
    [X, \aliftB{s}{-1}]_{E^2} =  -  X\circ s \in \Sec(C) = \VF_{-2}(E^2),
$$
see Lemma~\ref{l:structure_VF_Ek}.
We take $X:=\Delta$. Vanishing of $\Delta\circ s$ for any $s\in\Sec(A)$ implies $\Delta=0$.
\smallskip

We follow a similar idea in  the case $(i, j)=(0, 0)$. Consider $\heartsuit = \aliftB{[s_1, s_2]}{0} - \aliftB{[s_1, s_2]}{0}$, so $\heartsuit \in \VF_{-2}(E^2)$, and refer to the exact sequence \eqref{e:short_ex_sec_VF_0} 
in Lemma~\ref{l:structure_VF_Ek}. We aim to show that   $\heartsuit$ is in the kernel of the projection $\pi: X\mapsto (\T \sigma^2_1, X|_C)$. Indeed, the vector fields $\aliftB{s_1}{0}, \aliftB{s_2}{0}$ are tangent to the submanifold $C\subset E^2$ (as they have  weight $0$ and $C$ is given in $E^2$ by the equations $y^i=0$), so $[\aliftB{s_1}{0}, \aliftB{s_2}{0}]|_C = [\aliftB{s_1}{0}|_C, \aliftB{s_2}{0}|_C]$. Thus,  $\heartsuit|_C = 0$. Analogously to the case $(i, j)= (-1, 0)$, we have $(\T\sigma^2_1)\heartsuit = 0$, {\new as $(\T \sigma^2_1) [\aliftB{s_1}{0}, \aliftB{s_2}{0}]_{E^2} =  [s_1^{\langle 0\rangle_{\kappa^1}}, s_2^{\langle 0\rangle_{\kappa^1}}]_{E^1} = [s_1, s_2]^{\langle 0\rangle_{\kappa^1}}$.} Hence we know, that $\heartsuit \in \Hom(\Sym^2 A, \core{E^2}) \subset \VF_0(E^2)$, i.e.,  
it has a form
$$
    \heartsuit = \frac12 c^\mu_{ij}(x) y^i y^j \pa_{z^\mu}.
$$
for some functions  $c^\mu_{ij}$ on $M$. Next, we notice that for any section $s\in \Sec(A)$ we have
$$
    [\heartsuit, \aliftB{s}{-1}] = 0.
$$
Indeed, $[\aliftB{[s_1, s_2]}{0}, \aliftB{s}{-1}] \stackrel{(i, j)=(0, -1)}{=} \aliftB{[[s_1, s_2], s]}{-1}$ and
\begin{equation*}
\begin{split}
    [[\aliftB{s_1}{0}, \aliftB{s_2}{0}], \aliftB{s}{-1}] = [\aliftB{s_1}{0}, [\aliftB{s_2}{0}, \aliftB{s}{-1}]] -
    [\aliftB{s_2}{0}, [\aliftB{s_1}{0}, \aliftB{s}{-1}]] \stackrel{(i, j)=(0, -1)}{=} \\
    [\aliftB{s_1}{0}, \aliftB{[s_2, s]}{-1}] -
    [\aliftB{s_2}{0}, \aliftB{[s_1, s]}{-1}] =
    \aliftB{[s_1, [s_2, s]]}{-1} - \aliftB{[s_2, [s_1, s]]}{-1} = \aliftB{[[s_1, s_2], s]}{-1}.
\end{split}
\end{equation*}
Therefore, $[[\heartsuit, \aliftB{s_1}{-1}], \aliftB{s_2}{-1}] = 0$ for any $s_1, s_2\in \Sec(A)$. {\newMR On the other hand, for any  $\chi \in \Hom(\Sym^2 A, C)$,   we have
$$[[\chi, \aliftB{s_1}{-1}], \aliftB{s_2}{-1}]  = \chi(s_1, s_2)\in \Sec(C),$$
up to isomorphisms given in Lemma~\ref{l:structure_VF_Ek}. }Therefore, $\heartsuit =0$.
\end{proof}

\subsubsection{HAs of order two and representations up to homotopy of Lie algebroids}

The notion of the representation up to homotopy of Lie algebroids was introduced in \cite{AbCr2012}. Some recollection on this subject is given in Appendix, Subsection~\ref{sSec:representations}. In  our case of  interest (2-term representations), the definition given in \cite{AbCr2012} boils down to the following data:
 a Lie algebroid $(A\ra M, [\cdot, \cdot], \sharp)$,
a 2-term complex $F_0 \xrightarrow{\pa} F_1$ of vector bundles over $M$ concentrated in degrees $0$ and  $1$,
   $A$-connections $\nabla^i$ on $F_i$, for $i=0, 1$, and
   $A$-form $K\in \Omega^2(A; \Hom(F_1, F_0))$ such that
\begin{enumerate}[(i)]
    \item \label{i:pa_nabla} $\nabla_s^1 \circ \pa =  {\new \pa} \circ \nabla_s^0: \Sec(F_0)\ra \Sec(F_1)$
    for any $s\in \Sec(A)$;
    \item \label{i:pa_K} $\curv_{\nabla^0} = - K \circ \pa$, and  $\curv_{\nabla^1} = - \pa \circ K$ where $\curv_{\nabla}$ denotes the curvature of an $A$-connection $\nabla$, see \eqref{df:curv};
     \commentMR{Czy nie lepiej pozbyć się minusów i przejść do notacji jak u Ortiz? Nie, podążam [AC] choć tam jest misptint: chodzi o Remark 3.7, gdzie pownno być $R_{\nabla^E} = - K \comp \pa$ i $R_{\nabla^F} = - \pa \comp  K$.
    Ostatnie równanie to: $\dd_\nabla \circ \dd_\nabla + \wh{\pa} \circ \wh{K} + \wh{K} \circ \wh{\pa} = 0 $, czyli warunek $[D, D]=0$ z $n=2$, czyli $ \curv_{\nabla^0} = - K \comp \pa$ and $\curv_{\nabla^1} = - \pa \comp K$. }
    \item \label{i:d_K} the covariant derivative of $K$ vanishes, i.e.,  $\dd_{\nabla^{\Hom}} K = 0$ where $\nabla^{\Hom}$ is the $A$-connection on $\Hom(F_1, F_0)$ induced by $\nabla^0$ and $\nabla^1$, see \eqref{eqn:conn_Hom}.
    \end{enumerate}
    \commentMR{Znaki są inne czy te same co u \cite{AbCr2012}? Znaki są rzeczywiście przeciwne do tych w \cite{AbCr2012}???}
    \commentMR{Wydaje mi się, że powinny być jak powyżej. Mamy $\pa[K] \in \Omega^2(A, \und{\End}^0(F))$ zadane wzorem   $$\pa[K] (s_1, s_2) = (-1)^2 \pa(K(s_1, s_2)): v\mapsto \pa(K(s_1, s_2)(v)) - (-1)^1 K(s_1, s_2)(\pa v).$$ Mamy również $\pa \comp K= \pa \circ K$ i $K\comp \pa = K\circ \pa$.}
    (Note that $K\circ \pa = K \comp \pa$ and $\pa \circ K= \pa\comp K$, where $\comp$ denotes an operation on $A$-forms induced by the composition of maps, see \eqref{e:wedge_product}.) All this data can be gathered together to a so called \emph{the structure operator} $D: \Omega(A; F) \to \Omega(A; F)$,
    determined by the triple $(\pa, \nabla = (\nabla^0, \nabla^1), K)$ (also denoted by $D$)  defined by means of the wedge product, as $D := \wh{\pa} + d_\nabla + \wh{K}$, see \eqref{df:str_operator}.
    The  compatibility conditions \eqref{i:pa_nabla} - \eqref{i:d_K} can be shortly written as $D \circ D=0$, see Appendix.
    A morphism $(E_0\xrightarrow{\pa^E }E_1; \nabla^E, K^E)$ to $(F_0\xrightarrow{\pa^F}F_1; \nabla^F, K^F)$  consists of  a morphism of complexes $\Phi_0: (E, \pa^E) \to (F, \pa^F)$  (i.e.,  $\Phi_0 \circ \pa^E = \pa^F \circ \Phi_0$ \commentMR{$n=0$}) and a 1-form $\Phi_1\in \Omega^1(A; \Hom(E_1, F_0))$ such that
 \begin{enumerate}[(i)]
     \item $\pa^{\Hom} \Phi_1 + \dd_{\nabla} \Phi_0 = 0$, \commentMR{$n=1$, 1-forms}
     \item $\dd_{\nabla} \Phi_1 + K^F \comp \Phi_0 - \Phi_0 \comp K^E = 0$.  \commentMR{2-forms. Nie ma $\pa \Phi_2$, bo $\Phi_2=0$. }  
 \end{enumerate}
 These conditions  
 can be shortened to $[\wh{\Phi}, D]=0$
 and     can be rewritten in a more explicit form as
    \begin{equation}  {\new - } \Phi_1(s; \pa e) - \Phi_0(\nabla_s^{E_0} e) + \nabla_s^{F_0} \Phi_0(e) = 0 \text{ for  $e\in \Sec(E_0)$;} \label{i:Phi_0}
    \end{equation}
    \begin{equation}
     {\newMR - } \pa(\Phi_1(s; v)) - \Phi_0(\nabla_s^{E_1} v) + \nabla_s^{F_1} \Phi_0(v) = 0 \text{ for $v\in \Sec(E_1)$};  \label{i:Phi_1}
     \end{equation}
    \begin{equation} \label{i:Phi_2} \begin{split}
    K^F(s_1, s_2; \Phi_0(v)) - \Phi_0(K^E(s_1, s_2; v))   - \Phi_1([s_1, s_2]; v)
     {\newMR -} \Phi_1(s_2; \nabla_{s_1}^{E_1}v)  +  \nabla_{s_1}^{F_0} \Phi_1(s_2; v)  \\
      {\newMR + }  \Phi_1(s_1; \nabla_{s_2}^{E_1}v) - \nabla^{F_0}_{s_2} \Phi_1(s_1; v) =0,  \text{ for } s_1, s_2\in \Sec(A), v\in \Sec(E_1).
    \end{split}
     \end{equation}
\commentMR{The first two conditions go with $n=1$, while the third one go with 2-forms. All signs has been verified (MR notes). Note also that the above formulas define maps which are linear in $v$. }
The advantage of the framework of representations u.t.h. of Lie algebroids is that it is more flexible and contains  generalizations of some important  concepts from the theory of Lie algebras. The example is the adjoint representation.
It is modelled on the complex
$$
     A \xrightarrow{\sharp} \T M,
$$
and the $A$-connections on this complex is induced by  a  linear connection $\nabla: (X, s) \mapsto \nabla_X s$ on the vector bundle $A\to M$ in the following way\footnote{Two choices  of connections on $A$ leads to isomorphic representations.}
\begin{align}
    \nabla^{A}_{s_1} s_2 &= \nabla_{\sharp s_2} s_1 + [s_1, s_2], \label{e:nabla_adjA}\\
    \nabla^{\T M}_{s} X &= \sharp (\nabla_X s) + [\sharp s, X]_{\tau_M}.\label{e:nabla_adjTM}
\end{align}
The curvatures of the $A$-connections $\nabla^{A}$ and $\nabla^{\T M}$ are expressed in the terms of  the following $2$-form, called \emph{the basic curvature}  $R_\nabla^{\mathrm{bas}} \in \Omega^2(A; \Hom(\T M, A))$, as $\curv_{\nabla^A} = - R_\nabla^{\mathrm{bas}} \circ \sharp$, $\curv_{\nabla^{\T M}} = - \sharp \circ R_\nabla^{\mathrm{bas}}$, where \commentMR{Może zamienić $\circ$ na $\comp$ i $\sharp$ potraktować jako $0$-form? To nie ma znaczenia, bo znak $(-1)^{qi}$ w wedge product formula tutaj zawsze wynosi $1$.}
\begin{equation}\label{e:R_bas}
    R_\nabla^{\mathrm{bas}} (s_1, s_2; X) = \nabla_X [s_1, s_2] - [\nabla_X s_1, s_2] - [s_1, \nabla_X s_2] - \nabla_{\nabla_{s_2}^{\T M} X} s_1 + \nabla_{\nabla_{s_1}^{\T M} X} s_2,
\end{equation}
see \cite{AbCr2012}.
The structure operator for the adjoint representation of a Lie algebroid is denoted by { $ \ad_\nabla = (\sharp, (\nabla^A, \nabla^{\T M}),  \curvR^{\bas})$.}

\subparagraph{From order-two Lie HA to 2-term  representations.}
Let $(E^2, \kappa^2)$ be a {\new Lie} HA of order two.   Recall that it is determined by the Lie algebroid structure on the vector bundle $A\ra M$ (being  the order-one reduction of $E^2$), and the structure maps $\pa$, $\Box$, $\sharp^C$, see Remark~\ref{r:LieHA}.  We shall define a Lie algebroid representation u.t.h. on the complex
$
    A \xrightarrow{\pa} C
$,
$A$ in degree 0, $C$ -- degree 1.
Our construction mimics the adjoint representation {\new of a Lie algebroid.}

\begin{df} \label{df:HA_repr} Let us choose a linear connection $\nabla$ on the vector bundle $\sigma: A \ra M$ and define:
    \begin{itemize}
        \item an $A$-connection $\nabla^C$ on $C$:
            $$
                \nabla^C_s v := \Box_s v + \pa\circ \nabla_{\sharp^C(v)} s,
            $$
            where $s\in \Sec(A)$, $v\in \Sec(C)$;
        \item an $A$-connection $\nabla^A$ on $A$:
        $$
            \nabla^A_{s_1}{s_2} := \nabla_{\sharp s_2 } s_1 + [s_1, s_2];
        $$
        \item an two-form $K  \in \Omega^2(A, \Hom(C, A))$:
        $$
            K(s_1, s_2; v) : =\nabla_{\sharp^Cv} [s_1, s_2] - [\nabla_{\sharp^Cv} s_1, s_2] - [s_1, \nabla_{\sharp^Cv} s_2] - \nabla_{\sharp^C(\nabla_{s_2}^C v)} s_1 + \nabla_{\sharp^C(\nabla_{s_1}^C v)} s_2.
        $$
    \end{itemize}
    \end{df}
We assume that in both constructions, the adjoint representation and the representation on the complex $\pa: A \ra C$, we have chosen the same linear connection on $A$. Then, the $A$-connections $\nabla^A$  defined above and  in the adjoint representation, also coincide. Moreover,
\begin{equation}\label{e:K}
    K = \curvR^{\bas}_\nabla \circ \sharp^C,
\end{equation}
where $\curvR^{\bas}_\nabla$ is given in \eqref{e:R_bas}.
Indeed, by comparing the formulas for $K$ and  $\curvR^{\bas}_\nabla$, for \eqref{e:K} we need to show that
\begin{equation}\label{e:nabla_C}
\sharp^C\circ \nabla^C_s v = \nabla_s^{\T M}(\sharp^C v)
\end{equation}
This can be rewritten as
$$
\sharp^C (\Box_s v) + \sharp^C \circ \pa \, \nabla_{\sharp^C v} s = \sharp \left(\nabla_{\sharp^C v} s\right) + [\sharp s, \sharp^C v]
$$
and it is true due to the AL assumption (see  \eqref{i:AL_nabla}, \eqref{i:AL_sharpEF} in Theorem~\ref{th:AL_HA_str_maps}).

{\new
\begin{lem} \label{l:HA_to_Rep} An order-two Lie higher algebroid $(E^2, \kappa^2)$
gives rise, as explained in Definition~\ref{df:HA_repr}, to a representation u.t.h. of the Lie algebroid $A$ (the order-one reduction of $(E^2, \kappa^2)$) on the complex
\begin{equation} \label{e:rep}
A\xrightarrow{\pa} C
\end{equation}
with the structure operator given by $D = (\pa, (\nabla^A, \nabla^C), K)$. 
Two choices of the connection on the vector bundle $\sigma: A\ra M$ result in isomorphic representations. Moreover,  $\id_A \oplus \sharp^C$ gives rise to a  morphism from  the constructed representation  $(A_{[0]}\oplus C_{[1]}, D)$ to the adjoint representation $(A_{[0]} \oplus (\T M)_{[1]}, \ad_\nabla)$ of $A$.
\end{lem}
}
\begin{proof}
It is straightforward to check that $\nabla^C$ is an  $A$-connection. Indeed, using tensor-like properties of $\Box$ described in Theorem~\ref{th:skew_HA},  we get
$$
    \nabla^C_{fs} v = {\new f (\Box_s v)} - (\sharp^C v)(f)\, \pa s + f \, \pa \nabla_{\sharp^C v} s + \pa \left((\sharp^C v)(f)\,  s\right) = f \nabla_s^C v.
$$
Similarly, we check that $\nabla^C_s: \Sec(C) \ra \Sec(C)$ is  a derivative endomorphism,
 $$
    \nabla^C_{s} (f v) - f \nabla^C_{s} v =  \Box_{f s} v - f \Box_s v + \pa\circ\left(\nabla_{f\sharp^c v} s - f \nabla_{\sharp^C v} s\right) = (\sharp s)(f) v.
 $$
 The $A$-connections $\nabla^A$ and $\nabla^C$ are compatible with $\pa: A\to C$.  Indeed,
$$
    \pa \nabla^A_{s_1}s_2 = \pa \left([s_1, s_2]+ \nabla_{\sharp s_2} s_1\right)= \Box_{s_1}(\pa s_2) + \pa \nabla_{\sharp^C(\pa s_2)} s_1 = \nabla^C_{s_1}{\pa s_2}.
$$
Here, we used \eqref{i:AL_sharpEF} and \eqref{i:Lie_ax:pa}  which are true in any Lie HA. In  analogy to \cite[Proposition~2.11]{AbCr2012} we shall prove that
    \begin{enumerate}[(i)]
        \item $\curv_{\nabla^A} = -K\circ \pa$ and  $\curv_{\nabla^C} = - \pa \circ K$;
        \item $\dd_{\nabla^{\Hom}} K =0$, i.e.,   $K$ is closed with respect to the $A$-connection {\new $\nabla^{\Hom}$} on $\Hom(C, A)$ induced by $\nabla^A$ and $\nabla^C$.
    \end{enumerate}
Recall that the curvature of $\nabla^A$ is  $- \curvR^{\bas} \circ \sharp$, hence
$$
    \curv_{\nabla^A}(s_1, s_2; s) = - \curvR^{\bas}_\nabla(s_1, s_2; \sharp s) = - \curvR^{\bas}_\nabla(s_1, s_2;  \sharp^C(\pa s)),
$$
i.e., $\curv_{\nabla^A} = - (R_\nabla^{\bas}\circ \sharp^C) \circ \pa = - K \circ \pa$ due to \eqref{e:K}.
For the curvature of $\nabla^C$ we apply $\pa$ to \eqref{e:R_bas} and get
     \begin{align*}
  \pa \circ \curvR^{\bas}_\nabla(s_1, s_2; \sharp^C v) &=  \underbrace{\pa \nabla_{\sharp^C v} [s_1, s_2]}_\mathrm{I} - \underbrace{\pa [\nabla_{\sharp^C v} s_1, s_2]}_{\mathrm{II}}
    - \underbrace{\pa [s_1, \nabla_{\sharp^C v} s_2]}_{\mathrm{II}'}  - \underbrace{\pa \nabla_{\nabla_{s_2}^{\T M} \sharp^C v} s_1}_{\mathrm{III}} +  \underbrace{\pa \nabla_{\nabla_{s_1}^{\T M} \sharp^C v} s_2}_{\mathrm{III}'}
\end{align*}
On the other hand,
\begin{align*}
    \nabla^C_{[s_1, s_2]} v &= \underbrace{\Box_{[s_1, s_2]} v}_{\mathrm{J}} + \underbrace{\pa \nabla_{\sharp^C v} [s_1, s_2]}_{\mathrm{I}}, \\
    \nabla_{s_1}^C \nabla_{s_2}^C v &=  \nabla_{s_1}^C \left(\Box_{s_2}v + \pa \nabla_{\sharp^C v} s_2 \right) =  \\
     &= \underbrace{\Box_{s_1} \Box_{s_2} v}_{\mathrm{J}'} + \underbrace{\pa \nabla_{\sharp^C(\Box_{s_2}v)} s_1}_{\mathrm{III}_1} + \underbrace{\Box_{s_1}\left(\pa \nabla_{\sharp^C v} s_2\right)}_{\mathrm{II}} + \underbrace{\pa \nabla_{\sharp^C (\pa \nabla_{\sharp^C v} s_2)} \, s_1}_{\mathrm{III}_2}
\end{align*}
We have {analogous}  expressions $(\mathrm{J}'')$, $(\mathrm{III}_1')$, $(\mathrm{II}')$ and $(\mathrm{III}_2')$ for $\nabla_{s_2}^C \nabla_{s_1}^C v$. We should show that
$$
     \pa \circ \curvR^{\bas}_\nabla(s_1, s_2; \sharp^C v) +   \nabla_{s_1}^C \nabla_{s_2}^C v - \nabla_{s_2}^C \nabla_{s_1}^C v - \nabla^C_{[s_1, s_2]} v = 0.
$$
We see that the expressions $(\mathrm{J})$, $(\mathrm{J}')$ and $(\mathrm{J}'')$
cancel, due to \eqref{i:Lie_ax:A_on_C}. Similarly for  the two expressions denoted by $(\mathrm{I})$. Next, the expressions $(\mathrm{II})$  cancel due to  \eqref{i:Lie_ax:pa}, and similarly for $(\mathrm{II}')$. Finally,  for $(\mathrm{III})$ we have
$$
\nabla_{\nabla_{s_2}^{\T M} \sharp^C v} \, s_1  =   \underbrace{\nabla_{\sharp \nabla_{\sharp^C v} s_2} \, s_1}_{\mathrm{III}_2} +  \underbrace{\nabla_{[\sharp s_2, \sharp^C v]} s_1}_{\mathrm{III}_1},
$$
hence $(\mathrm{III})$ equals $(\mathrm{III}_1) + (\mathrm{III}_2)$. Similarly, $(\mathrm{III}')$  cancels with  the sum of $(\mathrm{III}_1')$ and $(\mathrm{III}_2')$. Here we used Theorem~\ref{th:AL_HA_str_maps}: \eqref{i:AL_nabla} and \eqref{i:AL_sharpEF}.
\smallskip

We shall prove the second claim that the 2-form $K$ is closed. We shall use \eqref{e:K} and the equality $\dd_{\nabla^{\bas}} \curvR^{\bas}=0$  which is proved in \cite{AbCr2012}.
We have
\begin{align*}
    \dd_{\nabla^{\Hom}} K (s_1, s_2, s_3; v) &=  \sum_{\text{cyclic}}  \left(\nabla^{\Hom}_{s_1} K(s_2, s_3)\right)(v)  - K([s_1, s_2], s_3; v) = \\
    & = \sum_{\text{cyclic}} \underbrace{\nabla_{s_1}^A  \left( K(s_2, s_3; v) \right)}_{\curvR^{\bas}_\nabla(s_2, s_3; \sharp^C v)} - \underbrace{K(s_2, s_3)(\nabla^C_{s_1} v)}_{\curvR^{\bas}_\nabla(s_2, s_3)(\sharp^C \nabla^C_{s_1} v)} - \curvR^{\bas}_\nabla([s_1, s_2], s_3; \sharp^C v)    = \\
     &=\dd_{\nabla^{\bas}} \curvR^{\bas}_\nabla(s_1, s_2, s_3; \sharp^C v) =  0,
\end{align*}
by \eqref{e:nabla_C}.
The proof that $D = (\pa, \nabla, K)$ is a structure operator is completed.

Let us assume that we have chosen two linear connections $\nabla$ and $\wt{\nabla}$ on the vector bundle $\sigma: A\ra M$. We define $\Phi_0 = \id_{A\oplus C}$  and $\Phi_1(s) (v) = \nabla_{\sharp^C v} s - \wt{\nabla}_{\sharp^C v} s$, $\Phi_1\in \Omega^1(A, \Hom(C, A))$. Note that $\Hom(C, A) = \End^{-1}(A_{[0]}\oplus C_{[1]})$. Then $\Phi_0 + \Phi_1$ establishes an isomorphism between the representations u.t.h. of the Lie algebroid $A$, induced from a given HA $(E^2, \kappa^2)$, defined by means of the linear connections $\nabla$ and $\wt{\nabla}$, respectively. Indeed, the equation \eqref{i:Phi_0} writes as
$$
    \Phi_1(s_1; \pa s_2) = \nabla^A_{s_1} s_2 - \wt{\nabla}^A_{s_1} s_2.
$$
The RHS is $\nabla_{\sharp s_2} s_1 - \wt{\nabla}_{\sharp s_2} s_1$  and the same is LHS as
 $\sharp^C \circ \pa  = \sharp$. The second equation \eqref{i:Phi_1} writes as
$$
    \pa(\Phi_1(s; v)) = \nabla^C_{s} v - \wt{\nabla}^C_{s} v
$$
 and both sides are equal to $\pa \circ \left(\nabla_{\sharp^C v} s - \wt{\nabla}^C_{\sharp^C v} s\right)$  due to the definitions of $\Phi_1$ and the $A$-connection on $C$ (see Definition~\ref{df:HA_repr}). The third equation \eqref{i:Phi_2} is a consequence of a similar result for the adjoint representation. Namely, if $\Psi_0 = \id_{A\oplus \T M}$, {\newMR $\Psi_1(s; X) = \nabla_X s - \wt{\nabla}_X s$,} $\Psi_1 \in \Omega^1(A,  \und{\End}^{-1}(A\oplus \T M))$, then $\Psi_0 + \Psi_1$ is an isomorphism $(A\oplus \T M, \ad_{\nabla}) \ra  (A\oplus \T M, \ad_{\wt{\nabla}})$ between the adjoint representations of the Lie algebroid $A$ associated with the linear connections $\nabla$ and $\wt{\nabla}$, respectively. We have
 $$\Phi_1(s; v) = \Psi_1(s; \sharp^C v), $$
 hence, from \eqref{e:K} and \eqref{e:nabla_C}, we find that
 $$
    \Phi_1(s_2; \nabla^C_{s_1} v) = \Psi_1(s_2; \sharp^C \nabla^C_{s_1} v) = \Psi_1(s_2; \nabla_{s_1}^{\T M} \sharp^C v).
 $$
Hence, \eqref{i:Phi_2} can be written in our case as:
\begin{align*}
    K_{\wt{\nabla}}(s_1, s_2; v) - K_{\nabla}(s_1, s_2; v) = \underbrace{\Phi_1([s_1, s_2]; v)}_{\Psi_1([s_1, s_2]; \sharp^C v)} + \Phi_1(s_2; \nabla^C_{s_1}v) {\newMR - } \Phi_1(s_1; \nabla^C_{s_2}v) + \\
        \wt{\nabla}^A_{s_2}(\Phi_1(s_1; v)) - \wt{\nabla}^A_{s_1}(\Phi_1(s_2; v)),
\end{align*}
and it follows from the same equation \eqref{i:Phi_2} applied to the adjoint representation, i.e., with  $K$, $v$, and $\Phi$ replaced with $\curvR^{\bas}$, $\sharp^C v$, and $\Psi$, respectively.

For the last statement, we clearly see that $\Phi_0 = (\id_A, \sharp^C): A \oplus C\to A\oplus \T M$ is a morphism of complexes, due to $\sharp = \sharp^C \circ \pa$ (see Theorem~\ref{th:AL_HA_str_maps}). We set $\Phi_1 = 0$ and find that  equation \eqref{i:Phi_0} holds automatically,   equation \eqref{i:Phi_1} is true due to \eqref{e:nabla_C}, and \eqref{i:Phi_2} reduces to \eqref{e:K}.
\end{proof}

\subparagraph{Recovering HA.}
Assume we are given a Lie algebroid $(A\to M, [\cdot, \cdot], \sharp)$, the structure operator $D=(\pa, (\nabla^A, \nabla^C), K)$, which provides a representation u.t.h. of $A$ on the complex $\pa: A\to C$. Let $\Phi$ be a morphism to the adjoint representation $(A, \ad_{\nabla})$, where $\nabla$ is a chosen linear connection on $A$. We assume that $\Phi_0|_A = \id_A$ and $\Phi_1 =0$, where $\Phi_0$, $\Phi_1$ are the components of $\Phi$:
    $$
        \xymatrix{
            A\ar[d]_{\id_A}\ar[rr]^{\pa} && C\ar[d]^{\sharp^C := \Phi_0|_C} \\
            A \ar[rr]^\sharp && \T M
            }
    $$
    We shall show how to recover the structure of a Lie HA  on the \grB\ $E^2$ constructed in Lemma~\ref{l:structureE2} by means of the VB morphism $\pa: A\to C$ given already.
    The structure map $\sharp^C$ of the HA are taken from the diagram above, as  $\sharp^C = \Phi_0|_C$.
    The structure map $\Box: \Sec(A)\times \Sec(C) \ra \Sec(C)$ is recovered by means of the formula given in Definition~\ref{df:HA_repr},
        $$
        \Box_s v  = \nabla^C_s v - \pa\circ \nabla_{\sharp^C(v)} s.
        $$
    We easily check the tensor-like properties of the action $(s, v)\mapsto \Box_s v$:
        \begin{itemize}
        \item $\Box_{fs} v - f (\Box_s v)  = - \pa\circ (\nabla_{\sharp^C v} (fs) - f  \nabla_{\sharp^C v} s) = - (\sharp^C v)(f)\,\pa(s)$,
        as $\nabla_{s}^C v$ is $\Cf(M)$-linear in $s$.
        \item $\Box_s(fv) - f\, (\Box_s v) = \nabla^C_s(f v) - f\nabla_s^C v = (\sharp s)(f) \, v$  due to the properties of $A$-connections.
        \end{itemize}
     We shall show that the compatibility conditions given in Theorem~\ref{th:Lie_axiom_maps}, which ensure  Lie HA structure are satisfied.

    Obviously,  \eqref{i:AL_sharpEF} is true due to the commutativity of the diagram above.
    Since $\Phi_1 =0$ the conditions \eqref{i:Phi_0},  \eqref{i:Phi_1}, \eqref{i:Phi_2} simplify to:
    \begin{enumerate}[(i)]
        \item The $A$-connections on the vector bundle $\sigma:A \to M$, being  part of the structure operators $D$ and $\ad_\nabla$, coincide;
         \item $\sharp^C \left(\nabla^C_s v\right)  = \nabla^{\T M}_s \left(\sharp^C v\right)$;
         \item \label{i:R^bas_vs_K} $\curvR^{\bas}_{\nabla}(s_1, s_2; \sharp^C v) = K(s_1, s_2; v)$.
    \end{enumerate}

    We have $$\sharp^C(\Box_s v) = \sharp^C\left(\nabla^C_s v - \pa\circ \nabla_{\sharp^C v} s\right) =  \nabla^{\T M}_s \left(\sharp^C v\right) - (\sharp^C \circ \pa) \nabla_{\sharp^C v} s = [\sharp s, \sharp^C v]$$
    due to the formula \eqref{e:nabla_adjTM} for $\nabla^{\T M}$. It proves  \eqref{i:AL_nabla}.
    Next,
    \begin{equation*}
    \begin{split}
    \Box_{s_1} \pa(s_2) =  \nabla^C_{s_1} (\pa s_2) - \pa \nabla_{(\sharp^C \circ \pa)s_2} s_1 = \\
    \nabla^C_{s_1} (\pa s_2) - \pa \left(\nabla^A_{s_1}s_2 - [s_1, s_2]\right) = \pa [s_1, s_2] + \left(\nabla_{s_1}^C \pa s_2 - \pa \nabla^A_{s_1} s_2\right) =  \pa [s_1, s_2],
    \end{split}\end{equation*}
    by the compatibility $\nabla^C$ with $\nabla^A$, so \eqref{i:Lie_ax:pa} is true. It remains to prove that $\Box$ satisfies \eqref{i:Lie_ax:A_on_C}, see Remark~\ref{r:LieHA}.
    \begin{align*}
        \Box_{[s_1, s_2]}  v  &=    \nabla^C_{[s_1, s_2]} v -  \pa \circ \nabla_{\sharp^C v} [s_1, s_2], \\
        \Box_{s_1}\Box_{s_2} v &= \Box_{s_1}\left(\nabla^C_{s_2} v - \pa\circ \nabla_{\sharp^C v}s_2\right)
            =\nabla^C_{s_1}\left(\nabla^C_{s_2} v - \pa \circ \nabla_{\sharp^C v} s_2\right) - \pa \circ \nabla_{\new \sharp^C  \Box_{s_2} v} s_1,
     \end{align*}
     hence, using \eqref{i:Lie_ax:pa}, we get 
     \begin{equation*}
     \begin{split}
       -\curv_{\Box} (s_2, s_2; v) :=  \Box_{[s_1, s_2]} v  - \Box_{s_1}\Box_{s_2}v + \Box_{s_2} \Box_{s_1}v = -\curv_{\nabla^C}(s_1, s_2; v)  \\
       - \pa \circ \left(\nabla_{\sharp^C v} [s_1, s_2] - \nabla_{[\sharp s_2, \sharp^C v]} s_1 + \nabla_{[\sharp s_1, \sharp^C v]} s_2\right)
         {\new +} \nabla^C_{s_1}(\pa \circ \nabla_{\sharp^C v} s_2) {\new - } \nabla^C_{s_2}(\pa \circ  \nabla_{\sharp^C v} s_1).
     \end{split}
     \end{equation*}
     We replace $-\curv_{\nabla^C}(s_1, s_2; v)$ with  $\pa \circ K(s_1, s_2; v) =  \pa \circ \curvR^{\bas}_\nabla(s_1, s_2; \sharp^C v)$, see \eqref{i:R^bas_vs_K}, and  $\nabla^C \circ \pa$ with $\pa \circ \nabla^A$, and  find   that $-\curv_{\Box} (s_2, s_2; v) = \pa \circ \triangle(s_1, s_2; \sharp^C v)$, where
     $$
        \triangle(s_1, s_2; X) = \curvR^{\bas}_\nabla(s_1, s_2; X) - \left(\nabla_X[s_1, s_2] - \nabla_{[\sharp s_2, X]} s_1 + \nabla_{[\sharp s_1, X]} s_2\right) + \left(\nabla_{s_1}^A \nabla_X s_2  - \nabla^A_{s_2}\nabla_X s_1\right).
     $$
     We expand the last bracket using the formula \eqref{e:nabla_adjA} for $\nabla^A$ and  replace $\curvR^{\bas}_\nabla(s_1, s_2)(X)$ with \eqref{e:R_bas}, and  after cancelling similar terms we get{\new
     \begin{equation*}
        \begin{split}
            \triangle(s_1, s_2)(X) =  - \nabla_{\nabla_{s_2}^{\T M} X} s_1 + \nabla_{\nabla_{s_1}^{\T M} X} s_2 +\left(\nabla_{[\sharp s_2, X]} s_1 - \nabla_{[\sharp s_1, X]} s_2\right) + \\
            \left( \nabla_{\sharp\nabla_X s_2} s_1  - \nabla_{\sharp\nabla_X s_1} s_2 \right) = 0
        \end{split}
     \end{equation*}
      thanks to the formula for $\nabla^{\T M}$ given in \eqref{e:nabla_adjTM}.  The proof of \eqref{i:Lie_ax:A_on_C} is completed. }  We have obtained the following result.
    \begin{thm} \label{t:Rep_to_HA}  Let $(A\to M, [\cdot, \cdot], \sharp)$ be a Lie algebroid and let us fix a vector bundle  $C\to M$, a linear connection  $\nabla$ on $A$ and a VB morphism $\pa: A \to C$ over $\id_M$.

    Assume, in addition, that  we are given a representation u.t.h.  of the Lie algebroid $A$ on the complex $\pa: A\ra C$,  and a morphism $\Phi = (\Phi_0, \Phi_1)$ from this to the adjoint representation $(A, \ad_{\nabla})$ such that $\Phi_1=0$ and $\Phi_0|_A = \id_A$. Here, $\Phi_i \in \Omega^i(A; \End^i(A_{[0]} \oplus C_{[1]}))$, $i=0, 1$,  are the components of $\Phi$.

    { Then,  there exists a unique HA  structure on the \grB\ $E^2$ constructed  in Lemma~\ref{l:structureE2},
    such that the representation u.t.h. of the Lie algebroid $A$, and the morphism $\Phi$, described in Lemma~\ref{l:HA_to_Rep}, are the given ones.  This
    establishes a one-to-one correspondence between order-two Lie HA structures on the \grB\ $E^2$
    and morphisms $\Phi$ of the above form.}
    \end{thm}

 \subsubsection{HAs, VB-alegbroids and representations up to homotopy}
    A brief account of VB-algebroids is given in Preliminaries.

    Recall that the constructions from Definition~\ref{df:HA_repr} and the adjoint representation depend on the choice of  a linear connection on the vector bundle $A\to M$. However, there is a way  to avoid this choice. The motivation comes from description of 2-term representations in a framework of VB-algebroids, as discovered in \cite{Gr-Meh2010}. In this framework, the adjoint representation of $A$ is the VB-algebroid $(\T A; \T M, A; M)$ -- the tangent prolongation of the algebroid $A$.
      \begin{cor} \label{cor:HA_VB-alg} A Lie HA $(E^2, \kappa^2)$ can be described by means of the following data:
        \begin{enumerate}[(i)]
            \item \label{i:HA_VB:DVB}  a VB-algebroid structure on a DVB $D$ whose side bundles are $C$ and $A$ and the  core is  also $A$;
            \item \label{i:HA_VB:Psi} a VB-algebroid morphism $\Psi$ from $D$ to $\T A$ (the adjoint representation of $A$) such that $\Psi$ is the identity on the side bundle $A$ and also on  the core bundle $A$:
    \begin{equation} \label{diag:VB-alg-Psi} %
     \xymatrix{D \ar[dd] \ar[rr] && A \ar[dd] \\
        & A \ar[rd] \ar@{_{(}->}[ul]& \\
        C \ar[rr] && M
    }
    \quad
    \xymatrix{ & \\ \ar[r]^{\Psi} & \\ & }
        \quad
    \xymatrix{ \T A \ar[dd] \ar[rr]^{\tau_A} && A \ar[dd] \\
        & A \ar[rd]  \ar@{_{(}->}[ul] & \\
        \T M \ar[rr] && M
    }
    \end{equation}
        \end{enumerate}
    \end{cor}
    We proceed with the proof by recalling   the correspondence between 2-term representations and VB-algebroids. Details are nicely presented in~\cite{GJMM18}.

    Let $(D; \sigma_E, \sigma_A; M)$  be  a DVB  with the core $C$, as in \eqref{e:VBalg}.
       As shown in \cite{Gr-Meh2010}, a VB-algebroid structure on the DVB $(D\ra E; A\ra M)$, as in \eqref{e:VBalg},  together with a horizontal lift $\theta_A: \Sec(A) \ra \Sec^{\lin}_E(D)$, i.e., a splitting of the short exact sequence \eqref{e:ses_Sec_lin},  gives rise to a representation u.t.h. of the Lie algebroid $A$. (Recall that such horizontal lifts are in bijective correspondence with  decompositions $D\ra E\times_M A\times_M C$ of the DVB $D$, and with the inclusions $\sigma: E\times_M A \ra D$.) We shall review this construction. First of all, it is  a representation on the 2-term complex $\pa: C\to E$, where $\pa$  is the core of the anchor map $\sharp^D: D\to \T E$. (Note that $\core{D} =C$ and $\core{\T E}   =E$.)  The $A$-connections on $C$ and $E$, denoted by $\nabla^{\mathrm{core}}$ and $\nabla_a^{\mathrm{side}}$  are the following (see \cite{GJMM18}):
    \begin{equation}\label{e:A-conn_VB-alg}
        \left(\nabla^{\mathrm{core}}_{a} c\right)^\dag = [\theta_A(a), c^\dag]_D, \quad \wt{\nabla_a^{\mathrm{side}}} = \sharp_D(\theta_A(a)),
    \end{equation} 
    where $a\in \Sec(A)$, $c\in \Sec(C)$ and $\xi \mapsto \wt{\xi}$ denotes 1-1 correspondence between derivative endomorphism of $\Sec(E)$ and linear vector fields on $E$. The last component, $A$-form $K \in \Omega^2(A; \Hom(E, C))$ is defined as
        \begin{equation}\label{e:K_VB_alg}
            K(a_1, a_2) = \theta_A([a_1, a_2]_A) - [\theta_A(a_1), \theta_A(a_2)]_D.
        \end{equation}
    A DVB morphism $\Psi$ between {\newMR decomposed } vector bundles, $E \times_M A \times_M C \to E' \times_M A' \times_M C'$, covering $\id_M$,  is uniquely defined by restrictions of $\Psi$ to the side bundles $A, E$ and the core $C$ and a 1-form $\chi \in \Omega^1(A, \Hom(E, C')) = \Sec(A^\ast \otimes E^\ast \otimes C')$,
    \begin{equation}\label{e:Psi_VB_alg}
      \Psi(e, a, c) = (\Psi|_E(e), \Psi|_A(a), \Psi|_C(c) +  \chi(a, e)).
    \end{equation}
    If $A=A'$ and $\Psi|_A = \id_A$ then $\Psi$ defines a graded VB morphism $\Phi_0: C_{[0]} \oplus E_{[1]} \ra  C'_{[0]} \oplus E'_{[1]}$, $\Phi_0 = \Psi|_C \oplus \Psi|_E$. Note that $\Hom(E, C') = \Hom_{-1}(C_{[0]} \oplus E_{[1]}, C'_{[0]} \oplus E'_{[1]})$.
\begin{thm}\label{th:VB-2-term_repr}  Let $A\to M$ be a Lie algebroid.
\begin{enumerate}[(i)]
    \item \cite{Gr-Meh2010} Let $D$ be a DVB as in \eqref{e:VBalg},  and $\theta_A$ be a horizontal lift. (It gives rise to a decomposition $D \simeq E\times_M A\times_M C$.)
        Then the formulas \eqref{e:A-conn_VB-alg} and  \eqref{e:K_VB_alg} establish a one-to-one correspondence  between algebroid structures on $D\to E$ that provide a VB-algebroid structure on the DVB  $D$ and 2-term representations u.t.h. of the Lie algebroid $A$ on the complex $\pa: C\to E$.
             \item \cite{DJO15}  Let the decomposed DVBs $D: A\times_M E \times_C$,  $D' = A'\times_M E' \times_C'$ carry VB-algebroid structures and assume that the Lie algebroids $A$,  $A'$ are the same. Then a DVB morphism $\Psi: D \to D'$ such that $\Psi|_A = \id_A$, {\newMR is } a VB-alegbroid morphism if and only if {\newMR $\Phi = (\Phi_0, \Phi_1)$, where $\Phi_0= \id_A \oplus \Psi|_C$, and $\Psi_1 = \chi$, is a morphism between the associated 2-term representations. }
\end{enumerate}
\end{thm}

    \begin{proof}[Proof of Corollary~\ref{cor:HA_VB-alg}] Let us assume that we are given a VB-algebroid morphism $\Psi$, as above. Denote $\sharp^C  := \Psi|_C: C \to \T M$. Let $\nabla$ be any linear connection on $A$. This corresponds to a decomposition $\sum^{\nabla}: A \times_M \T M \to \T A$ of the DVB $\T A$. Thanks to presence and properties of $\Psi$, the DVB $D$ has a decomposition, induced by $\nabla$, as well. Indeed, $\Psi$ is an affine bundle morphism from $D$ to $\T A$ covering  $\und{\Psi} = \id_A\times \sharp^C: A\times_M C  \to A\times_M \T M$. It is fiber-wise bijective since $\Psi$ is the identity on the core bundle $A$. Hence, there exists  a unique decomposition  $\sum^D: A\times_M C \to D$ such that $\Phi \circ \sum^D  =\sum^{\nabla} \circ \und{\Phi}$.

    In our case, the VB-algebroid structure on the DVB $D$, given in \eqref{diag:VB-alg-Psi},   induces a representation of the Lie algebroid $A$ on the complex $\pa: A \to C$.
    Besides, $\Psi$ as a morphism of VB-algebroids, induces a morphism $\Phi = (\Phi_0, \Phi_1)$ of 2-term representations, as described in Theorem~\ref{th:VB-2-term_repr}. In our case, $\Phi_1\in \Omega^1(A; \Hom(C, A))$ vanishes, since $\Psi$ respects the decompositions of $D$ and $\T A$. Therefore,  $\Phi$ is of the form described in Theorem~\ref{t:Rep_to_HA}, i.e.,     $\Phi = (\Phi_0, \Phi_1)$, $\Phi_1 = 0$, $\Phi_0|_{A} = \id_A$.

    We shall prove that the $A$-connection on $A$ in the complex $A\to C$  is the same as in the adjoint representation.
    According to \eqref{e:A-conn_VB-alg}, these  $A$-connections on the core bundles of $D$ and $\T A$, denoted by $\nabla^{\mathrm{core}(D)}$ and $\nabla^{\mathrm{core}(\T A)}$, respectively, are given by
    $$
        \left(\nabla^{\mathrm{core}(D)}_{a_1} a_2\right)^\dag = [\theta^D_A(a_1), a_2^\dag]_D,  \quad \left(\nabla^{\mathrm{core}(\T A)}_{a_1} a_2\right)^\dag = [\theta^D_A(a_1), a_2^\dag]_{\T A}
    $$
    where $a_1, a_2\in \Sec(A)$.
    We have $\Psi([\theta_A^D(a_1), a_2^\dag]_D) = [\theta_A^{\T A}(a_1), a_2^\dag]_{\T A}$ since $\Psi: D\to \T A$ is a Lie algebroid morphism (covering the projection $E\to M$), the corresponding decompositions of DVBs $D$ and $\T A$ are $\Psi$-related ($\Psi\circ \theta^D_A = \theta_A^{\T A} $) and $\Psi$ induces the identity on the core bundles. It follows that $\nabla^{\mathrm{core}(D)} = \nabla^{\mathrm{core}(\T A)}$.

    Hence, due to Theorem~\ref{t:Rep_to_HA}, we get a HA structure on a certain canonically constructed \grB\ $E^2$, defined in Lemma~\ref{l:structureE2}.
    We shall  prove that the obtained HA $(E^2, \kappa^2)$ does not depend on the choice of the linear connection $\nabla$ on $A$.
        It amounts  to showing that the structure map
        $$
            \Box_s v  =\nabla^C_s v - \pa\circ \nabla_{\sharp^C(v)} s
        $$
        does not depends on the choice of $\nabla$. (The $A$-connection  $\nabla^C$ on $C\to M$ associated with the VB-algebroid  $(D; C, A; M)$ in \eqref{diag:VB-alg-Psi} is the $A$-connection denoted by $\Delta^{\mathrm{side}}$ in \eqref{e:A-conn_VB-alg}.)
    Let $\wt{\nabla}$ be another linear connection on $A$, so $\wt{\nabla} - \nabla =: \varphi \in \Hom(A\otimes \T M, A)$. From  \cite[Remark~2.12]{GJMM18}, we find that
    $$
        \wt{\nabla}^C_s v - \nabla^C_s v = \pa \circ \phi(s, v) \in \Sec(C),
    $$
    where $\pa: A\to C$ is as above and $\phi\in \Hom(A\otimes C, A)$ is the difference of the decompositions of the DVB $D$ induced by the linear connections $\wt{\nabla}$ and $\nabla$. We have $\Phi(\sum^D(s, v)) = \sum^{\T A} (s, \sharp^C(v))$, hence $\phi(s, v) = \varphi(s, \sharp^C(v))$. Moreover,
    $$
        \pa \circ \left(\wt{\nabla}_{\sharp^C(v)} s - \nabla_{\sharp^C(v)} s \right) = \pa\circ \varphi(s, \sharp^C(v)),
    $$
    what finishes the proof of our claim, $\wt{\Box}_s v = \Box_s v$, as $\wt{\Box}_s v - \Box_s v = \wt{\nabla}^C_s v - \nabla^C_s v$, see Definition~\ref{df:HA_repr}.
    \end{proof}

\begin{ex}\label{ex:repr_from_A2} We shall describe the representation u.t.h. of  $A$ associated with the HA $(\At{2}, \kappa^{[2]})$ -- the $\nd{2}$ prolongation of a Lie algebroid $(A\ra M, [\cdot, \cdot], \sharp)$. From Example~\ref{ex:structure_maps_E[2]}, it follows from that this is a representation on the complex $\pa= \id_A: A\ra {\new \core{\At{2}}}  \simeq A$, and  the $A$-connections defined in Definition~\ref{df:HA_repr}, denoted by $\nabla^A$ and $\nabla^C$, coincide. Moreover,  the 2-form $K$ given in  Definition~\ref{df:HA_repr} is  the curvature of the $A$-connection $\nabla^A$. Indeed, we know from \cite{AbCr2012} that $\curv(\nabla^A) = {\new - } R_\nabla^{\mathrm{bas}} \circ \sharp$ while $K = R_\nabla^{\mathrm{bas}} \circ \sharp^C$, see \eqref{e:K}, so $\curv(\nabla^A) = {\new - } K$ as $\sharp = \sharp^C$ in our case.

Now consider the linearisation $\pLinr(\At{2})$ of the \grB\ $\At{2}$ as a DVB, where we shall recognize a VB-algebroid structure and a morphism to the adjoint representation corresponding to the HA structure on $\At{2}$, as described in Corollary~\ref{cor:HA_VB-alg}.
It was shown in \cite[Theorem~2.3.8]{BGG_2015} that $\pLinr(\At{k}) \simeq A \times_{\T M} \T \At{k-1}$ and that it carries a natural weighted algebroid structure. In the special case $k=2$,  we find that
$$
\pLinr(\At{2}) \simeq A\times_{\T M} \T A = \{(a, X) \in A\times \T A: \sharp a = (\T \sigma) X  \},
$$
and the DVB $\pLinr(\At{2})$ carries a canonical structure of a  VB-algebroid. Note that the side bundles and the core of $\pLinr(\At{2})$  are naturally identified with the VB $A\ra M$.
 Moreover, the Lie algebroid structure on the vector bundle $\pr_1: A\times_{\T M} \T A \ra A$, where $\pr_1$ is the projection onto the first factor,   is a special case of the construction called  \emph{the prolongation of a Lie algebroid}, see \cite{Martinez_geom_form_mech_lie_alg_2001} and \cite{Bruce_Grabowska_Grabowski_2016}.
 The morphism $\Psi$  of VB-algebroids has a straightforward form,  $\Psi: A\times_{\T M}\T  A \ra \T A$ is induced by the projection onto the second factor:
$$
    \xymatrix{A\times_{\T M} \T A \ar[dd]^{\pr_1} \ar[rr]^{\tau_A \circ \pr_2} && A \ar[dd] \\
        & A \ar[rd] \ar@{_{(}->}[ul] & \\
        A \ar[rr]  && M
    }
    \quad \xymatrix{ & \\ \ar[r]^{\Psi} & \\ & }
        \quad
    \xymatrix{ \T A \ar[dd] \ar[rr]^{\tau_A} && A  \ar[dd] \\
        & A  \ar[rd]  \ar@{_{(}->}[ul] & \\
        \T M \ar[rr] && M
    }
$$
\end{ex}

We shall illustrate now the procedure of reconstructing an HA from a given representation a Lie algebroid  and a morphism to the adjoint representation.

\begin{ex}\label{ex:HA_from_id} We shall reconstruct a HA $(E^2, \kappa^2)$ out of the adjoint representation of a Lie algebroid $(\sigma:A \to M, [\cdot, \cdot], \sharp)$ 
and the morphism $\Phi$ being the identity on $A_{[0]} \oplus (\T M)_{[1]}$. According to Lemma~\ref{l:structureE2}, $E^2$  is the quotient  $E^2= \quotient{\At{2}\times_M (\T M)_{[2]}}{\sim}$ where the relation $\sim$ is induced by the graph of $-\sharp: \core{\At{2}} \simeq A \ra \T M = \core{T^2 M}$. Note that order-one reduction of $E^2$ is $A$, and its   core  is $\T M$. We can geometrically describe the \grB\ $E^2$ as follows:

 Take $(X, v)$ and  $(Y, w)$ in $\At{2}\times_M (\T M)_{[2]}$. Then $(X, v) \sim (Y, w)$  if and only if $\tau_A(X) = \tau_A(Y)$, $(\T \sigma) X = (\T \sigma) Y$, and $\sharp (Y-X)  =v-w$, where $Y - X$ is consider as an element of $A$ via  the isomorphism between  $A$ and  $\V_M A \subset \T A$.
(Recall, $\At{2}$ consists of  $X\in \T A$ such that $(\T \sigma) X = \sharp \tau_A(X)  \in \T M$.)

 The structure maps of the HA  $(E^2, \kappa^2)$  are easy to describe: $\pa = \sharp$, $\sharp^C  =\id_{\T M}$ and $\Box_s v = \nabla^{\T M}_s v - \pa \nabla_v s  = [\sharp s, v]$, according to the definition of $\nabla^{\T M}$, see \eqref{e:nabla_adjTM}. Since $(E^2, \kappa^2)$ is a Lie HA we have $\omega =0$, and $\beta(s_1, s_2) = \sharp [s_1, s_2]$.
 \end{ex}

\subsection{Final remarks and questions}

The results presented in this paper (e.g. Theorems~\ref{t:Rep_to_HA},~\ref{thm:HA_point})  are the source of new examples of order-two \grBs\ and HAs, eg. Example~\ref{ex:HA_from_id}, and raise questions about the classification of HAs under certain natural assumptions.  This represents one potential direction for further development based on the findings of this paper.

Another avenue of research on HAs involves exploring how HAs of order
$\geq 3$ are related to representations u.t.h. of Lie algebroids.

In light of the paper \cite{BO19}  on the  integration of 2-term representations of Lie algebroids, a natural question arises about the integration of HAs. What higher-order groupoids are and how they relate to HAs?

Recall that HAs were introduced as geometric-algebraic structures providing a proper language to formulate a geometric formalism of higher-order variational calculus (generalizing the first-order case). We hope this work will encourage further developments in the area of HAs and higher-order geometric mechanics.


\section{Appendix}\label{sec:app}

{In what follows, $(x^a, y^i_w)$ denotes graded coordinates on a graded bundle $E^k\ra M$. In the case $k=2$,
we continue using the notation from Subsection~\ref{sSec:HA_order_two} and Example~\ref{ex:kappa2_coord}.
 In particular, $A = E^1$, $C= \core{E^2}$, and $(e_i)$, $(c_\mu)$ denote local frames of the  VBs $A$ and $C$, respectively; $(x^a, y^i, z^\mu)$ are graded coordinates on $E^2$ compatible with the chosen frames  $(e_i)$, $(c_\mu)$. }

\subsection{Vector fields of non-positive weight on  \grBs }

In the following lemma, we study  the structure of the space of vector fields of non-negative weight on a \grB\ $E^k$.
\begin{lem}\label{l:structure_VF_Ek} Let $\sigma^k: E^k\ra M$ be a
\grB\ of order $k$ and let $\VF_{\leq 0}(E^k) = \bigoplus_{j=-k}^0\VF_j(E^k)$ denotes the Lie algebra of non-positively graded vector fields on $E^k$.

\begin{enumerate}[(i)]
\item \label{i:VF0_on_E} The Lie subalgebra $\VF_0(E)$ of \emph{linear} vector fields on the total space $E$ of a  vector bundle $\sigma: E\ra M$ coincides with the Lie algebra of derivative endomorphisms of the dual bundle, $\VF_0(E) \simeq \derEnd({E}^\ast)$.  \commentMR{Może przypomnieć nazwę derivative endomorphism?}
\item \label{i:VF0_on_Ek} A vector field $X\in \VF_0(E^k)$  of  weight zero is projectable onto $E^j$ for any $0\leq j\leq k$, in particular on $M = E^0$. \commentMR{Może ustalić $j$ i poprawić weight zero na weight $\leq j$?}
\item \label{i:VF0_on_E2} $\VF_{0}(E^2)$ is an abelian extension by $\Sec(\Hom(\Sym^2 E^1, \core{E^2}))$ of  the Lie subalgebra of $\VF_0(E^1)\oplus \VF_0(\core{E^2})$ consisting of pairs $(X_1, X_2)$ such that
    $X_1$ and $X_2$ project onto the same vector field on $M$:
    \begin{equation}\label{e:short_ex_sec_VF_0}
0 \to  \Hom(\Sym^2 E^1, \core{E^2}) \to \VF_{0}(E^2) \xrightarrow{\pi}
\VF_{0}(E^1) \times_{\VF(M)} \VF_{0}(\core{E^2}) \to 0
\end{equation}
where the projection $\pi$ is given by $\pi(X) = ((\T \sigma^2_1)(X), X|_{\core{E^2}})$ and the kernel of $\pi$ can be canonically identified with the space of VB morphisms $\Sym^2 E^1 \to \core{E^2}$.
\item \label{i:VFshort_ex_sec} There is a short exact sequence of graded Lie algebras
\begin{equation}\label{e:short_ex_sec_VF}
0 \to  \VFvert_{<0}(E^k) \to \VF_{<0}(E^k) \to \VF_{<0}(E^{k-1}) \to 0
\end{equation}
where $\VFvert_{<0}(E^k)$ denotes the subspace of $\VF_{<0}(E^k)$ of those vector fields which are vertical with respect to the projection $\sigma^k_{k-1}: E^k\ra E^{k-1}$. 
\item \label{i:VF_E2:VF_-1} In case $k=2$, the homogeneous part of weight $-1$ of \eqref{e:short_ex_sec_VF} reads as
$$
0 \to  \VFvert_{-1}(E^2)\simeq \Hom(E^1, \core{E^2}) \to
\VF_{-1}(E^2) \to \VF_{-1}(E^1) \simeq \Sec(E^1) \to 0
$$
\end{enumerate}
\end{lem}
\begin{proof}
    Point \eqref{i:VF0_on_E}  is  well known, see e.g. \cite{K-S_M2002} or \cite[Remark 2.1]{ETV2011}.
    For the proof of \eqref{i:VF0_on_Ek}, write a vector field $X\in \VF_0(E^k)$ in a general local  form
    $$
        X  = f^a(x) \pa_{x^a} + f^i(x, y) \pa_{y^i_w},
    $$
    where functions $f_i(x, y)$ are homogenous of weight $w = \w(y^i_w)$. It follows that the function $f_i(x, y)$ does not depend on coordinates of weights greater than $\w(y^i)$, so it is the pullback of a function on $E^{\w(i)}$.

    Obviously,  $\T \sigma^k$ annihilates $\pa_{y^i_w}$ and  defines a projection  of the vector field $X$ onto $M$,  which is $f_a(x) \pa_{x^a}$.  A very similar proof works for the projections $\sigma^k_j: E^k\ra E^j$ where $j>0$.
     For the proof of \eqref{i:VF0_on_E2} write $X\in \VF_0(E^2)$ in the form
        $$
            X =  f^a(x) \pa_{x^a} + f^i_j(x) y^j \pa_{y^i} + \left(f^\mu_\nu(x) z^\nu + f^\mu_{ij} (x) y^i y^j\right) \pa_{z_\mu}.
        $$
      It follows that the vector field $X$ restricted to the submanifold $\core{E^2}$ is tangent to it and $X{|}_{\core{E^2}} =   f^a(x) \pa_{x^a} +f^\mu_\nu(x) z^\nu \pa_{z_\mu}$. Besides, $(\T \sigma^2_1) X = f^a(x) \pa_{x^a} + f^i_j(x) y^j \pa_{y^i}$, hence  $X{|}_{\core{E^2}}$ and $(\T \sigma^2_1) X$ project to the same vector field on $M$. Moreover, the kernel of the projection $\pi$ consists of vector fields of the form $ f^\mu_{ij} (x) y^i y^j  \pa_{z_\mu}$ which can be identified with a VB morphism from $\Hom(C^\ast, \Sym^2 (E^1)^\ast) = \Hom(\Sym^2 E^1, C)$, where $C = \core{E^2}$.

      For \eqref{i:VFshort_ex_sec} it is enough  to notice that a vector field $X \in \VF(E^k)$ of weight $\leq -1$ has a well defined projection on $E^{k-1}$. Point \eqref{i:VF_E2:VF_-1} is a direct consequence of (iv).
\end{proof}

\subsection{Leibniz-type identities of the structure maps of HAs}\label{sSec:Ap:tensor}

\paragraph{Proof of Lemma~\ref{l:grBA[2]_to_C}.} Let $(e_j)$ and $(\und{e}_j)$ be local frames of sections of the vector bundle $A\ra M$, related by $e_j  = T^i_j(x)\und{e}_i$. Let $(c_\mu)$ be a frame of the vector bundle  $C\ra M$.
The \grB\  morphism $\Phi: \At{2} \to C_{[2]}$ has the local form
$$
    \Phi(x^a, y^i, \dot{y}^i) = (\Phi^\mu_i(x) \dot{y}^i + \frac12 \Phi^\mu_{ij}(x) y^i y^j) c_\mu,
$$
where $\Phi^\mu_{ij} =  \Phi^\mu_{ji}$.
On the other hand, a map  $\Psi: \Sec(A) \times \Sec(A) \ra \Sec(C)$ satisfying the Leibniz-type identity \eqref{e:Phi_Leibniz} is locally determined by the VB morphism $\rho$ and local functions $\Psi^\mu_{ij}$, where $\Psi^\mu_{ij} = \Psi^\mu_{ji}$, as follows:
$$
    \Psi(e_i, e_j) = \Psi^\mu_{ij}(x) c_\mu.
$$
In the given correspondence, $\rho$ corresponds to  the core VB morphism $\core{\Phi}: A\to C$, via the isomorphism $\core{\At{2}} \simeq A$. To complete the proof, we shall show that the change, $(e_j) \mapsto (\und{e}_j)$, of local frames of $\Sec(A)$ results in the same transition functions for the local functions $(\Phi^\mu_{ij})$ as for   $(\Psi^\mu_{ij})$. By calculating the differential of $\und{y}^i$, we find that the local coordinates $(x^a, y^i, \dot{y}^i)$ on $\At{2}$ transform as $\und{x}^a=x^a$, $\und{y}^i = T^i_j y^j$,
$$\und{\dot{y}}^i = T^i_j \dot{y}^j +  \frac12 \alpha^i_{jk} y^j y^k,  \text{where } \alpha^i_{jk} =\frac{\pa T^i_j}{x^a} Q^a_k + \frac{\pa T^i_k}{x^a} Q^a_j.
$$
It follows that if $\und{\Phi}^\mu_i \und{\dot{y}}^i + \frac12 \und{\Phi}^\mu_{ij} \und{y}^i \und{y}^j = \Phi^\mu_j \dot{y}^j + \frac12 \Phi^\mu_{jk} y^j y^k$ then $\Phi^\mu_j = \und{\Phi}^\mu_i T^i_j$, and
$$
    \Phi^\mu_{jk} = \und{\Phi}^\mu_i \alpha^i_{jk} + \und{\Phi}^\mu_{j', k'} T^{j'}_j T^{k'}_k.
$$
On the other hand,
\begin{equation*}
\begin{split}
    \Psi(e_j, e_k) = \Psi(T^{j'}_j \und{e}_{j'},  T^{k'}_k \und{e}_{k'}) = T^{j'}_{j} T^{k'}_k \und{\Psi}^\mu_{j'k'} c_\mu + \frac12 (\sharp e_{j})(T^{k'}_k) \rho(\und{e}_{k'}) + \frac12 (\sharp e_{k})(T^{j'}_j) \rho(\und{e}_{j'}) = \\
    c_\mu \left(T^{j'}_{j} T^{k'}_k \und{\Psi}^\mu_{j'k'} +\frac12 \alpha_{jk}^i \und{\rho}^\mu_i\right).
   \end{split}
\end{equation*}
Therefore, the transformations for $\Psi^\mu_{ij}$ are the same as those for  $\Phi^\mu_{ij}$, as we claimed. \qed

 The following three lemmas concern the calculus with algebroid lifts introduced {\newMR in \eqref{df:algebroid_lift}. }
The first one, Lemma~\ref{l:alifts1},  is the most general -- we do not assume any HA structure.
\begin{lem} {\new Let $k\in \N$ and $M$ be a smooth manifold.} \label{l:alifts1}
    \begin{enumerate}[(i)]
    \item \label{i:alifts:Xk} If $X\in\VF(M)$, $f\in\Cf(M)$ then
        \begin{equation} \label{e:alifts:Xk}
            \pullback{(\tau^k_M)} X(f) =  \frac{1}{k!} \aliftB{X}{-k}(f^{(k)}),
         \end{equation}
        where $\aliftB{X}{-k}\in \VF_{-k}(\T^k M)$ is $(\T^k M, \kappa^k_M)$-lift of $X$ in weight $-k$.
    \item \label{i:alifts:v_rhoF} Let $\rho^k: E^k \to \T^k M$ be  any morphism of \grBs\ covering $\id_M$.  Let $v\in \Sec(\core{E^k})$, $f\in \Cf(M)$. Then
        \begin{equation}\label{e:rhoVF}
            (\coreVF{v})(\pullback{(\rho^k)} f^{(k)}) = \core{\rho^k} (v)(f),
        \end{equation}
        where $\coreVF{v}$ is the image of $v$ in $\VF_{-k}(E^k)$ (see Lemma~\ref{l:VF-k}) and $\core{\rho^k} (v) \in \Sec(\core{\T^k M})\simeq \VF(M)$ is understood as a vector field on $M$ thanks to the isomorphism  $\jM{k}: \T M \to \core{\T^k M}$ given in \eqref{e:i_kM}.
    \end{enumerate}
\end{lem}
\begin{proof} \eqref{i:alifts:Xk} Recall that the vector field $\frac{1}{k!}\aliftB{X}{-k}$ is constructed in two steps. First,  we take the vertical lift $X^{(0)}\in \Sec_{\T^k M}(\T^k E)$  of $X$ where $E=\T M$,  and then {\new we compose it with  $\kappa^k_M: \T^k \T M \ra \T \T^k M$, see \eqref{df:algebroid_lift}}. We shall describe $\frac{1}{k!}\aliftB{X}{-k}$ by means of its flow.

We take $E=\T M$ and $\alpha = k$ in \eqref{diag:vertical_lift_ZM} and read from \eqref{df:V} that the vertical lift $X^{(0)}$ sends $\tclass{k}{\gamma}\in \T^k M$ to $\tclass{k}{t\mapsto {\new \frac{1}{k!}} t^k X_{\gamma(t)}} \in \T^k_{\gamma(t)} \T M$. Hence, the vector field $\frac{1}{k!}\aliftB{X}{-k} \in \VF(\T^k M)$ is given by
$$
\frac{1}{k!} \aliftB{X}{-k}_{\tclass{k}{\gamma}} = \tclass{1}{u \mapsto \tclass{k}{t\mapsto \phi^X_{u t^k/k!}(\gamma(t))}} \in \T_{{\tclass{k}{\gamma}}}\T^k M,
$$
{\new where $(t, x)\mapsto \phi^X_t(x)$, $x\in M$, is the flow of the vector field $X$. Hence, }
\begin{multline*}
     \frac{1}{k!}\aliftB{X}{-k}(f^{(k)})  = \left. {\frac{\mathrm{d}}{\mathrm{d}u}} \right|_{u=0} f^{(k)}(\tclass{k}{t\mapsto \frac{1}{k!} \phi^X_{u t^k/k!}(\gamma(t))}) =
    \left. \dbydtk{k}  \right|_{t=0} t^k/k! \, \left.{\frac{\mathrm{d}}{\mathrm{d}u}} \right|_{u=0} f(\phi^X_u(\gamma(t))) = \\
    = \left. {\frac{\mathrm{d}}{\mathrm{d}u}} \right|_{u=0} f(\phi^X_u(\gamma(0))) = X(f)(\gamma(0))
\end{multline*}
as we claimed.

\eqref{i:alifts:v_rhoF}:
First of all, note that  the function $(\coreVF{v})(\pullback{(\rho^k)} f^{(k)})\in \Cf(E^k)$ has weight $-k+k=0$; hence, it is the pullback of a function on the base $M$, and it is enough  to verify the equality \eqref{e:rhoVF} at a point $m\in M$.

The tangent vector $(\coreVF{v})_m \in \T_m E^k$ is represented by the curve $t\mapsto t\cdot_{\core{E^k}} v_m$, which is equal to $h_{\sqrt[k]{t}}^{E^k}(v_m)$ if $t\geq 0$, where $h^{E^k}$ is the homogeneity structure on $E^k$. Hence,
$$
    \text{\LHS\ of \eqref{e:rhoVF}} = \left.\frac{\dd}{\dd t}\right|_{t=0}  f^{(k)}(\rho^k(t\cdot_{\core{E^k}} v_m)).
$$
Assume that the image of $\core{\rho^k}(v_m)$ in $\core{\T_m M}$ is represented by a curve $\gamma:\R \to M$, $\gamma(0)=m$, i.e.,
 $$
 \text{\RHS\ of \eqref{e:rhoVF}} =  \left.\frac{\dd}{\dd s}\right|_{s=0} f(\gamma(s)).
 $$
 Then $\rho^k(v_m) = \core{\rho^k}(v_m) = \tclass{k}{s\mapsto \gamma(s^k/k!)}$ as $\thh{k}$-velocity in  $\core{\T^k M} \subset \T^k M$, hence
 $$
 \rho^k(t\cdot_{\core{E^k}} v_m) = t\cdot_{\core{\T^k M}} \rho^k(v_m) = \tclass{k}{s\mapsto \gamma(t \cdot s^k/k!)}.$$
 Therefore,
$$
 \text{\LHS\ of \eqref{e:rhoVF}} = \left.\frac{\dd}{\dd t}\right|_{t=0} \left.\frac{\dd^k}{\dd s^k}\right|_{s=0} f(\gamma(t s^k/k!)).
$$
 Let $(x^a)$ be local coordinates on $M$ around $m$ such that $x^a(m)=0$. It is enough to prove \eqref{e:rhoVF} for $f=x^a$.  If $\gamma^a(t) := x^a(\gamma(t)) = c^a t + o(t)$ then  $\left.\frac{\dd^k}{\dd s^k}\right|_{s=0} \gamma^a(t s^k/k!) = c^a t + o(t)$. Hence, the left and right hand sides of \eqref{e:rhoVF} coincide with $c^a$.
\end{proof}

\begin{lem} \label{l:alifts_AL}  Let $(E^k, \kappa^k)$ be an AL HA.
\begin{enumerate}[(i)]
    \item \label{i:alifts:s_k} For $s\in \Sec(E^1)$ and $f\in \Cf(M)$, the following identities hold
        $$
             \frac{1}{k!} \aliftB{s}{-k}((\sharp^k)^\ast f^{(k)}) =  {\new (\sigma^k)^\ast} \sharp^1(s)(f) =
             \frac{1}{\K!}\aliftB{s}{-\K}((\sharp^k)^\ast f^{(\K)}) 
        $$
        for any $0\leq \K\leq k-1$.
    \item  \label{i:alifts:rho_rhoF}   $\sharp^1 = \sharp^{\core{E^k}}\circ \pa^k$  where $\sharp^{\core{E^k}}: \core{E^k} \ra \T M$ and $\pa^k: E^1 \ra \core{E^k}$ are the VB morphisms given in  Remark~\ref{r:str_maps_Ek}.
\end{enumerate}
\end{lem}
\begin{proof}
Proof of \eqref{i:alifts:s_k}: Denote $X=\sharp^1 s \in \VF(M)$ for time being.   We know from Theorem~\ref{th:HA_axioms_and_lifts} that the vector fields $\aliftB{s}{-k}$ and $\aliftB{X}{-k}$ are $\sharp^k$-related, hence for any {\new function} $\psi\in \Cf(\T^k M)$ we have $\aliftB{s}{-k}((\sharp^k)^\ast \psi) = (\sharp^k)^\ast \aliftB{X}{-k}(\psi)$. We take $\psi = f^{(k)}$, use \eqref{e:alifts:Xk} and get
$$
\frac{1}{k!} \aliftB{s}{-k}((\sharp^k)^\ast f^{(k)}) = \frac{1}{k!}(\sharp^k)^\ast \aliftB{X}{-k}(f^{(k)}) = (\sharp^k)^\ast (\tau^k_M)^\ast  X(f) =  (\sigma^k)^\ast X(f)
$$
 {\new as $\sigma^k = \tau^k_M\circ \sharp^k$.} This proves the first equality.

The second one follows from Lemma~\ref{l:comp_alg_lifts}. Indeed, consider the reduction of $(E^k, \kappa^k)$ to weight $\alpha$. We find that the vector field $\aliftB{s}{-\K}\in \VF_{-\K}(E^k)$ is projectable onto $E^{\K}$ and its projection is
$s^{\langle-\K\rangle_{\kappa^{\K}}}$ hence the equality $(\sigma^k)^\ast \sharp^1(s)(f) = \frac{1}{\K!}\aliftB{s}{-\K}((\sharp^k)^\ast f^{(\K)})$ follows from the previous one by replacing $k$ with $\K$.

Proof of \eqref{i:alifts:rho_rhoF}: The claim follows from the commutativity of the following diagram
$$
\xymatrix{
\Sec(E^1) \ar[rr]^{\pa_{E^k}} \ar[d]_{\sharp} && \VF_{-k}(E^k) \ar@{-|>}[d]^{\sharp^k} \ar[rr]^\simeq && \Sec(\core{E^k})\ar[d]^{\core{\sharp^k}}  & \\
\VF(M) \ar[rr]^{\pa_{\T^k M}} && \VF_{-k}(\T^k M) \ar[rr]^\simeq && \Sec(\core{T^k M}) \ar[r]^\simeq  &  \VF(M)
}
$$
where the arrow in the middle, labelled by $\sharp^k$,  denotes a relation: $(X, Y)\in  \sharp^k$ if the vector fields $X\in \VF(E^k)$ and $Y\in \VF(\T^k M)$ are $\sharp^k$-related. Actually, this relation restricted to the lowest degree $-k$ becomes a mapping $\sharp^k: \VF_{-k}(E^k) \ra \VF_{-k}(\T^k M)$. Moreover,  $\pa_{E^k}=\pa^k$ and ${\pa_{\T^k M}}$ are defined by means of algebroid lifts, as in Remark~\ref{r:str_maps_Ek}. It follows from Lemma~\ref{l:VF-k_and_core} that the composition of maps in the lower row is the identity on $\VF(M)$.  All maps in the diagram  are $\Cf(M)$-linear, hence they give rise  to VB morphisms. The square on the left is commutative due to Theorem~\ref{th:HA_axioms_and_lifts} and the AL assumption. The square on the right is also commutative and it is a more general fact: $\sharp^k$ can be replaced there with any \grB\ morphism $\rho^k: E^k \ra F^k$, as stated in Lemma~\ref{l:VF-k}.
\end{proof}
Recall, that for $k=2$, the formula \eqref{e:fs_lift} gives
\begin{equation} \label{e:fs2}
\begin{split}
   \aliftB{(f s)}{-2} &= f \aliftB{s}{-2} = 2 f \pa(s), \\
   \aliftB{(fs)}{-1} &= f \aliftB{s}{-1}  + (\pullback{\sharp}\dot{f}) \aliftB{s}{-2}, \\
        \aliftB{(fs)}{0} &= f \aliftB{s}{0}  +   (\pullback{\sharp} \dot{f}) \aliftB{s}{-1} +  \frac12 ((\sharp^2)^\ast \ddot{f}) \aliftB{s}{-2}.
 \end{split}
\end{equation}
 We shall need the following lemma for proving tensor-like properties of some structure maps associated with a skew HA $(E^2, \kappa^2)$.
\begin{lem} \label{l:alifts_skew} Let $(E^2, \kappa^2)$ be a skew HA. 
Let  $f\in \Cf(M)$, $s, s_1, s_2\in \Sec(A)$ and $v\in \Sec(C) \simeq \VF_{-2}(E^2)$. Then
    \begin{enumerate}[(i)]
        \item \label{e:sharpFv} $(\sharp^C v) (f) =  \coreVF{v}((\sharp^2)^\ast \ddot{f})$, where $\coreVF{v}$ is given in Lemma~\ref{l:VF-k}, \commentMR{Czy nie powinno być $\frac12 (\sharp^2)^\ast \ddot{f}$? Nie, jest dobrze.}
        \item \label{e:sharpFpa} $\frac12 \aliftB{s}{-2}((\sharp^2)^\ast \ddot f) = (\sharp^C \circ \pa )(s)(f)$,
        \item \label{e:sharp_lift1} $(\sharp s)(f) = \aliftB{s}{-1}(\pullback{\sharp} \dot{f}) = \aliftB{s}{0}(f)$,
        \item \label{e:sharp_lift2ddf} $\aliftB{s_1}{-1} \aliftB{s_2}{0} (\pullback{\sharp} \dot{f}) = \left(\sharp [s_1, s_2] +  \sharp s_2 \circ \sharp s_1\right) (f)$.
    \end{enumerate}
\end{lem}
\begin{rem}
We consider $f\in \Cf(M)$ and $\pullback{\sharp}(\dot{f})$ as functions on $E^2$ {\new using} the pullbacks of $f$ by $\sigma^2: E^2\ra M$ and $\sigma^2_1: E^2\ra E^1$, respectively. Note that  if $(E^2, \kappa^2)$ is AL, then $\sharp^C\circ \pa = \sharp$, hence \eqref{e:sharpFpa} coincides in this case with Lemma~\ref{l:alifts_AL} \eqref{i:alifts:s_k} with $k=2$.  It is tempting to add in \eqref{e:sharp_lift1} the equality $(\sharp s)(f) = \frac12 \aliftB{s}{-2}(\left(\sharp^2\right)^\ast \ddot{f})$, but this requires the assumption \eqref{i:AL_sharpEF}, see Lemma~\ref{l:alifts_AL}. In the AL case, point \eqref{e:sharp_lift2ddf} simplifies to $\aliftB{s_1}{-1} \aliftB{s_2}{0} (\pullback{\sharp} \dot{f}) =    \sharp s_1(\sharp s_2 (f))$.
\end{rem}
\begin{proof}
    In general, the $(\alpha)$-lift $f^{(\alpha)} \in \Cf(\T^k M)$ has weight $0\leq \alpha\leq k$, the anchor map $\sharp^k: E^k\ra \T^k M$ preserves the gradings on $E^k$ and $\T^k M$, {\new while } the vector field $\aliftB{s}{\beta}\in \VF(E^k)$, where $s\in \Sec(E^1)$, has weight $-k\leq \beta\leq 0$. Hence $\aliftB{s}{\beta}(f^{(\alpha)}) = 0$ whenever $\alpha+\beta<0$.


Clearly, the equation \eqref{e:sharpFv} is $\Cf(M)$-linear in $v$, so it is enough to show \eqref{e:sharpFv} for $v$ from a  frame $(c_\mu)$ of local sections of $C\ra M$.  We have {\new  $\ddot f = \frac{\pa f}{\pa x^a} \ddot{x}^a + \frac{\pa^2 f}{\pa x^a \pa x^b} \dot{x}^a \dot{x}^b$, so }
\begin{equation}\label{e:sharp_ddot}
  (\sharp^2)^\ast(\ddot{f}) = \frac{\pa f}{\pa x^a} \left(Q^a_\mu z^\mu + \frac12 Q^a_{ij} y^i y^j\right) +{\new \frac{\pa^2 f}{\pa x^a \pa x^b} Q^a_i Q^b_j y^i y^j,}
\end{equation}
hence {\new $\coreVF{c_\mu}((\sharp^2)^\ast(\ddot{f})) = Q^a_\mu   \frac{\pa f}{\pa x^a} = \sharp^C(c_\mu) (f)$, see \eqref{df:sharpC}, } and  \eqref{e:sharpFv} follows immediately.

 Similarly, {\new using \eqref{e:fs2}, we find that} \eqref{e:sharpFpa} and \eqref{e:sharp_lift1} are $\Cf(M)$-linear in $s$.
Thus it is enough to verify these equalities  for $s=e_k$. This is straightforward:  we use the formulas $\eqref{e:alg_lifts_coord}$ for $\aliftB{e_k}{\alpha}$ and find that $\frac12 \aliftB{e_k}{-2}((\sharp^2)^\ast \ddot f) = Q^{\mu}_kQ^a_{\mu} \frac{\pa f}{\pa x^a}$, which coincides with $(\sharp^C \circ \pa )(e_k)(f) = Q^\mu_k \sharp^C(c^\mu)(f)$ due to \eqref{e:coord_pa_beta_gamma}, thereby proving \eqref{e:sharpFpa}. Next, $\pullback{\sharp} \dot f = \frac{\pa f}{\pa x^a} Q^a_i y^i$, and all three expressions in  \eqref{e:sharp_lift1} are equal to $Q^a_k \frac{\pa f}{\pa x^a}$.

 It remains to prove \eqref{e:sharp_lift2ddf}.  We claim that the left and right hand sides of \eqref{e:sharp_lift2ddf} are $\Cf(M)$-linear in $s_1$, and their  difference is also $\Cf(M)$-linear in {\new $s_2$}. Indeed,  {consider \eqref{e:sharp_lift2ddf} with $s_1$ replaced with {\new $g s_1$} and expand $\aliftB{(g s_1)}{-1}$ as in \eqref{e:fs2}. Note that } $\aliftB{s_2}{0}(\pullback{\sharp} \dot f)$ has weight 1  and so it is killed by $\aliftB{s_1}{-2}$, hence $\aliftB{(g \,s_1)}{-1} \aliftB{s_2}{0} (\pullback{\sharp} \dot{f})  = g\, \aliftB{s_1}{-1} \aliftB{s_2}{0} (\pullback{\sharp} \dot{f})$ while
$$
    \sharp[gs_1, s_2] + \sharp s_2 \circ \sharp (g s_1) = \left(g  \sharp[s_1, s_2] - (\sharp s_2) (g) \, \sharp s_1\right) + \sharp s_2 (g)\, \sharp s_1)  + g \sharp s_2 \circ \sharp s_1=  g\, \cdot  \mathrm{RHS}_{\eqref{e:sharp_lift2ddf}}.
$$
Similarly, using point \eqref{e:sharp_lift1} and a weight argument, we obtain
\begin{align*}
   \aliftB{s_1}{-1} \aliftB{(g s_2)}{0} (\pullback{\sharp} \dot{f}) &=  \aliftB{s_1}{-1} \left(  g \aliftB{s_2}{0} +  \pullback{\sharp}(\dot{g})\aliftB{s_2}{-1} + \frac12 (\sharp^2)^\ast (\ddot{g}) \aliftB{s_2}{-2}\right)  (\pullback{\sharp} \dot{f}) = \\
  & (g \aliftB{s_1}{-1} \aliftB{ s_2}{0} +  \aliftB{s_1}{-1}(\pullback{\sharp} \dot{g}) \, \aliftB{s_2}{-1})(\pullback{\sharp} \dot{f}) = g \cdot \mathrm{LHS}_{\eqref{e:sharp_lift2ddf}} + (\sharp s_1)(g) \cdot (\sharp s_2)(f),
\end{align*}
and, in the same way, $\left(\sharp[s_1, gs_2] + \sharp (g s_2) \circ (\sharp s_1)\right)(f)  = g \cdot  \mathrm{RHS}_{\eqref{e:sharp_lift2ddf}} + (\sharp s_1)(g) \cdot (\sharp s_2)(f)$.  It proves our claim, and thus it is enough to check \eqref{e:sharp_lift2ddf} with $s_1 = e_{k'}$ and $s_2 = e_{k}$. We have
$$
\aliftB{e_k}{0} (\pullback{\sharp} \dot{f}) = \aliftB{e_k}{0}\left(\frac{\pa f}{\pa x^a} Q^a_i y^i\right) =  Q^b_k y^i \frac{\pa}{\pa x^b} (\frac{\pa f}{\pa x^a} Q^a_i) + Q^i_{jk}  Q^a_i y^j \frac{\pa f}{\pa x^a},
$$
hence $\pa_{z^\mu}$ kills above expressions. By applying $\aliftB{e_{k'}}{-1}$ we get
\begin{multline}
\aliftB{e_{k'}}{-1} \aliftB{e_k}{0} (\pullback{\sharp} \dot{f}) =  Q^b_k \frac{\pa}{\pa x^b} (\frac{\pa f}{\pa x^a} Q^a_{k'}) +  Q^i_{\kp k} Q^a_i \pa_{x^a} 
 = \sharp e_k (\sharp e_{k'} (f)) + (\sharp [e_{k'}, e_k])(f), 
\end{multline}
and we are done.
\end{proof}

\paragraph{Proof of Lemma~\ref{l:tensor_psi} and Theorem~\ref{th:skew_HA} part (a).}
    {\new The formulas $\frac12 \veps_1(s_1, s_2) = \beta(s_1, s_2) - \pa([s_1, s_2])$ and ${\newMR \frac12 } \veps_0(s_1, s_2) =  \Box_{s_1} (\pa s_2) - \pa ([s_1, s_2])$ come from the definitions of the corresponding maps, compare \eqref{df:eps_k} with \eqref{df:beta}, \eqref{df:pa} and \eqref{df:gamma}.  {\newMR The other properties of the maps  $\veps_0$ and $\veps_1$ } given in  Lemma~\ref{l:tensor_psi} follow immediately from the properties of the maps $\beta$ and $\Box$ given in Theorem~\ref{th:skew_HA}, which we are going to prove first.}
\begin{itemize}
\item Proof of \eqref{e:tensor_beta}: As the Lie bracket of vector fields is skew symmetric, so $\beta(s_1, s_2)= -\beta(s_2, s_1)$. There are no vector fields on $E^2$ of weight less than $-2$. {\newMR Hence,
    using \eqref{e:fs2}, we get }
\begin{multline}
{\beta}(s_1, f s_2)  = {\new \frac12 } [\aliftB{s_1}{-1}, f \aliftB{s_2}{-1}] + {\new \frac12 }[ \aliftB{s_1}{-1}, (\pullback{\sharp} \dot{f}) \aliftB{s_2}{-2}] = \\
 {\new \frac12 } f [\aliftB{s_1}{-1},  \aliftB{s_2}{-1}]
+  {\new \frac12 } \aliftB{s_1}{-1}(\pullback{\sharp} \dot{f}) \aliftB{s_2}{-2} = f {\beta}(s_1, s_2) + (\sharp s_1)(f) \pa(s_2),
\end{multline}
by a weight argument and Lemma~\ref{l:alifts_skew}\eqref{e:sharp_lift1}.

\item Proof of \eqref{e:gamma_1} and  \eqref{e:gamma_2}: We expand $\aliftB{(f s)}{0}$ as in \eqref{e:fs2}
 and find that for $v\in \VF_{-2}(E^2)\simeq \Sec(C)$ and $s\in \Sec(E)$ we have
\begin{align*}
\Box_{fs} v &= [f \aliftB{s}{0}, v] + {\new [(\pullback{\sharp} \dot{f}) \aliftB{s}{-1}, v] } + [{\new \frac12} (\sharp^2)^\ast \ddot{f}  \aliftB{s}{-2}, v]  = \\
       &= f [\aliftB{s}{0}, v] {\new -} v(\pullback{\sharp^2} \ddot{f}) \aliftB{s}{-2} \stackrel{\text{Lemma \eqref{l:alifts_skew}}}{=} f \cdot (\Box_s v) - (\sharp^C v)(f) \, \pa(s),
\end{align*}
as $v(f)$, $v(\pullback{\sharp}(\dot{f}))$,  $[\aliftB{s}{-1}, v]$, $[\aliftB{s}{-2}, v]$ vanish by inspecting weights.
 Similarly for \eqref{e:gamma_2}:
$$
\Box_s f v = [\aliftB{s}{0}, f v] = f [\aliftB{s}{0}, v] {\new + } \aliftB{s}{0}(f) v = f (\Box_s v) {\new +  }  (\sharp s)(f) v
$$
by Lemma~\ref{l:alifts_skew}\eqref{e:sharp_lift1}.

\item Proof of the properties of the map $\psi$ given in Lemma~\ref{l:tensor_psi}:

  Let  $h := \frac12 (\sharp^2)^\ast \ddot{f}$ for time being, so  $h$ is a function on $E^2$ of weight two.
      From \eqref{e:fs2} and the definition \eqref{df:psi} of $\psi$ 
      we get
      $$\left(\psi(g s_1, s_2) - g \psi(s_1, s_2)\right)(f) = {\newMR \frac12 }(\pullback{\sharp} \dot{g})\, \aliftB{s_1}{-2} \aliftB{s_2}{-1}(h) = 0$$
      by a weight argument, hence $\psi$ is tensorial in its first argument. Also
 \begin{multline*}\left(\psi(s_1, g s_2)- g \psi(s_1, s_2)\right)(f) =
 {\newMR \frac12 }\aliftB{s_1}{-1}\left(g\, \aliftB{s_2}{-1}(h) + \pullback{\sharp} \dot{g}\, \aliftB{s_2}{-2} (h)\right) - {\newMR \frac12 } g  \aliftB{s_1}{-1}\aliftB{s_2}{-1}(h)  \\ -  (\sharp s_1\, g)\, (\sharp s_2)(f) =
 {\newMR \frac12 } \aliftB{s_1}{-1}(\pullback{\sharp} \dot{g})\cdot \aliftB{s_2}{-2}(h) - (\sharp s_1)(g)\, (\sharp s_2)(f) =
  (\sharp s_1) (g)\left(\sharp^C\circ \pa - \sharp\right){\newMR (s_2)}(f),
 \end{multline*}
 by Lemma~\ref{l:alifts_skew}\eqref{e:sharpFpa} and \eqref{e:sharp_lift1}, hence we get \eqref{e:psi}.
    Set $A(f) = \aliftB{s_1}{-1} \aliftB{s_2}{-1} (h)$, $B(f) = (\sharp s_1)\circ (\sharp s_2)(f)$, so $\psi(s_1, s_2) = A -  B$. By inspecting  weights  and using $(fg)^{(2)} = f \ddot{g} + \ddot{f} g + 2 \dot{f}\dot{g}$ and Lemma~\ref{l:alifts_skew}\eqref{e:sharp_lift1} we find that
    $$
        A(fg) = A(f) g  +f A(g) + (\sharp s_1) (f) (\sharp s_2)(g) +  (\sharp s_1) (g) (\sharp s_2)(f),
    $$
    while
    $$
        B(fg) = f B(g) + B(f) g + (\sharp s_1) (f) (\sharp s_2)(g) +  (\sharp s_1) (g) (\sharp s_2)(f)
    $$
    hence $\psi(s_1, s_2)$ is a  derivation. The coordinate formula \eqref{e:psi_coord} for $\psi(e_{k'}, e_k)$ follows directly from \eqref{e:alg_lifts_coord} and \eqref{e:sharp_ddot}:
    \begin{align*}
        \aliftB{e_{k'}}{-1}  \aliftB{e_{k}}{-1} (h) &= \aliftB{e_{k'}}{-1}  \left(\frac12 Q^\mu_{ik} Q^a_\mu \frac{\pa f}{\pa x^a} y^i + \frac12 Q^a_{ik}  \frac{\pa f}{\pa x^a} y^i + \frac{\pa^2 f}{\pa x^a \pa x^b} Q^a_i Q^b_k y^i \right) = \\ &=\frac12\left(Q^\mu_{k' k}Q^a_\mu + Q^a_{k'k}\right) \frac{\pa f}{\pa x^a} +  \frac{\pa^2f}{\pa x^a \pa x^b} Q^a_{k'} Q^b_k, \\
        (\sharp e_{k'})(\sharp e_k)  (f) &=  Q^b_{k'} Q^a_k  \frac{\pa^2f}{\pa x^a \pa x^b} +  Q^b_{k'} \frac{\pa Q^a_k}{\pa x^b} \frac{\pa f}{\pa x^a}.
    \end{align*}
   The formula \eqref{e:psi_bar_coord} for $\sym{\psi}$ follows immediately from \eqref{e:psi_coord}.
   The skew-symmetric part of $\psi$ is derived from \eqref{df:psi}:
    $$\alt{\psi}(s_1, s_2)(f) = {\new \frac14} [\aliftB{s_1}{-1}, \aliftB{s_2}{-1}]((\sharp^2)^\ast \ddot{f}) - {\new \frac12} [\sharp s_1, \sharp s_2](f)$$
    and this coincides with the formula \eqref{e:psi_skew_coord} due to the definition of $\beta$ and Lemma~\ref{l:alifts_skew}\eqref{e:sharpFv}.   The direct computation of  $\sharp^2 \circ \Rk^2$  using \eqref{e:R2_coord}
    gives
    $$
    \sharp^2 \circ \Rk^2 (x^a, y^i, \dot{y}^i) = \left(x^a, \dot{x}^a = Q^a_i y^i, \ddot{x}^a = \frac12 Q^a_{ij} y^i y^j + Q^a_\mu\left(Q^\mu_i \dot{y}^i + \frac12 Q^\mu_{(ij)} y^i y^j\right)\right)
    $$
    and comparing it with
    $$\sharp^{[2]}(x^a, y^i, \dot{y}^i) = (x^a, Q^a_i y^i, \frac{1}{2}\Qhat{Q}^a_{ij} \, y^i y^j + Q^a_i\, \dot{y}^i)$$
    as read from Example~\ref{ex:structure_maps_E[2]},  gives the desired {\newMR equivalence:  $\sharp^2 \circ \Rk^2 = \sharp^{[2]}$ if and only if } $\sym{\psi}=0$ and $\sharp = \sharp^C \circ \pa$, see \eqref{df:Qhat} and \eqref{e:psi_bar_coord}.

   \item Proof of \eqref{e:tensor_bar_omega_1} and \eqref{e:tensor_bar_omega_2}:
The map $\delta$ defined in \eqref{df:delta} satisfies
\begin{equation}
            \delta_s(f s_1, s_2) = f \delta_s(s_1, s_2) +   (\sharp [s, s_2])(f)\, \pa(s_1), \label{e:delta_1} \tag{$\mathrm{Eq^1_\delta}$}
            \end{equation}
            \begin{equation}
            \delta_s(s_1,  fs_2) = f \delta_s(s_1, s_2) - (\sharp s)(f)\beta(s_1, s_2) { - (\sharp s_1)(f) \, \Box_{s} \pa(s_2) - (\sharp [s_1, s] + \sharp s \circ \sharp s_1)(f) \, \pa(s_2)},  \label{e:delta_2}\tag{$\mathrm{Eq^2_\delta}$}
            \end{equation}
            \begin{equation}
            \delta_{fs}(s_1, s_2) =  f \delta_s(s_1, s_2) + (\sharp s_1)(f)\beta(s_2, s) + (\sharp s_2)(f)\beta(s_1, s) + \left((\sharp s_1)\circ (\sharp s_2) +  \psi(s_1, s_2)\right)(f) \, \pa(s). \label{e:delta_3}\tag{$\mathrm{Eq^3_\delta}$}
\end{equation}
where $\delta_s(s_1, s_2) = \delta(s_1, s_2, s)$.
We expand $\aliftB{(f\,s)}{0}$  as in \eqref{e:fs2} and using Lemma~\ref{l:alifts_skew} we get
\begin{align*}
&[\aliftB{s_1}{-1}, [\aliftB{s_2}{-1}, f \aliftB{s}{0}]] = f [\aliftB{s_1}{-1},  [ \aliftB{s_2}{-1}, \aliftB{s}{0}]], \\
& [\aliftB{s_1}{-1}, [\aliftB{s_2}{-1}, \pullback{\sharp} \dot{f} \aliftB{s}{-1}]]  =
  [\aliftB{s_1}{-1}, 2\pullback{\sharp} \dot{f} \beta(s_2, s) +  \sharp s_2(f)\, \aliftB{s}{-1}] =  2 \sharp s_1(f) \beta(s_2, s) +  2   \sharp s_2(f) \beta(s_1, s), \\
  & [\aliftB{s_1}{-1}, [\aliftB{s_2}{-1}, (\frac12 \pullback{(\sharp^2)} (\ddot{f})) \aliftB{s}{-2}]] =  \aliftB{s_1}{-1} \aliftB{s_2}{-1}\frac12 \sharp_2^\ast (\ddot{f}) \aliftB{s}{-2} =  \left(\psi(s_1, s_2)(f) + (\sharp s_1)(\sharp s_2)(f)\right)\, {\new } 2 \pa(s),
\end{align*}
where $\psi$ is defined in \eqref{df:psi}. Summing up these three equalities, we get  \eqref{e:delta_3}.
The equalities \eqref{e:delta_1} and \eqref{e:delta_2} can be derived in a very similar way and we omit the proof.
 The direct use of the definition of $\omega$ (see \eqref{df:omega}) and the properties of $\delta$ and $\beta$
    lead to {\newMR  
           \begin{align*}
                \omega_s(f s_1,  s_2) &= f \omega_s (s_1, s_2),  \\
                \omega_s(s_1,  f s_2) &-  f \omega_s(s_1, s_2) =  (\sharp s_1)(f)\left(\pa([s, s_2] - \Box_s \pa(s_2))\right) + \left(\sharp[s, s_1] - [\sharp s, \sharp s_1]\right)(f) \pa(s_2) =  \\
                &- (\sharp s_1)(f) \veps_0(s, s_2) + \xi(s, s_1) \pa s_2, \\
                \omega_{fs}(s_1, s_2) &-  f \omega_s(s_1, s_2) =  (\sharp s_1)(f)\left(\pa([s, s_2])-\beta(s, s_2)\right) +  \psi(s_1, s_2)(f) \, \pa(s) = \\
                &\frac12  (\sharp s_1)(f) \veps_1(s_2, s) + \psi(s_1, s_2)(f) \, \pa(s),
            \end{align*}
            }
where $\omega_s(s_1, s_2) = \omega(s_1, s_2, s)$ and $\veps_0$, $\veps_1$, $\xi$  and $\psi$ are as in Definition~\ref{df:epss}. From this, the equations   \eqref{e:tensor_bar_omega_1} and \eqref{e:tensor_bar_omega_2} follow immediately.
   \end{itemize}
\qed

\paragraph{Proof of Theorem~\ref{thm:HA_point}.}
 \label{proof:HA_point}
First, we shall describe the structure of the Lie algebra $\VF_{\leq 0}(\g \times C)$ with respect to the decomposition given in \eqref{e:VF_leq2}.
We write $\phi \oplus \psi \oplus \chi \in \VF_0$, $x\oplus f \in \VF_{-1}$, $v\in \VF_{-2}$, { where } $\phi\in \End(\g)$, $\psi\in \End(C)$, $\chi\in \Hom(\Sym^2 \g, C)$, $x\in \g$, $f\in \Hom(\g, C)$, and $v\in C$. We have

{\scriptsize
\begin{equation}\label{e:VF_alg}
\begin{split}
[\VF_0, \VF_0]: & [\phi_1, \phi_2] = \phi_2\circ \phi_1 - \phi_1\circ \phi_2,   [\psi_1, \psi_2] =\psi_2\circ \psi_1 - \psi_1\circ \psi_2,  [\chi_1, \chi_2]= 0 \\
& [\phi, \chi](x_1, x_2) = \chi(\phi(x_1), x_2) + \chi(x_1, \phi(x_2)),   [\psi, \chi](x_1, x_2) = - \psi(\chi(x_1, x_2)),  [\phi, \psi] = 0,  \\
[\VF_0, \VF_{-1}]: & [\phi, x]= - \phi(x), [\phi, f] = f\circ \phi, [\psi, x] =0, [\psi, f] = - \psi\circ f, [\chi, x] = - \chi(x, \cdot), [\chi, f] = 0\\
[\VF_{-1}, \VF_{-1}]: & [x_1, x_2]=0, [f_1, f_2]=0, [f, x]=-f(x), \\
[\VF_0, \VF_{-2}]: & [\phi, v] = 0, [\psi, v] = -\psi(v), [\chi, v] = 0.
\end{split}
\end{equation}
}

{\newMR For example, the formula for $[\phi, x]$ can be derived as follows.  A vector $x= (x^i) \in \g$ is idenified with the vector field $x^i \pa_{y^i}$, and an endomorphism $\phi: \g\to \g$, such that $\pullback{\phi}(y^i) =  \phi^i_j y^j$, is identified with the vector field $\phi = \phi^i_j y^j \pa_{y^i}$. Then, $[\phi, x] = - \phi^i_j x^j \pa_{y^i} =  - \phi(x)$. In a similar way we derive the remaining formulas.}

The formulas for algebroid lifts $\aliftB{e}{\K}$, where $\K=-2, -1, 0$, given in the formulation of our theorem, define vector fields  which have the form as in \eqref{e:alg_lift_HA} since the projection of $\aliftB{e}{-1}$ onto $\g$ is $e$ and the bracket $[\cdot, \cdot]$ is skew-symmetric. Therefore,  these vector fields define an AL higher algebroid.

 Conversely,  let $(\g\times C, \kappa^2)$ be  a skew HA  defined by means of algebroid lifts $\aliftB{e}{\K}$ given above. Let us temporarily denote by $\wt{\pa}$, $\wt{\beta}$, $\wt{\Box}$, $\wt{\sym{\omega}}$ the maps associated with $\kappa^2$, defined in  Subsection~\ref{sSec:HA_order_two} by formulas {\new \eqref{df:pa}, } 
 \eqref{df:beta}, \eqref{df:gamma}, \eqref{df:omega_sym}, respectively. We shall show that $\wt{\pa}=\pa$, $\wt{\beta}=\beta$, $\wt{\Box}=\Box$, $\wt{\sym{\omega}}=\sym{\omega}$.

The definitions of $\wt{\pa}$ and $\pa$ coincide, hence $\wt{\pa} = \pa$. For the proof of the equality $\wt{\beta}  =\beta$  we have
$$
\wt{\beta}(x, y) = \frac12 [\aliftB{x}{-1}, \aliftB{y}{-1}] = \frac12 [x\oplus
\beta(\cdot, x), y \oplus \beta(\cdot, y)] =  \beta(x, y)
$$
due to the skew-symmetry of $\beta$ and the formulas \eqref{e:VF_alg} for
the bracket on $\VF_{-1}\oplus \VF_{-1}$. For $\wt{\Box} = \Box$ we write
$$
{\new \wt{\Box}_x v = [\aliftB{x}{0}, v]= [\underbrace{[\cdot, x]}_{\phi} \oplus
\underbrace{\Box_{-x} (\cdot)}_{\psi} \oplus \underbrace{2 \sym{\omega}_x(\cdot, \cdot)}_{\chi}, v] =   -\psi(v) = \Box_x v},
$$
due to the formulas for the bracket restricted to $\VF_{-2}\oplus
\VF_{0}$. The proof of $\wt{\delta} = \delta$ is a bit longer. First, we calculate
\begin{equation}\label{e:10}
\begin{split}
[\aliftB{y}{-1}, \aliftB{z}{0}] &= [
y\oplus  \beta(\cdot, y), [\cdot, z] \oplus
\Box_{-z} (\cdot) \oplus  2 \sym{\omega}_z(\cdot, \cdot)] =  \\
 & \underbrace{[y, z]}_X \oplus  \underbrace{ \beta([\cdot,  -z], y) + {\new
\Box_{-z} \beta(\cdot, y)} + {\new 2 \sym{\omega}_z(y, \cdot)}}_F,
\end{split}
\end{equation}
hence
\begin{align*}
\wt{\delta}(x, y, z) &= {\new \frac12} [\aliftB{x}{-1}, [\aliftB{y}{-1}, \aliftB{z}{0}]] =  \frac12 [x \oplus \beta(\cdot, x), X\oplus F] =  \frac12 F(x)
 - \frac12 \beta(X, x) = \\
& - \frac12 \beta([y, z], x) + \frac12 \beta([z, x], y) -   \frac12 \Box_z{\beta(x, y)}  +
\sym{\omega}_z (y, x).
\end{align*}

{\new Thus $\wt{\sym{\omega}_x}(y, z)$, which is obtained by symmetrizing $\wt{\delta}(x,  y, z) - \beta(x, [y, z])$ in $x, y$, coincides, due to the skew-symmetry of $\beta$,  with the symmetrization of $\sym{\omega}_z(x, y) + \frac12 \beta([y, z], x) + \frac12 \beta([z, x], y)$, and the latter  simplifies to $\sym{\omega}_z(x, y)$, as was claimed. }

We  now examine the Lie condition for HAs given in Remark~\ref{r:Lie_axiom}.
From \eqref{e:VF_alg} we get
\begin{multline*}
[\aliftB{x}{0}, \aliftB{y}{0}] =  [\underbrace{[-x, \cdot]}_{\phi_1} \oplus \underbrace{\Box_{-x} (\cdot)}_{\psi_1} \oplus \underbrace{2  \sym{\omega}_x(\cdot, \cdot)}_{\chi_1}, \underbrace{[-y, \cdot]}_{\phi_2} \oplus  \underbrace{\Box_{- y} (\cdot)}_{\psi_2} \oplus \underbrace{2 \sym{\omega}_y(\cdot, \cdot)}_{\chi_2}] = \\
 \left([y, [x, \cdot ]] - [x, [y, \cdot]] \right) \oplus
 [\Box_y , \Box_x]    \oplus   \chi,
\end{multline*}
for some $\chi\in \Hom(\Sym^2 \g, C)$. From the  condition $[\aliftB{x}{0}, \aliftB{y}{0}]  =  \aliftB{[x, y]}{0}$ we read that
$\g$ is a Lie algebra, $C$ is a left $\g$-module. We shall show that $\sym{\omega} = 0$, from which it follows that the identity obtain by comparing the $\Hom(\Sym^2 \g, C)$-components is satisfied automatically.

The equation  $[\aliftB{x}{-1}, \aliftB{y}{-1}] = \aliftB{[x, y]}{-2}$ and $[\aliftB{x}{0}, \aliftB{y}{-2}] = \aliftB{[x, y]}{-2}$ write as
\begin{equation}\label{e:tmp1}
\Box_x \pa(y) = \pa([x, y]) \text{ and  } \beta(x, y) = \pa([x, y]).
\end{equation}
Finally, for $(i, j) = (0, -1)$, the Lie condition from \eqref{e:10} is given by
$$
\beta(\cdot, [x, y]) = \beta([\cdot,  -x], y) +
\Box_{-x} \beta(\cdot, y) +  \sym{\omega}_x(y, \cdot).
$$
This, along with the equalities in \eqref{e:tmp1} and the Jacobi identity, yields $ \sym{\omega} = 0$, and completes the proof.
\qed

\subsection{Equations for AL and Lie HAs} \label{sSec:AL_and_Lie_HA_eqns}

We shall use Theorem~\ref{th:HA_axioms_and_lifts} to write equations for structure functions corresponding to almost Lie {\new HAs.}
 The obtained equations will be  used to complete the proof of Theorem~\ref{th:AL_HA_str_maps}. 

\paragraph{AL HAs.} Let $(E^2, \kappa^2)$ be an AL HA and let $(e_k)$, $(c_\mu)$ be as in Subsection~\ref{sSec:HA_order_two}. The vector fields
$\aliftB{e_k}{\K}\in \VF_\K(E^2)$ for $\K = 0, -1, -2$ are given in \eqref{e:alg_lifts_coord}.
The formulas for $(\T^2 M, \kappa^2_M)$-algebroid lifts $\aliftB{\und{e}_k}{\K} := \aliftB{(\sharp e_k)}{\K} = \aliftB{(Q^a_k \pa_{x^a})}{\K}
\in \VF_\K(\T^2 M)$ are easily derived from {\new \eqref{e:fs_lift}} {\new by } notting that $\aliftB{\pa_{x^a}}{\K}$ is equal to $\pa_{x^a}$, $\pa_{\dot{x}^a}$ and $2 \pa_{\ddot{x}^a}$ for $\K=0, -1, -2$, respectively. Thus
\begin{equation}\label{e:alifts_T2M}
\begin{cases}
\aliftB{\und{e}_k}{0} &= Q^a_k \pa_{x^a} + \frac{\pa Q^a_k}{\pa {x^b}} \dot{x}^b \pa_{\dot{x}^a} + \left(\frac{\pa Q^a_k}{\pa x^b} \ddot{x}^b + \frac{\pa^2Q^a_k}{\pa {x^b} \pa {x^c}} \dot{x}^b \dot{x}^c \right) \pa_{\ddot{x}^a}, \\
\aliftB{\und{e}_k}{-1} &=  Q^a_k \pa_{\dot{x}^a} + {\new 2} \frac{\pa Q^a_k}{\pa x^b} \dot{x}^b \pa_{\ddot{x}^a}, \\
\aliftB{\und{e}_k}{-2} &=  2 Q^a_k \, \pa_{\ddot{x}^a}.
\end{cases}
\end{equation}
(Note that $(Q^a_k)^{(1)} = \frac{\pa Q^a_k}{\pa {x^b}} \dot{x}^b $ and
$(Q^a_k)^{(2)} =  \frac{\pa Q^a_k}{\pa x^b} \ddot{x}^b + \frac{\pa^2Q^a_k}{\pa {x^b} \pa {x^c}} \dot{x}^b \dot{x}^c$. The above formulas for $\aliftB{\und{e}_k}{\K}$ can also be obtained from  \eqref{e:alg_lifts_coord} and \eqref{e:local_kappa_E2}.)
We check whether vector fields $\aliftB{e_k}{\K}$ and $\aliftB{\und{e}_k}{\K}$ are $\sharp^2$-related, where $\sharp^2(x^a, y^i, z^\mu) = (x^a, \dot{x}^a = Q^a_k y^k, \ddot{x}^a = Q^a_\mu z^\mu + \frac12 Q^a_{ij} y^i y^j)$. Straightforward calculations leads to the following system of equations {\new (referred to as AL HA equations)}:
\begin{subnumcases}{}
Q^a_\mu Q^\mu_k  =  Q^a_k   \label{e:QamuQmu_i} \\
Q^a_{kj} + Q^a_\mu Q^\mu_{jk}  = 2\,\frac{\pa Q^a_k}{\pa x^b} Q^b_j    \label{e:Qa_kj} \\
Q^a_k Q_{ij}^k = \Qcheck{Q}^a_{ij} \label{e:QakQk_ij} \\
Q^a_\nu Q^\nu_{\mu i} + \frac{\pa Q^a_{\mu}}{\pa x^b} Q^b_i = \frac{\pa Q^a_i}{\pa x^b} Q^b_\mu   \label{e:QanuQnu_mui} \\
  Q^a_\mu Q^\mu_{ij,k} +  \frac{\pa Q^a_{ij}}{\pa x^b} Q^b_k + Q^a_{li} Q^l_{jk} + Q^a_{lj} Q^l_{ik} =   \frac{\pa Q^a_k}{\pa x^b} Q^b_{ij} + 2\, \frac{\pa^2 Q^a_k}{\pa x^b \pa x^c} Q^b_i Q^c_j  \label{e:QamuQmu_ijk}
\end{subnumcases}
where
\begin{equation}\label{df:Qcheck} 
\Qcheck{Q}^a_{ij} := Q^b_i \frac{\pa Q^a_j}{\pa x^b}- Q^b_j \frac{\pa Q^a_i}{\pa x^b},
\end{equation}
{\newMR (The equations \eqref{e:QamuQmu_i},  \eqref{e:Qa_kj} correspond to the cases $\K= -2$ and $\K=-1$, respectively; while  \eqref{e:QakQk_ij} \eqref{e:QanuQnu_mui} and \eqref{e:QamuQmu_ijk} correspond to the case $\K=0$. Note also that the equations \eqref{e:QamuQmu_i}, \eqref{e:Qa_kj},  \eqref{e:QakQk_ij}, and \eqref{e:QanuQnu_mui}  follows immadiately from \eqref{i:AL_sharpEF}, \eqref{i:AL_psi}, \eqref{i:E1}, and \eqref{i:AL_nabla}, respectively.)}
Note that
\begin{equation}\label{e:Qcheck}
    [\sharp e_i, \sharp e_j] =  \Qcheck{Q}^a_{ij} \pa_{x^a}.
\end{equation}
The equation  \eqref{e:Qa_kj} can be replaced with
\begin{subnumcases}{}
 Q^a_\mu Q_{[ij]} = \Qcheck{Q}^a_{ij}, \label{e:QamuQmu_ij_skew}\\
Q^a_\mu Q^\mu_{(ij)} + Q^a_{ij} = \wh{Q}^a_{ij} \label{e:QamuQmu_ij_symm}
\end{subnumcases}
where
$\Qhat{Q}^a_{ij}$ is given in \eqref{df:Qhat}.

\paragraph{Completion of the proof of Theorem~\ref{th:AL_HA_str_maps}.} \label{p:completion_thm_K=0}
We shall prove that if a skew HA $(E^2, \kappa^2)$ satisfy the conditions 
listed in Theorem~\ref{th:AL_HA_str_maps},
then it is almost Lie. It amounts to proving
that the vector fields $\alift{e_k}{\K}$ and $\alift{\und{e}_k}{\K}$, see \eqref{e:alg_lifts_coord} and \eqref{e:alifts_T2M},  are $\sharp^2$-related for $\K=-2, -1, 0$. We have already proved this for $\K=-2, -1$, so it remains to prove this for $\K = 0$, i.e.,  to verify the equations (\ref{e:QakQk_ij},\ref{e:QanuQnu_mui} \ref{e:QamuQmu_ijk}).

The equation \eqref{e:QakQk_ij} {\new  means $\sharp [e_i, e_j] = \Qcheck{Q}^a_{ij} \pa_{x^a}$, which is true since the algebroid $(E^1, \kappa^1)$ is AL. Next, the condition \eqref{e:QanuQnu_mui} means $\sharp^C \Box_{e_i}c_\mu = [\sharp e_i,  \sharp^C c_\mu]$, and it follows }
from \eqref{i:AL_nabla}.

The proof of \eqref{e:QamuQmu_ijk}   is a bit more involved.
We {\new claim} that
\begin{equation}\label{e:HA_rhoF_delta_symm}
    \sharp^C \circ \sym{\delta} = \sym{\delta}_{\T^2 M} \circ \sharp^{\times 3}: A\times A\times A\ra \T M,
\end{equation}
where $\sym{\delta}$ is given in \eqref{df:delta_sym} and $\sym{\delta}_{\T^2 M}$ is the same structure map 
but  associated with the HA $(\T^2 M, \kappa^2_M)$. Indeed, $\sharp^C\circ \omega =  0 $ implies $\sharp^C\circ \sym{\omega} =  0 $ and from $\sym{\delta} = \sym{\omega} + \frac12 (\beta(s_1, [s_2, s]) + \beta(s_2, [s_1, s]))$ we find that {\new
$$
    \sharp^C\circ \sym{\delta}_s (s_1, s_2) = \frac12 \sharp^C\circ \left(\beta(s_1, [s_2, s]) + \beta(s_2, [s_1, s])\right) \stackrel{\eqref{i:AL_beta}}{=} \sharp ([[s, s_1], s_2]+ [[s, s_2], s_1]) =  \left(\sym{\delta}_{\T^2 M}\right)_{\sharp s}(\sharp s_1, \sharp s_2).
$$
}
We shall show that \eqref{e:HA_rhoF_delta_symm}
gives \eqref{e:QamuQmu_ijk}.
 We shall { work } with an adapted coordinate system $(x^a, y^i , w^\mu)$ for $(E^2, \kappa^2)$ (see Definition~\ref{df:adapted_coord}), so $Q^\mu_{(ij)} =0$.  The general idea is to express $\sym{\delta}_{e_k}(e_i, e_j)$ entirely in terms of the structure functions $Q^a_i$ and its derivatives and  then compare with $\left(\sym{\delta}_{\T^2 M}\right)_{\sharp e_k}(\sharp e_i, \sharp e_j)$,
which is easily seen to be of this form. From the expression for $\tilde{Q}^\mu_{ijk}$ in \eqref{e:delta_coord}
we find that
$$
    \sym{\delta}_{e_k}(e_i, e_j) =  \frac12\left( Q^\mu_{ij, k} - Q^l_{ik}Q^\mu_{lj} - Q^l_{jk}Q^\mu_{li} \right) c_\mu,
$$
hence
$$
2 \cdot \sharp^C \circ \sym{\delta}_{e_k} (e_i, e_j) = Q^a_\mu \left( Q^\mu_{ij, k} - Q^l_{ik}Q^\mu_{lj} - Q^l_{jk}Q^\mu_{li} \right) \pa_{x^a}.
$$
We replace $Q^a_\mu Q^l_{jk}Q^\mu_{li}$ with
$$
  Q^a_\mu Q^l_{jk}Q^\mu_{li}   \stackrel{\eqref{e:Qa_kj}}{=} Q^l_{jk}(-Q^a_{li} + 2 Q^b_l \frac{\pa Q^a_i}{\pa x^b}) \stackrel{\eqref{e:QakQk_ij}}{=}  - Q^l_{jk} Q^a_{li} + 2 \Qcheck{Q}^b_{jk} \frac{\pa Q^a_i}{\pa x^b}
$$
and similarly for  $Q^a_\mu Q^l_{ik}Q^\mu_{lj}$ and get
\begin{equation}\label{e:tmp2}
  2 \cdot \sharp^C \circ \sym{\delta}_{e_k}(e_i, e_j) = (Q^a_\mu Q^\mu_{ij, k} + Q^l_{jk} Q^a_{li} +  Q^l_{ik} Q^a_{lj}- 2 (\Qcheck{Q}^b_{jk} \frac{\pa Q^a_i}{\pa x^b} + \Qcheck{Q}^b_{ik} \frac{\pa Q^a_j}{\pa x^b})) \pa_{x^a}
\end{equation}
As $Q^a_{ij}= \wh{Q}^a_{ij}$ by \eqref{e:QamuQmu_ij_symm}  the condition \eqref{e:QamuQmu_ijk} can be equivalently written as
$$
2\cdot \sharp^C \circ \sym{\delta}_{e_k}(e_i, e_j) = \left( 2 \frac{\pa^2 Q^a_k}{\pa x^b \pa x^c} Q^b_i Q^c_j + \wh{Q}^b_{ji} \frac{\pa Q^a_k}{\pa x^b}  - Q^b_k \frac{\pa \wh{Q}^a_{ij}}{\pa x^b}- 2 (\Qcheck{Q}^b_{jk} \frac{\pa Q^a_i}{\pa x^b} + \Qcheck{Q}^b_{ik} \frac{\pa Q^a_j}{\pa x^b})\right) \pa_{x^a},
$$
It remains to show that the last expressions coincides with {  $[\sharp e_i, [\sharp e_j, \sharp e_k]] + [\sharp e_j, [\sharp e_i, \sharp e_k]]$.  This a direct calculation of the brackets of vector fields. Namely, from \eqref{e:Qcheck} we get
$$
[\sharp e_i, [\sharp e_j, \sharp e_k]] = \left(Q^b_i \frac{\pa \Qcheck{Q}^a_{jk}}{\pa x^b} - \Qcheck{Q}^b_{jk} \frac{\pa Q^a_i}{\pa x^b}\right)\pa_{x^a}.
$$
On the other hand,  the following identity holds
$$
 2 \frac{\pa^2 Q^a_k}{\pa x^b \pa x^c} Q^b_i Q^c_j + \wh{Q}^b_{ji} \frac{\pa Q^a_k}{\pa x^b}  - Q^b_k \frac{\pa \Qhat{Q}^a_{ij}}{\pa x^b} =  Q^b_i \frac{\pa \Qcheck{Q}^a_{jk}}{\pa x^b} + \Qcheck{Q}^b_{jk} \frac{\pa Q^a_i}{\pa x^b}+
 Q^b_j \frac{\pa \Qcheck{Q}^a_{ik}}{\pa x^b} + \Qcheck{Q}^b_{ik} \frac{\pa Q^a_j}{\pa x^b}
$$
which can be  verified by expanding  $\Qhat{Q}^a_{ij}$ and $\Qcheck{Q}^a_{jk}$ using \eqref{df:Qhat} and \eqref{df:Qcheck}, and then grouping and cancelling similar terms.
On the LHS is the part of the expression \eqref{e:tmp2}  involving second derivatives. After plugging  the  RHS to \eqref{e:tmp2} we shall easily recognize the desired formula.
}
\qed

\paragraph{Lie HAs.}
{\new For completeness we provide a system of equations ensuring that a given AL HA is Lie.
They are obtained by examining the conditions listed  in Theorem~\ref{th:Lie_axiom_maps} on local frames $(e_i)$ and $(c_\mu)$ of the VBs $A$ and $C$, respectively, see Remark~\ref{r:LieHA:axioms}. The can be also obtained by examining the condition given in Remark~\ref{r:Lie_axiom}.} 

\begin{subnumcases}{}
\sum_{\text{cyclic } i,j,k} Q_{ij}^l Q_{lk}^m = 0, \label{e:coord_Jac} \\
Q^\mu_{\nu k} Q^\nu_i =  Q^\mu_l Q^l_{ik} + Q^a_k \frac{\pa Q^\mu_i}{\pa x^a}, \label{e:coord_pa} \\
Q^\mu_{[ik]} = Q^\mu_j Q^j_{ik}, \label{e:coord_beta} \\
Q^\mu_{\ip i, k} = Q^j_{ik} Q^\mu_{j\ip} - Q^\nu_{i \ip} Q^\mu_{\nu k} - Q^j_{k \ip} Q^\mu_{ij} +
Q^a_k \frac{\pa Q^\mu_{i \ip}}{\pa x^a}, \label{e:coord_A_on_F} \\
Q^j_{k\kp} Q^\mu_{\nu j} + \frac{\pa Q^j_{k\kp}}{\pa x^a} Q^a_\nu Q^\mu_j = Q^a_k \frac{\pa Q^\mu_{\nu \kp}}{\pa x^a} + Q^\rho_{\nu k} Q^\mu_{\rho \kp} -
Q^a_{\kp} \frac{\pa Q^\mu_{\nu k}}{\pa x^a} + Q^\rho_{\nu \kp} Q^\mu_{\rho k}\label{e:coord_omega}
\end{subnumcases}
{\newMR The equation \eqref{e:coord_Jac} corresponds to the Jacobi identity, while   \eqref{e:coord_pa},  \eqref{e:coord_beta},  \eqref{e:coord_A_on_F}, and \eqref{e:coord_omega} correspond to \eqref{i:Lie_ax:pa}, \eqref{i:Lie_ax:beta}, \eqref{i:Lie_ax:omega}, and \eqref{i:Lie_ax:A_on_C}, respectively.
}

\paragraph{Proof of Conjecture~\ref{conj:Rk} in the case $k=2$.} \label{p:conj:Rk}
{
Let $(e_i)$ be a local basis of sections of $\sigma: A \ra M$ and denote  $E_i^{[\K]} := e_i^{\langle \K\rangle_{\kappa^{[2]}}}$ -- the $(\At{2}, \kappa^{[2]})$--algebroid lifts, $\alpha=0, -1, 2$.
Using \eqref{e:alg_lifts_coord} and Example~\ref{ex:structure_maps_E[2]} we find that
\begin{equation}\label{e:lifts_E[2]}
\begin{cases}
E_k^{[0]} &= Q^a_k \pa_{x^a} + Q_{jk}^i y^j \pa_{y^i} + (Q^m_{n k} \, \dot{y}^n + \frac{1}{2} \, \wh{Q}_{ij, k}^m \, y^i y^j)\,\pa_{{\dot{y}}^m}, \\
E_k^{[-1]} &=  \pa_{y^k} +  Q_{ik}^m\, y^i\,\pa_{{\dot{y}}^m}, \\
E_k^{[-2]} &=  2 \pa_{{\dot{y}}^k}
\end{cases}
\end{equation}
 It remains to show that the vector fields $E_k^{[\alpha]}\in \VF_{\alpha}(\At{2})$ and $\aliftB{e_k}{\alpha}\in \VF_{\alpha}(E^2)$, given in \eqref{e:alg_lifts_coord}, are $\Rk^2$-related.  As $\Rk^1$ is the identity on $A$ we only need to show that
\begin{equation} \label{e:R-related_VF}
(\Rk^2)^\ast \aliftB{e_k}{\alpha} (z^\mu) = E_k^{[\alpha]}((\Rk^2)^\ast z^\mu) =  E_k^{[\alpha]}(Q^\mu_i \dot{y}^i + \frac12 Q^\mu_{(ij)} y^i y^j)
\end{equation}
for $\alpha = 0, -1, -2$ where $\Rk^2: \At{2} \ra E^2$ is given in \eqref{e:R2_coord}.
For $\alpha  = -2$, this equation is satisfied automatically. For $\alpha =-1$, it results in equation \eqref{e:coord_beta}. For $\alpha = 0$, the equation \eqref{e:R-related_VF} can be expressed  as a combination of \eqref{e:coord_pa} and \eqref{e:coord_A_on_F}.
}
\qed

\subsection{Representations up to homotopy of Lie algebroids} \label{sSec:representations}

We shall review some calculus and sign conventions concerning representations up to homotopy of Lie algebroids. We follow the presentation given in \cite{AbCr2012}.

Let $(\sigma: A\ra M, [\cdot, \cdot], \sharp)$ be a Lie algebroid. Then $\Omega(A) = \Sec(\bigwedge^\cdot A^\ast)$  is known as  \emph{the algebra of $A$-differential forms}. In the case the tangent algebroid, $A= \T M$, it is simply the algebra of differential forms on the manifold $M$. There is an algebroid de Rham differential, called $A$-differential $d_A$, which is a derivation of $\Omega(A)$ such that
\begin{enumerate}[(i)]
  \item  $d_A(f): s\mapsto (\sharp s)(f)$,  for   $s\in \Sec(A)$, $f\in \Cf(M)= \Omega^0(A)$;
   \item $d_A(\omega): (s_1, s_2) \mapsto \omega([s_1, s_2]) - (\sharp s_1)\left( \langle \omega, s_2\rangle \right)+ (\sharp s_2)\left(\langle \omega, s_1\rangle\right)$,   where  $s_1, s_2\in \Sec(A)$, $\omega\in \Sec(A^\ast) = \Omega^1(A)$.
\end{enumerate}
It is well known that a Lie algebroid can be equivalently described by means of $d_A$ --- a degree 1 {derivation} on $\Omega(A)$, see \cite{Vaintrob_1997}.

Let $F$ be a vector bundle over the same base $M$. An $A$-connection on $F$ is a mapping $\nabla: \Sec(A) \times \Sec(F) \ra \Sec(F)$, $(s, v)\mapsto \nabla_s v$ such that
$$
    \nabla_{f s} v = f \nabla_s v, \quad \nabla_s (f v) = f \nabla_s v + (\sharp s)(f)(v)
$$
for $f\in \Cf(M), v\in \Sec(F), s\in \Sec(A)$.
Recall that \emph{the curvature} of an $A$-connection $\nabla$ on $F$ is the tensor given by
\begin{equation} \label{df:curv}
    \curv_\nabla(s_1, s_2)(v) = \nabla_{s_1} \nabla_{s_2} v -  \nabla_{s_2} \nabla_{s_1} v - \nabla_{[s_1, s_2]}v,
\end{equation}
where $s_1, s_2\in \Sec(A)$, $v\in \Sec(F)$. \commentMR{Czyli jak u \cite{AbCr2012}.}

The space of \emph{$F$-valued $A$-differential forms} is defined as $\Omega(A; F) = \Sec(\bigwedge A^\ast \otimes F)$. 
In the setting of representations u.t.h., the vector bundle $F$ is $\Z$-graded, i.e.,  $F = \bigoplus_{i\in \Z} F^i$, where $F^i$ is, so called, the vector bundle of \emph{homogenous vectors of degree $i$}. 
Given graded vector bundles $F$, $G$ over the same manifold $M$, let $\und{\Hom}(F, G) = \bigoplus_{k\in \Z} \und{\Hom}^k(F, G)$, where $\und{\Hom}^k(F, G)$ denotes the  bundle 
homomorphism from $F$ to $G$ that increase the degree by
$k$. In other words, the fiber $\left(\und{\Hom}^k(F, G)\right)_x$ over $x\in M$ is a collection of linear maps $T_i: F^i_x \ra G^{i+k}_x$. In the special case $F=G$ we write $\und{\End}(F)$ for $\und{\Hom}(F, F)$.
An element  $\omega\in \Omega^i(A; F^j)$ is said to be of total degree $|\omega| = i + j$. There is an important operation, called
\emph{the wedge product}
$$
    \Omega^p(A; E^i) \otimes \Omega^q(A; F^j) \to \Omega^{p+q}(A; G^{i+j}), \quad (\alpha, \beta) \mapsto \alpha \wedge_h \beta,
$$
associated with a degree preserving graded vector bundle morphism $h: E\otimes F\to G$. It is given by
\begin{equation}\label{e:wedge_product}
(\alpha \wedge_h \beta)(s_1, s_2, \ldots, s_{p+q}) = \sum_\sigma (-1)^{qi} \sgn (\sigma) h(\alpha(s_{\sigma(1)}, \ldots, s_{\sigma_p}), \beta(s_{\sigma(k+1)}, \ldots, s_{\sigma(p+q)})),
\end{equation}
where the summation is over all $(p+q)$-shuffles,  $s_1, s_2, \ldots, s_{p+q}\in \Sec(A)$.
The left  $\Omega(A)$-module structure on $\Omega(A; F)$ is given by the wedge product associated with the
isomorphism $h_L: \R \otimes F \xrightarrow{\simeq} F$ and is denoted by $\omega.\eta := \omega \wedge_{h_L} \eta$.
On the other hand, the isomorphism $h_R: F \otimes \R \xrightarrow{\simeq} F$ gives rise to the right $\Omega(A)$-module structure on $\Omega(A; F)$, $\eta.\omega = \eta \wedge_{h_R} \omega$ that makes $\Omega(A; F)$ a symmetric $\Omega(A)$-bimodule,
$$
    \omega.\eta = (-1)^{|\omega| |\eta|} \eta.\omega,
$$
thanks to the sign $(-1)^{qi}$ in \eqref{e:wedge_product}.
\commentMR{
Rachunek: $(-1)^{qp +jp} (-1)^{p(q+j)}  = 1$, $\omega \in \Omega^p(A; \R)$, $(-1)^{qp}$ to znak permutacji $q+1, q+2, \ldots q+p, 1, 2, \ldots, q$, a $(-1)^{p(q+j)} = (-1)^{|\omega| |eta|}$. }
We shall  frequently encounter the case of the wedge product associated with the composition of homomorphisms $\circ: \Hom(G, H) \otimes \Hom(F, G) \ra \Hom(F, H)$   which will be denoted by
$\alpha \comp \beta := \alpha\wedge_\circ \beta$.  \commentMR{Czy nie zamienic $\comp$ na $\circ$?} We have
\begin{equation}\label{e:comp_and_dot}
    (\alpha \comp \beta)\comp \gamma = \alpha \comp (\beta \comp \gamma).
\end{equation}
for $\Hom$-valued $A$-forms $\alpha$, $\beta$, $\gamma$.
The operation $\comp$ turns  $\Omega(A; \und{\End(F)})$ into a graded associative algebra. The graded commutator on $\Omega(A; \und{\End}(F))$ is defined by $[\alpha, \beta] = \alpha \comp \beta - (-1)^{|\alpha||\beta|} \beta\comp \alpha$.
Another  case  is the wedge product associated with the evaluation map $\ev: \Hom(F, G) \otimes {\new F} \ra G$ which will be  denoted in the same way as
$\alpha \comp \eta := \alpha\wedge_{\ev}\eta$ where  $\alpha\in \Omega(A, \und{\Hom}(F, G))$,  $\eta\in \Omega(A; F)$, since the evaluation map is  a special case of the composition of maps, thanks to the isomorphism $F \simeq \Hom(\R, F)$.
\begin{df}\label{df:repr_hom} \cite{AbCr2012} A \emph{representation up to homotopy} of a Lie algebroid $A$ consists
of a $\Z$-graded vector bundle $F = \oplus_{i\in \Z} F^i$ and an operator, called \emph{the structure operator},
$$
    D : \Omega(A; F) \ra \Omega(A; F)
$$
of total degree one which satisfies $D\circ D =0$ and the \emph{graded derivation rule}
$$
    D(\omega.\eta) = \dd_A(\omega).\eta + (-1)^k \omega.D(\eta)
$$
for $\omega\in \Omega^k(A)$, $\eta\in \Omega(A; F)$. A morphism $\Phi: (F, D_F) \to (G, D_G)$ linking two representations u.t.h. is a degree zero $\Omega(A)$-module map $\Phi: \Omega(A; F) \to \Omega(A; G)$ which commutes with the structure operators $D_F$ and $D_G$.
\end{df}

There is a one-to-one correspondence  between $A$-forms $\omega \in \Omega(A; \und{\Hom}(F, G))$ of total degree $n$ and operators $\wh{\omega}: \Omega(A; F) \to \Omega(A; G)$ of degree $n$ which are $\Omega(A)$-linear in the graded sense. The operator $\wh{\omega}$ is given by $\wh{\omega}(\eta) = \omega\comp \eta$. The equation \eqref{e:comp_and_dot} implies that $\wh{\alpha \comp \beta} = \wh{\alpha} \circ \wh{\beta}$, where $\beta\in \Omega(A; \und{\Hom}(F, G))$, $\alpha\in \Omega(A; \und{\Hom}(G, H))$.

A \emph{cochain complex} $(F, \pa)$ is a $\Z$-graded vector bundle $F$ equipped with an endomorphism $\pa \in \und{\End}^1(F)$ such that $\pa \circ \pa =0$, i.e.,  a differential on $F$. Such a differential can be consider as a $0$-form with values in $\und{\End}(F)$, and gives rise to an operator $\wh{\pa}: \Omega^p(A; F) \to \Omega^p(A, F)$,  $\wh{\pa}(\eta) = \pa \comp \eta$. It satisfies  $\wh{\pa} \circ \wh{\pa} = 0$ and
 $$
 \wh{\pa}(\eta)(x_1, x_2, \ldots, x_p) = (-1)^p \pa(\eta(x_1, x_2, \ldots, x_p)),
 $$
The sign $(-1)^p$ comes from \eqref{e:wedge_product}. Given two complexes $(F, \pa^F)$, $(G, \pa^G)$ and $\omega\in \Omega^p(A; \und{\Hom}^i(F, G))$ we get the induced $p$-form $\pa^{\Hom} \omega$ defined as
$$
\pa^{\Hom} \omega = \pa^G \comp \omega - (-1)^{|\omega|} \omega \comp \pa^F
$$
which takes  values in $\und{\Hom}^{i+1}(F, G)$, where $|\omega| = p+i$ is the total degree of $\omega$. We have  the complex $(\und{\Hom}(F, G), \pa^{\Hom})$ with the differential $\pa^{\Hom}$ obtained by specializing to the case  $p=0$ which reads as
$$
    (\pa^{\Hom} T)(v) = \pa^G(T(v)) - (-1)^{|T|} T(\pa^F(v)),
$$
where $v\in F$, $T\in \und{\Hom}(F, G)$, and $|T|$ stands for the degree of $T$.
It follows immediately from \eqref{e:comp_and_dot} that
$$
    \wh{\pa^{\Hom} \omega} = \wh{\pa^G} \circ \wh{\omega} - (-1)^{|\omega|} \wh{\omega} \circ \wh{\pa^F} : \Omega(A; F) \ra \Omega(A; G).
$$
In the case $F=G$ we can write $\wh{\pa^{\Hom} \omega} = [\wh{\pa^F}, \wh{\omega}]$.

{ Besides, an  $A$-connection $\nabla$ on a vector bundle $F$
induces an operator $\dd_\nabla$ on $\Omega(A; F)$ of degree one defined by means of Koszul formula. This formula can be derived from
the conditions: \begin{enumerate}[(i)]
\item $(\dd_\nabla \eta) (s) = \nabla_s \eta$ for  $0$-forms $\eta \in \Omega^0(A; F) =  \Sec(F)$;
\item the graded derivation rule: $\dd_\nabla(\omega.\eta) = \dd_A(\omega).\eta + (-1)^{|\omega|}  \omega.(\dd_\nabla \eta)$. \end{enumerate}
Given $A$-connections $\nabla^F$, $\nabla^G$ on the graded vector bundles $F, G$, respectively, we get an $A$-connection
on $\und{\Hom}(F, G)$ given by
\begin{equation}\label{eqn:conn_Hom}
    (\nabla^{\Hom} T)(v) = \nabla^G(T(v))  -  T (\nabla^F v) 
\end{equation}
where $v\in \Sec(F)$, and  $T$ is  a section of $\und{\Hom}(F, G)$. The corresponding operator $\dd_\nabla$ on $\Omega(A; \und{\Hom}(F, G))$ is given by \commentMR{weight of $T$ is not important?}
$$
\wh{\dd_\nabla \omega} = \dd_{\nabla^G} \circ \wh{\omega} {\new -} (-1)^{|\omega|} \wh{\omega} \circ \dd_{\nabla^F}, \quad \omega\in \Omega(A;  \Hom(F, G)).$$
To prove it one shows that $\dd_\nabla$
satisfies  the graded derivation rule and  that it reduces to the formula \eqref{eqn:conn_Hom} when $T  = \omega$ is a $0$-form.
}

The structure operator $D$  of a representation u.t.h. can be decomposed into a sequence of $\und{\End}(F)$-valued $A$-forms and an $A$-connection giving an equivalent description.  A precise statement is the following:

\begin{prop}\cite[Proposition 3.2 and Definition 3.3]{AbCr2012} The structure operator $D$ on a $\Z$-graded vector bundle $F$ can be equivalently given by a series of maps:
    \begin{itemize}
        \item A degree $1$ operator $\pa$ on $F$ making $(F, \pa)$ a complex, i.e.,  $\pa\circ \pa= 0$.
        \item An $A$-connection $\nabla^F$ on $(F, \pa)$, i.e.,  $\pa (\nabla^F v) = \nabla^F (\pa v)$ for $v\in \Sec(F)$.
        \item A 2-form $\omega_2\in \Omega^2(A; \und{\End}^{-1}(F))$ such that $\pa^{\Hom} \omega_2 + \curv_{\nabla^F} = 0$.\commentMR{Zgadza się.}
        \item A sequence $(\omega_2, \omega_3, \ldots )$ of $\und{\End}(F)$-valued $A$-forms,  $\omega_i\in \Omega^i(A; \und{\End}^{1-i}(F))$\footnote{Note that $\omega_i$ {\new has total } degree 1}
            such that for each $n\geq 3$
            \begin{equation}\label{e:repr_str_eq}
                 0 = \pa^{\Hom}  \omega_n + \dd_\nabla \omega_{n-1} + \sum_{i=2}^{n-2} \omega_i \comp \omega_{n-i} \in \Omega^n(A; \und{\End}^{2-n}(F)).
            \end{equation}
    \end{itemize}
    A morphism $\Phi$ from $(F, D_F)$ to $(G, D_G)$ is given by a sequence of $A$-forms $\Phi_i \in \Omega^i(A; \und{\Hom}^{-i}(F, G))$ such that $\Phi_0$ is a map of complexes and for each $n\geq 1$
    {
    \begin{equation}\label{e:repr_str_Phi}
     0 = \pa^{\Hom} \Phi_n + \dd_\nabla \Phi_{n-1} + \sum_{i=2}^{n-2} (\omega_i^G \comp \Phi_{n-i} - \Phi_{n-i} \comp \omega_i^F) \ \in \Omega^n(A; \und{\Hom}^{1-n}(F, G)).
    \end{equation}}
\end{prop}
\begin{rem} Note that  $\wh{\pa^{\Hom} \omega_i} = \wh{\pa^F} \circ \wh{\omega_i} + \wh{\omega_i}  \circ  \wh{\pa^F}$ and $\wh{\dd_\nabla \omega_i}  = \dd_{\nabla^F} \circ \wh{\omega_i} + \wh{\omega_i} \circ \dd_{\nabla^F}$ (as $|\omega_i|=1$) and the structure equation \eqref{e:repr_str_eq} is equivalent to $D\circ D = 0$ where
\begin{equation} \label{df:str_operator}
    D = \wh{\pa^F} + \dd_{\nabla^F} + \wh{\omega_2} + \ldots:\Omega(A; F) \ra \Omega(A; F).
 \end{equation}
 Similarly,  as $|\Phi| = 0$, we have $\wh{\pa^{\Hom} \Phi_n} = \wh{\pa^G} \circ \wh{\Phi}_n - \wh{\Phi}_n \circ \wh{\pa^F}$, $\wh{\dd_{\nabla} \Phi_{n-1}} = \dd_{\nabla^G} \circ \wh{\Phi}_{n-1} -  \wh{\Phi}_{n-1} \circ \dd_{\nabla^F}$ and the equation \eqref{e:repr_str_Phi} means that the operator $\wh{\Phi}= \wh{\Phi_0} + \wh{\Phi_1} + \ldots$ intertwines the structure operators $D_F$ and $D_G$.
\end{rem}


\small
\bibliographystyle{alpha}
\bibliography{references}

\end{document}